\documentclass[11pt,reqno]{amsart}

\usepackage[utf8]{inputenc} 

\usepackage[italian,english]{babel}

\usepackage[margin=1in]{geometry} 

\usepackage{graphicx} 
\usepackage{float} 

\usepackage[parfill]{parskip}

\usepackage{booktabs} 
\usepackage{array}
\usepackage{paralist} 
\usepackage{verbatim} 
\usepackage{subfig} 
\usepackage{mathrsfs}
\usepackage{amssymb}
\usepackage{amsmath}
\usepackage{amsfonts,amssymb,esint}
\usepackage{amsthm}
\usepackage{graphics,color}
\usepackage{enumerate, enumitem}
\usepackage{mathtools,centernot}
\usepackage{cases}
\usepackage{amsrefs}
\usepackage{bbm}
\usepackage{xfrac}
\usepackage{hyperref}
\usepackage[noabbrev, capitalize]{cleveref}
\usepackage{autonum}

\RequirePackage{etex}

\usepackage{bookmark}

\newtheorem{theorem}{Theorem}[section]
\newtheorem{lemma}[theorem]{Lemma}

\newtheorem{corollary}[theorem]{Corollary}
\newtheorem{definition}[theorem]{Definition}
\newtheorem{proposition}[theorem]{Proposition}
\newtheorem{remark}[theorem]{Remark}

\numberwithin{equation}{section} 

\newcommand{\abs}[1]{\left|#1\right|}

\newcommand{\norm}[1]{\left\|#1\right\|}

\newcommand{\T}{\ensuremath{\mathcal{T}}}
\newcommand*{\R}{\ensuremath{\mathbb{R}}}

\newcommand*{\N}{\ensuremath{\mathbb{N}}}

\newcommand{\eps}{\varepsilon}

\newcommand*{\G}{\ensuremath{\mathcal{G}}}

\newcommand{\e}{\varepsilon}

\renewcommand{\MR}[1]{} 

\usepackage{color, graphicx}
\usepackage{mathrsfs, dsfont}

\usepackage[]{hyperref}
\hypersetup{
    colorlinks=true,       
    linkcolor=red,          
    citecolor=blue,        
    filecolor=red,      
    urlcolor=cyan           
}

\def\dist{\mathop{\rm dist}\nolimits}    
 
\def\sgn{\mathop{\rm sgn\,}\nolimits}    

\newcommand{\be}{\begin{equation}}
\newcommand{\ee}{\end{equation}}

\title{$\Gamma$-limsup estimate for a nonlocal approximation of the Willmore functional}

\author{Hardy Chan}
\address{Departement Mathematik Und Informatik, Universit\"at Basel, CH-4051 Basel, Switzerland}
\email{hardy.chan@unibas.ch}

\author{Mattia Freguglia}
\address{Scuola Normale Superiore, Piazza dei Cavalieri 7, 56126 Pisa, Italy}
\email{mattia.freguglia@sns.it}

\author{Marco Inversi}
\address{Departement Mathematik Und Informatik, Universit\"at Basel, CH-4051 Basel, Switzerland}
\email{marco.inversi@unibas.ch}

\date{}

\begin{document}

\begin{abstract}
We propose a possible nonlocal approximation of the Willmore functional, in the sense of Gamma-convergence, based on the first variation of the fractional Allen--Cahn energies, and we prove the corresponding $\Gamma$-limsup estimate. Our analysis is based on the expansion of the fractional Laplacian in Fermi coordinates and fine estimates on the decay of higher order derivatives of the one-dimensional nonlocal optimal profile.
This result is the nonlocal counterpart of that obtained by Bellettini and Paolini, where they proposed a phase-field approximation of the Willmore functional based on the first variation of the (local) Allen--Cahn energies. 

\vspace{2.5ex}



\vspace{1ex}


\end{abstract}

\maketitle

\section{Introduction} \label{s: introduction}

In recent years, there has been a growing interest in geometric energies, such as the Area functional or the Willmore functional, the latter being defined by
    \begin{equation}
        \label{eq:def-Will}
        \mathcal{W}(\Sigma, \Omega) = \int_{\Sigma \cap \Omega} H_{\Sigma}^2(y) \, d \mathcal{H}^{d-1}(y),
    \end{equation}
where $d \ge 2$ is an integer number, $\Omega \subset \R^d$ is an open set, $\Sigma \subset \R^d$ is a smooth hypersurface, $H_{\Sigma}(y)$ is the mean curvature of $\Sigma$ at the point $y$ and $\mathcal{H}^{d-1}$ stands for the $(d-1)$-dimensional Hausdorff measure in $\R^d$. Among several questions related to this functional, a relevant problem consists in minimizing $\mathcal{W}$ within all the surfaces of a certain type, for instance with fixed genus~\cites{S93, K96, BK03}; or connected, closed, confined to a prescribed region and with fixed area~\cites{DMR14, DLW17}; or with prescribed isoperimetric ratio~\cite{S12}. Another problem that generates a lot of interest is the study of the geometric flow associated to $\mathcal{W}$, see among others~\cites{KS01, KS02, KS12, S01}. 

While both of these problems can be attacked directly, it could be useful to approximate $\mathcal{W}$ with \emph{simpler} functionals, solve similar problems for the approximating functionals and then try to pass to the limit to obtain a solution to the original problems. Here, the notion of approximation considered is that of Gamma-convergence, see~\cite{DGF75} for the definition and the fundamental properties, which is well suited for the convergence of the minima and the minimizers, and, under additional assumptions, also for the convergence of the flows, see~\cites{SS04, Serf11}.

For the Area functional, this approach has been highly successful when considering the phase-transition regularization given by the Allen--Cahn energies, which are defined for any $\eps > 0$ by
    \begin{equation}
        \label{eq:AC-local}
        \mathcal{E}_{\eps}(u, \Omega) = \int_{\Omega} \bigg( \frac{\eps}{2} \abs{\nabla u}^2 + \frac{W(u)}{\eps} \bigg) \, dx, 
    \end{equation}
where $W \colon \R \to [0,+\infty)$ is a double-well potential, e.g. $W(s)=(1-s^2)^2$, which has wells at $s=\pm 1$ as we assume for convenience. From now on, in any statement concerning $W$, we will implicitly assume that it satisfies few structural assumptions, see \cref{S:Notation} for the details.

In the foundational work~\cite{MM77} (see also~\cite{Modica}) Modica and Mortola proved that for any set $E \subset \R^d$ of finite perimeter $\mathrm{Per}(E,\Omega)$ inside a Lipschitz domain $\Omega$ it holds
    \begin{equation}
        \label{eq:MM}
        \Gamma(L^1(\Omega))-\lim_{\eps \to 0^+} \mathcal{E}_{\eps}(\chi_E, \Omega) = \sigma \mathrm{Per}(E,\Omega),
    \end{equation}
where $\chi_E := \mathds{1}_E - \mathds{1}_{E^c}$ denotes the difference between the characteristic functions of $E$ and $E^c$, while $\sigma$ is a positive constant depending on the potential $W$. For further details about the notation, see~\cref{S:Notation}. 

The convergence of (constrained) critical points of $\mathcal{E}_\eps$ to ``generalized'' critical points of the Area functional (also when the ambient space is a closed manifold) has been studied by many authors, see for instance~\cites{Modica, Stern88, LM89, HT00, T05, TW12} and the more recent results in~\cites{G18, GG18, CM20, CM21, M21, M24, Bel22, Bel24}. Moreover, the convergence of the rescaled $L^2$-gradient flows of $\mathcal{E}_\eps$ to Brakke's motion by mean curvature has been established in~\cite{Ilm93}.

Now, it is well-known that the mean curvature represents the first variation of the Perimeter. Therefore, it is reasonable to consider a suitable $L^2$-norm of the first variation of $\mathcal{E}_{\eps}$ to build an approximation for $\mathcal{W}$. Indeed, starting from a conjecture of De Giorgi stated in~\cite{DG91} (see also~\cite{BFP23} for more about this conjecture), Bellettini and Paolini proposed to consider in~\cite{Bel-Pao93} the following approximation of the Willmore functional:
    \begin{equation}
        \label{eq:Bel-Pao-mod}
        \mathcal{W}_{\eps}(u,\Omega) = \frac{1}{\eps} \int_{\Omega} \bigg( \eps \Delta u - \frac{W'(u)}{\eps} \bigg)^2 \, dx.
    \end{equation}
In addition, for any open set $E \subset \R^d$ with $\partial E \in C^2$, they proved the $\Gamma$-$\limsup$ estimate: \vspace{0.5ex}
    \begin{equation}
        \label{eq:Bel-Pao-main1}
        \Gamma(L^1(\R^d))-\limsup_{\eps \to 0^+} \big( \mathcal{E}_{\eps} + \mathcal{W}_{\eps} \big)(\chi_E, \R^d)
        \le
        \sigma \mathrm{Per}(E,\R^d)
        +
        \sigma \mathcal{W}(\partial E, \R^d).
    \end{equation}
To this end, they exhibited a family of functions converging to $\chi_E$ in $L^1$ for which the limit of the values of $\mathcal{E}_{\eps} + \mathcal{W}_{\eps}$ evaluated on this family equals the right-hand side of~\eqref{eq:Bel-Pao-main1}. They considered a slight modification of the rescaled transition profiles
    \begin{equation}
        \label{eq:std-rec}
        u_\eps(x):=q_0\bigg(\frac{\mathrm{dist}_{\partial E}(x)}{\eps}\bigg),
    \end{equation}
where $q_0$ is the one-dimensional optimal profile (i.e. the unique, up to translations, monotone increasing solution of the Euler--Lagrange equation of~\eqref{eq:AC-local} for $d=1, \Omega=\R$ and $\eps=1$) and $\mathrm{dist}_{\partial E}$ is the signed distance function from the boundary of $E$.
The $\Gamma$-$\liminf$ estimate completing~\eqref{eq:Bel-Pao-main1} turns out to be more delicate, and it was proved by Röger and Schätzle in~\cite{Roger-Schatzle} when $d \in \{ 2, 3\}$ and by Nagase and Tonegawa in~\cite{Nagase-Tonegawa} for $d=2$ with different techniques, while being still open in higher dimension.  

The aim of the present paper is to prove a nonlocal counterpart of the estimate~\eqref{eq:Bel-Pao-main1} exploiting the approximation of the (local) Perimeter given by the fractional Allen--Cahn energies. Given $s \in \left(\sfrac{1}{2},1\right)$ and $\eps > 0$, the $s$-fractional Allen--Cahn energy is defined by
    \begin{equation}
        \label{eq:AC-fract}
        \mathcal{F}_{s,\eps}(u, \Omega)
        = \e^{2s-1}
        \frac{\gamma_{d,s}}{4} \iint_{\R^d \times \R^d \setminus (\Omega^c \times \Omega^c)}
        \frac{\abs{u(x)-u(y)}^2}{\abs{x-y}^{d+2s}} \, dx \, dy + \int_{\Omega} \frac{W(u(x))}{\eps} \, dx,
\end{equation}
where $\gamma_{d,s}$ is given by \eqref{eq: constant fractional laplacian}. The starting point of our analysis is the work by Savin and Valdinoci~\cite{SV12} where, among other results, they established a nonlocal version of~\eqref{eq:MM}. Specifically, they proved that for any set $E \subset \R^d$ of finite perimeter $\mathrm{Per}(E,\Omega)$ inside a Lipschitz domain $\Omega$ it holds \vspace{0.5ex}
    \begin{equation}
        \label{eq:SV}
        \Gamma(L^1_{\mathrm{loc}}(\R^d))-\lim_{\eps \to 0^+} \mathcal{F}_{s,\eps}(\chi_E, \Omega) = c_{\star} \mathrm{Per}(E,\Omega),
    \end{equation}
where $c_{\star}$ is a positive constant depending on $s,d, W$ and the functionals $\mathcal{F}_{s,\eps}$ are thought to be defined on those functions in $L^{\infty}(\R^d)$ that take values between the zeros of $W$ (see~\cite{CP16} for similar results on more general functionals of nonlocal type). This result holds also for $s=\sfrac{1}{2}$, but with a different scaling in $\eps$, that is with an extra $\abs{\log(\eps)}^{-1}$ factor in front of $\mathcal{F}_{s,\eps}$. We mention also~\cite{ABS94} for related results in the case $d=1$ and $s= \sfrac{1}{2}$. Similarly to the local case, the estimate from above in the previous theorem is achieved by considering functions of the form of~\eqref{eq:std-rec}, at least when the set $E$ meets the boundary of $\Omega$ transversally, and where $q_0$ is replaced by the nonlocal one-dimensional optimal profile $w$. In particular, $w \colon \R \to (-1,1)$ is defined as the unique, up to translations, monotone increasing solution of the fractional Allen--Cahn equation
    \begin{equation}
        \label{eq:f-one-dim-profile}
        (-\partial_{zz})^s w + W'(w) = 0,
    \end{equation}
where $(-\partial_{zz})^s$ is the fractional Laplacian of order $2s$ in dimension one. We refer to \cref{ss: optimal profile } for a collection of some well-known properties of $w$. More in general, the first variation of $\mathcal{F}_{s,\eps}$ is represented by $\eps^{2s-1}(-\Delta)^s u + \e^{-1} W'(u)$, where $(-\Delta)^s$ denotes the $s$-fractional Laplacian (see \cref{ss: fractional laplacian} for precise definitions). At this point, it is natural to consider the following nonlocal version of the functional defined by~\eqref{eq:Bel-Pao-mod}, that is 
    \begin{equation}
        \label{eq:frac-Bel-Pao}
        \G_{s,\e}(u, \Omega) = \frac{1}{\e}
        \int_{\Omega} \bigg(\e^{2s-1} (-\Delta)^s u + \frac{W'(u)}{\e} \bigg)^2 \, dx,
    \end{equation} 
if $u \in L^{\infty}(\R^d) \cap C^2(\Omega)$, and $\G_{s,\e}(u, \Omega) := +\infty$ otherwise in $L^1_{\mathrm{loc}}(\R^d)$. Here, the nonlocal behaviour of the functional is encoded in the fractional Laplacian. We are ready to state our main result.

\begin{theorem}
        \label{thm:main}
        Let $s \in \left(\sfrac{3}{4},1\right)$. Then, for any bounded open set $E \subset \R^d$ with $\partial E \in C^2$ it holds \vspace{1ex}
            \begin{equation}
                \label{eq:main}
                \Gamma(L^1_{\mathrm{loc}}(\R^d))-\limsup_{\e\to 0^+}
                \big(\mathcal{F}_{s,\e}
                +
                \G_{s,\e}\big)(\chi_E,\Omega)
                \le
                c_{\star} \mathrm{Per}(E, \Omega)
                +
                \kappa_{\star}
                \mathcal{W}(\partial E, \Omega),
        \end{equation}
        where $\Omega \subset \R^d$ is any bounded open set with $\partial \Omega \in C^1$, $c_{\star}$ is the constant in~\eqref{eq:SV} and $\kappa_{\star}$ is a positive constant depending only on $s$ and~$W$. In the case $s=\sfrac{3}{4}$ the same conclusion holds if in the definition of $\mathcal{G}_{s,\eps}$ we add an extra $\abs{\log(\eps)}^{-1}$ factor in front of the integral.
\end{theorem}

We recall the definition of the $\Gamma$-$\limsup$, more precisely
\begin{equation}
    \Gamma(L^1_{\mathrm{loc}}(\R^d))-\limsup_{\e\to 0^+} \big(\mathcal{F}_{s,\e} + \G_{s,\e}\big)(\chi_E,\Omega) = \inf \left\{ \limsup_{\e \to 0^+} (\mathcal{F}_{s,\e}+\mathcal{G}_{s,\e})(u_\eps,\Omega) \colon u_\eps \to \chi_E \text{ in }L^1_{\textrm{loc}}(\R^d) \right\}.
\end{equation}

The constant $\kappa_\star$ is defined by \eqref{eq: the exact constant}. We mention that without the regularity assumption on $\Omega$ an analogous of \eqref{eq:main} still holds, that is with $\overline{ \Omega}$ on the right-hand side (see \cref{r: limsup inequality with closure}). Before discussing the strategy of the proof, some comments are in order.

a) We point out that we consider the sum of $\mathcal{F}_{s,\e}$ and $\G_{s,\e}$ in order to rule out the behaviour of certain trivial sequences that are uniformly bounded with respect to the functionals $\{ \G_{s,\e} \}$ but unbounded for the functionals $\{ \mathcal{F}_{s,\e} \}$. In particular, the ultimate goal is to show that there is a relationship between the limit of $\G_{s,\e}(u_\eps)$ and the limit of $\mathcal{F}_{s,\e}(u_\eps)$ when both are uniformly bounded in $\eps > 0$. Since $\mathcal{F}_{s,\e}$ and $\G_{s,\e}$ are non-negative functionals, it is reasonable to consider their sum. An example of a trivial sequence that we want to exclude from our analysis is the following. We know from Rolle's Theorem that there exists $c_0 \in (-1,1)$ such that $W'(c_0)=0$. For any $\eps > 0$, we consider $u_\eps \equiv c_0$, then it holds that $\G_{s,\e}(u_\eps) = 0$. On the other hand, since $W(c_0) > 0$, we have
    \[
        \lim_{\eps \to 0^+} \mathcal{F}_{s,\e}(u_\eps) = \lim_{\eps \to 0^+} \frac{W(c_0)}{\eps} \mathcal{L}^d(\Omega) = +\infty.
    \]

b) The need of a different scaling for $s=\sfrac{3}{4}$ mildly suggests that the previous result might no longer be true when $s \in \left[\sfrac{1}{2},\sfrac{3}{4}\right)$. Maybe, for these values of the parameter $s$ the limit of the functionals $\mathcal{F}_{s,\eps}+\mathcal{G}_{s,\eps}$ (with a different scaling in $\eps$ for $\mathcal{G}_{s,\eps}$) could be equal to the sum of the (local) Perimeter and a combination of the Willmore functional and a nonlocal quantity possibly depending on the nonlocal mean curvature (for the definition see, for example,~\cite{AbVal14} or~\cite{FFMMM15}*{Section~6}). We refer to~\cref{s:proof of main theorem} for further comments.

c) On the other hand, when $s \in \left(0,\sfrac{1}{2}\right)$ the situation is fairly different. Already for the functionals $\mathcal{F}_{s,\eps}$ it is not longer true that they approximate the classical (local) Perimeter. In \cite{SV12}, Savin and Valdinoci showed that the rescaled functionals $\eps^{1-2s} \mathcal{F}_{s,\eps}$ approximate the $2s$-fractional Perimeter and that any family of minimizers with equibounded energy converges to a minimizer of the $2s$-fractional Perimeter in $\Omega$. In~\cite{MSW19}, the previous convergence was extended to more general families $\{ u_{\eps} \}$ of critical points of $\eps^{1-2s} \mathcal{F}_{s,\eps}$ having equibounded energy. In particular, the authors proved that $u_{\eps}$ converges in a quite strong sense to the characteristic function of a set $E_{*}$ such that $\partial E_{*} \cap \Omega$ is a stationary $2s$-fractional minimal surface in $\Omega$. Roughly speaking, $\partial E_{*}$ has vanishing $2s$-fractional mean curvature on $\Omega$. Moreover, they obtained compactness results also for family of functions $\{ u_\eps \}$ with equibounded energy and a uniform Sobolev bound on their first variations. These results are the nonlocal analogies, in the range $s \in \left(0,\sfrac{1}{2}\right)$, of those in \cites{HT00, Ton05} in the local case.

d) In the case $s \in \left[\sfrac{3}{4},1\right)$, it would be interesting to know if a full Gamma-convergence result holds, namely, if the functionals $\mathcal{F}_{s,\eps} + \mathcal{G}_{s,\eps}$ Gamma-converge to the right-hand side of~\eqref{eq:main}, at least in small dimensions and within the class of $C^2$ sets, where the corresponding result in the local case is known to be true~\cites{Roger-Schatzle, Nagase-Tonegawa}. More in general, understanding the compactness properties of families of functions having $\mathcal{F}_{s,\eps} + \mathcal{G}_{s,\eps}$ equibounded would be of interest. However, up to our knowledge the situation is not clear even for families $\{ u_\eps \}$ of critical points of $\mathcal{F}_{s,\eps}$ having equibounded energy. It is reasonable to expect a similar behaviour to the local case, where the energies of the critical points concentrate towards a ``generalized'' critical point of the Perimeter, we refer again to~\cite{HT00}.

In order to prove \cref{thm:main} we need to exhibit, for any set $E \subset \R^d$ with $\partial E$ of class $C^2$, a family of functions $\{ u_{\eps} \} \subset L^{\infty}(\R^d) \cap C^2(\Omega)$ converging to $\chi_E$ in $L^1_{\mathrm{loc}}(\R^d)$ such that \vspace{0.5ex}
    \begin{equation}
        \label{eq:how-proof-main}
        \lim_{\eps \to 0^+} \mathcal{F}_{s,\eps}(u_{\eps},\Omega) = c_{\star} \mathrm{Per}(E,\Omega)
        \quad \text{and} \quad 
        \lim_{\eps \to 0^+} \mathcal{G}_{s,\eps}(u_{\eps},\Omega) = \kappa_{\star} \mathcal{W}(\partial E, \Omega).
    \end{equation}
If $E$ is a smooth set with $\overline{E} \subset \Omega$, then a natural candidate for $u_{\eps}(x)$ would be $w_\e(\dist_{\partial E}(x))$, where $w_\e(z) = w\left( \sfrac{z}{\e}\right)$ and $w$ is the nonlocal one-dimensional optimal profile (see \cref{ss: optimal profile }). Indeed, as mentioned earlier, this family of functions is a recovery sequence for the $\Gamma$-limit of $\mathcal{F}_{s,\eps}$, meaning that the first identity in~\eqref{eq:how-proof-main} holds true when considering this family of functions. On the other hand, while the signed distance function $\mathrm{dist}_{\partial E}$ is globally $1$-Lipschitz, it is smooth only near the boundary of $E$, and therefore $\mathcal{G}_{s,\eps}$ has infinite value at $w_\e( \dist_{\partial E}(x))$. To overcome this issue, we introduce a suitable extension $\beta_{\partial E} \in C^2(\R^d)$ of $\mathrm{dist}_{\partial E}$ outside a fixed small tubular neighbourhood of $\partial E$ (see \cref{ss: distance function}). Accordingly, we consider the family of functions defined by $u_{\eps}(x):=w_\e(\beta_{\partial E}(x))$. We point out here that the first identity in~\eqref{eq:how-proof-main} remains true also when considering this family of functions, this follows by the same argument used in~\cite{SV12}. After doing so, our strategy to prove the second identity in~\eqref{eq:how-proof-main} is similar to the one adopted in~\cite{Bel-Pao93}. We split $\mathcal{G}_{s,\eps}(u_\eps, \Omega)$ into the sum of two contributions, the first one being the integral in a small tubular neighbourhood of $\partial E$ and the second one being the contribution coming from the complement of this tubular neighbourhood. However, the analysis of these two quantities is technically more involved than in the local case. We also remark that it is also quite different from the approach used in~\cite{SV12} to obtain the $\Gamma$-limsup estimate for the functionals $\{ F_{s,\eps} \}$ alone. Indeed, in our case we need a precise description of the fractional Laplacian of the function $u_\eps$ near the interface $\partial E$, as well as an accurate asymptotic of the higher-order derivatives of the nonlocal one-dimensional optimal profile $w$. Both of these aspects were not necessary for the corresponding estimate in~\cite{SV12}.
As mentioned shortly before, in order to estimate the contribution of $\mathcal{G}_{s,\eps}(u_\eps, \Omega)$ away from the interface, we require fine estimates on the decay of the higher order derivatives of~$w$. While such estimates could be known by the experts, we were not able to find them explicitly stated in the literature. For the local one-dimensional optimal profile $q_0$ they are well-known (see for instance~\cite{BAHN15}). In particular, borrowing ideas from~\cites{AT91, KMR11} we deduce the following result.

\begin{theorem} \label{thm:hi-decay}
     Let $s \in (0,1)$ and $k \ge 2$. In addition to our structural assumptions, suppose that $W \in C^{k+1}(\R)$. Letting $w$ be the nonlocal one-dimensional optimal profile, then $w\in C^{k} (\R)$ and there exists a positive constant $C = C(s,W,k)$ such that
        \begin{equation} \label{eq: decay of w''}
            \abs{\partial_x^k w(x)} \leq \frac{C}{1+ \abs{x}^{k+2s}}, \qquad \forall x \in \R.
        \end{equation} 
\end{theorem}

The estimate in the theorem above is coherent with the fact that $w'$ is asymptotic to $\abs{x}^{-1-2s}$, as proved in \cite{CSbis14}*{Theorem 2.7}. See also \cref{r: assumption thm decay} for further discussions. 

To deal with the contribution coming from a small tubular neighbourhood $\mathcal{U}_{\delta}$ of $\partial E$ of size $\delta$, it is convenient to use Fermi coordinates (see~\cref{ss: uniform fermi coordinates} for the precise definition). More precisely, we identify any $x \in \mathcal{U}_{\delta}$ with a couple $(y,z)$, where $y \in \partial E$ is the point of minimum distance between $x$ and $\partial E$ and $z=\mathrm{dist}_{\partial E}(x)$. Using these coordinates the (local) Laplacian can be written as
    \begin{equation}
        \label{eq:lL-Fermi}
        \Delta = \partial_{zz} - H_z \partial_z + \Delta_{\Sigma_z},
    \end{equation}
where $H_z$ denotes the mean curvature (with respect to its outer normal) of the hypersurface $\Sigma_z := \partial\{ x \in \mathcal{U}_{\delta} \colon \mathrm{dist}_{\partial E}(x) > z \}$, which is diffeomorphic and parallel to $\partial E$ and $\Delta_{\Sigma_z}$ is the Laplace--Beltrami operator on $\Sigma_z$. One advantage of these coordinates is that whenever $u$ is a function that depends only on the distance from the boundary of $E$, then the computation of its Laplacian considerably simplifies. In contrast, a similar expression to~\eqref{eq:lL-Fermi} for the fractional Laplacian has been derived only recently. Indeed, the first-named author studied Fermi coordinates in the context of the fractional Laplacian in his PhD Thesis and he introduced suitable expansions for the fractional Laplacian in~\cites{ChanWei17, CLW17}
, in collaboration with Liu and Wei. Roughly speaking, if $s \in \left(\sfrac{1}{2},1\right)$, $u_\eps$ and $w_\eps$ are defined as above, then we have the following expansion
    \begin{equation}
        \label{eq:fL-Fermi}
        (-\Delta)^s u_{\eps}(x) = (-\partial_{zz})^s w_{\eps}(z) + H_{\partial E}(y) L_s[w_{\eps}'](z) + \mathcal{R}_{\eps}(y,z),
    \end{equation}
where $L_{s}$ is a suitable operator evaluated at $w'_{\eps}$ and $\mathcal{R}_{\eps}(y,z)$ is an error term. In addition, the first term on the right-hand side in the previous expansion scales as $\eps^{-2s}$, the second one scales as $\eps^{1-2s}$, while the remainder term is uniformly bounded in $\eps$. We refer to \cref{s:expansion of fractional laplacian} for the precise statement and a detailed proof.

In conclusion, we briefly comment about the assumption $s \in \left[\sfrac{3}{4},1\right)$. If we neglect the error term in~\eqref{eq:fL-Fermi}, then from~\eqref{eq:f-one-dim-profile} and~\eqref{eq:frac-Bel-Pao} we obtain the following asymptotic
    \begin{equation}
        \label{eq:why-s-range}
        \G_{s, \eps}(u_\eps, \mathcal{U}_{\delta})
        \simeq
        \eps^{4s-3} \norm{L_{s}[w_{\eps}']}^2_{L^2((-\delta, \delta))}
        \mathcal{W}(\partial E, \mathcal{U}_{\delta})
        =
        \norm{L_{s}[w']}^2_{L^2\left( \left(-\sfrac{\delta}{\eps}, \sfrac{\delta}{\eps}\right) \right)}
        \mathcal{W}(\partial E, \mathcal{U}_{\delta}).
    \end{equation}
Moreover, we will see in~\cref{s: finiteness of constants} that $L_{s}[w'] \in L^2(\R)$ when $s \in (\sfrac{3}{4},1)$, while for $s=\sfrac{3}{4}$ the norm in the right-hand side of~\eqref{eq:why-s-range} diverges logarithmically in $\eps$. We refer to \cref{s:proof of main theorem} for a more detailed description of the heuristic argument as well as the proof of \cref{thm:main}.

The paper is organised as follows. In \cref{s:tools} we discuss some geometric lemmas and we deduce some useful decay properties of the nonlocal one-dimensional optimal profile $w$. In \cref{s: decay optimal profile} we discuss the proof of \cref{thm:hi-decay}. \cref{s:expansion of fractional laplacian} is entirely devoted to the expansion of the fractional Laplacian in Fermi coordinates. In this section, we follow the presentation of~\cite{CLW17}, adapting some of their results to our framework. In~\cref{s: finiteness of constants} we show that the constant $\kappa_{\star}$ appearing in~\eqref{eq:main} is finite. In \cref{s:proof of main theorem} we combine our results to conclude the proof of Theorem~\ref{thm:main}.

\section{Tools} \label{s:tools}

In this section we collect some tools which will be needed. We start by introducing some notation that we keep throughout the manuscript. 

\subsection{Notation}\label{S:Notation}

\begin{itemize}
    \item We denote by $\Gamma((\mathbb{X},d)) - \lim_{\e} F_\e$, $\Gamma((\mathbb{X},d)) - \limsup_{\e} F_\e$ and $ \Gamma((\mathbb{X},d)) - \liminf_{\e} F_\e$ the $\Gamma$-limit, the $\Gamma$-limsup, and the $\Gamma$-liminf, as $\eps \to 0$, of a family of  functions $F_\e : \mathbb{X} \to \R \cup \{ +\infty \}$ with respect to the metric $d$ on $\mathbb{X}$. For the definitions see for instance \cites{DGF75}.    
    \item We denote by $W$ a double-well potential satisfying the following structural assumptions: 
        \begin{enumerate}[label=($W1$),ref=$W1$]
        \item\label{h: zero of potential}
        $W: \R \to [0, +\infty)$, $W$ is even and $ \{W(x) = 0 \} = \{\pm 1\}$; 
        \end{enumerate}
        \begin{enumerate}[label=($W2$),ref=$W2$]
        \item\label{h: potential is smooth}
        $W \in C^{3}(\R)$; 
        \end{enumerate}
        \begin{enumerate}[label=($W3$),ref=$W3$]
        \item \label{h: W'' positive} 
        $W''(\pm 1) = \lambda > 0$; 
        \end{enumerate}   
    \item we denote by $\mathcal{W}$ the Willmore functional, $\text{Per}$ the Perimeter functional, $\mathcal{F}_{s,\e}$ the scaled nonlocal Allen--Cahn energy and $\mathcal{G}_{s,\e}$ the squared $L^2$-norm of its first variation (\cref{s: introduction}); 
    \item unless otherwise specified, we denote by $w$ the one-dimensional optimal profile (\cref{ss: optimal profile }) and $w_\e(z) = w(\sfrac{z}{\e})$ and $u_\e(x) = w_\e(\beta_{\Sigma}(x))$ (\cref{d:regular distance}); 
    \item we denote by $[\cdot]_{C^{k,\theta}}, \norm{\cdot}_{C^{k,\theta}}$ the H\"older seminorms and norms (\cref{ss: function spaces}); 
    \item we denote by $\mathscr{F}_d$ the Fourier transform in $\R^d$ (\cref{ss: fractional laplacian}); 
    \item we denote by $(-\Delta)^s$ the fractional Laplacian operator of power $s \in (0,1)$ (see \eqref{eq: fractional laplacian}); 
    \item $\gamma_{d,s}$ is the constant given by \eqref{eq: constant fractional laplacian}; 
    \item we denote by $P^{(s)}_d$ the fractional heat kernel of power $s$ in $\R^d$ (\cref{ss: fractional heat kernel}); 
    \item we denote by $G_{s,\lambda}$ the fundamental solution of $(-\Delta)^s + \lambda \mathrm{Id}$ (\cref{p: decay of fundamental solution}); 
    \item we denote by $B_\delta^m = \{ x \in \R^m \colon \abs{x} < \delta \}$ and $\mathcal{H}^m$ the $m$-dimensional Hausdorff measure; 
    \item we denote by $\Omega \subset \R^d$ a bounded open set; 
    \item we denote by $E \subset \R^d$ an open set of class $C^2$ as in \cref{d: uniform principal coordinates} and we set $\Sigma = \partial E$; 
    \item we denote by $\chi_E = \mathds{1}_E - \mathds{1}_{E^c}$, where $\mathds{1}_E$ stands for the indicator function of the set $E$;
    \item we denote by $\dist_{\Sigma}$ the signed distance from $\Sigma$ and $\beta_{\Sigma}$ the smoothed signed distance (\cref{d:regular distance}); 
    \item for any $\ell>0$ we denote by $\Sigma_\ell = \dist_{\Sigma}^{-1}((-\ell,\ell))$;   
    \item we denote by $\pi_{\Sigma}: \Sigma_\ell \to \Sigma$ the projection onto $\Sigma$ (\cref{l: regularity of distance function}); 
    \item given an open set $E \subset \R^d$, for any $x \in \partial E = \Sigma$ we denote by $T_x \Sigma, N_\Sigma(x), H_\Sigma(x)$ the tangent space to $\Sigma$ at $x$, the inner unit normal to $\Sigma$ at $x$ and the scalar mean curvature of $\Sigma$ at $x$ computed with respect to $-N_\Sigma(x)$, respectively;
    \item given an open set $E \subset \R^d$ and $ x_0 \in \partial E = \Sigma$, we denote by $Y = \mathrm{Id}\times g$ a principal parameterization around $x_0$, $N$ the inner unit normal vector given by the parameterization $g$, $k_1, \dots, k_{d-1}$ the principal coordinates at $x_0$ and $\Phi(y,z) = Y(y) + z N(y)$ the Fermi coordinated around $x_0$ (\cref{d: uniform principal coordinates}); 
    \item we denote by $\abs{(Y(y) - z_0 e_d)_\tau}$ the tangential distance (\cref{d: tangential distance}); 
    \item we denote by $\eta_{\e, \ell}$ the function in \eqref{eq: eta_e,kappa}; 
    \item we denote by $c_\star,\mu_w, \kappa_\star$ the constants in \eqref{eq:SV}, \eqref{eq: limit constant 1} and \eqref{eq: the exact constant}.
\end{itemize}

\subsection{Function spaces} \label{ss: function spaces}

Given $\Omega \subset \R^d$ open, $\theta \in (0,1]$, $k \in \N $ and $u: \Omega \to \R$, we set
\begin{equation}
    [u]_{C^\theta(\Omega)}  = \sup_{x,y \in \Omega, x \neq y} \frac{\abs{u(x)-u(y)}}{\abs{x-y}^\theta}, \quad \norm{u}_{C^\theta(\Omega)} = \norm{u}_{L^\infty(\Omega)} + [u]_{C^\theta(\Omega)},
\end{equation}
\begin{equation}
    [u]_{C^k(\Omega)} = \sum_{\abs{\alpha} = k}\norm{\partial^\alpha u}_{L^\infty(\Omega)}, \quad \norm{u}_{C^k(\Omega)} = \norm{u}_{L^\infty(\Omega)} + \sum_{j=1}^k [u]_{C^j(\Omega)},
\end{equation}
\begin{equation}
    [u]_{C^{k,\theta}(\Omega)} = \sum_{\abs{\alpha} = k} [\partial^\alpha_x u]_{C^\theta(\Omega)}, \quad \norm{u}_{C^{k,\theta}(\Omega)} = \norm{u}_{C^k(\Omega)} + [u]_{C^{k,\theta}(\Omega)}. 
\end{equation}

\subsection{The fractional Laplacian} \label{ss: fractional laplacian}

Here we recall the definition and some basic properties of the fractional Laplacian. We refer to \cites{Ga19, DPV12} for an extensive discussion on the topic. Given $s \in (0,1)$ and a function $u \colon \R^d \to \R$ globally bounded and of class $C^2$, the fractional Laplacian 
 and $(-\Delta)^s u$ is defined by 
\begin{equation}
    (-\Delta)^s u(x) = \gamma_{d,s} P.V. \int_{\R^d} \frac{u(x) -u(y)}{\abs{x-y}^{d+2s}}\, dy =  \gamma_{d,s} \lim_{\nu \to 0} \int_{B_\nu(x)^c} \frac{u(x) -u(y)}{\abs{x-y}^{d+2s}}\, dy  \label{eq: fractional laplacian}, 
\end{equation}
where the constant $\gamma_{d,s}$ is defined for convenience by 
\begin{equation}
    \gamma_{d,s} := s 2^{2s} \pi^{-\frac{d}{2}} \frac{\Gamma\left( \frac{d+2s}{2} \right)}{\Gamma(1-s)}. \label{eq: constant fractional laplacian} 
\end{equation}
It is easy to check that for globally bounded functions of class $C^2$ it holds that
\begin{equation}
    (-\Delta)^s u(x) = \frac{\gamma_{d,s}} {2} \int_{\R^d} \frac{2 u(x) - u(x+y) - u(x-y)}{\abs{y}^{d+2s}} \, dy, 
\end{equation}
We define the Fourier transform by
\begin{equation}
    \mathscr{F}_d u(\xi) = \int_{\R^d} u(x) e^{2 \pi i x \cdot \xi} \, dx,  
\end{equation}
so that the inversion formula reads as $(\mathscr{F}_d^{-1} u ) (x) = (\mathscr{F}_d u) (-x)$. With this convention, it can be checked that (see e.g.~\cite{Ga19}) 
\begin{equation}
    \mathscr{F}_d [(-\Delta)^s u] (\xi) =  \abs{2\pi \xi}^{2s} \mathscr{F}_d u(\xi), 
\end{equation}

The definition of the fractional Laplacian extends to tempered distributions by duality. We state and prove the following elementary estimate. 

\begin{lemma} \label{l:bound fractional laplacian}
Given $s \in (0,1)$ and $\Omega' \subset \joinrel \subset \Omega'' \subset \joinrel \subset \Omega \subset \R^d$ open sets, then for any $u \in L^\infty(\R^d) \cap C^{2}(\Omega)$ it holds
\begin{equation}
    \norm{(-\Delta)^s u}_{L^\infty(\Omega')} \leq C(d,s,\Omega', \Omega'') \left( [u]_{C^{2}(\Omega'')} + \norm{u}_{L^\infty(\R^d)} \right). \label{eq: L^infty bound frac lapl}
\end{equation}
\end{lemma}

\begin{proof}
Let $\delta>0$ be such that $B_{\delta}(x) \subset \Omega''$ for any $x \in \Omega'$. Then, up to a multiplicative constant $C(d,s,\Omega', \Omega'')>0$, we have 
\begin{align}
    \abs{(-\Delta)^s u(x)} & \lesssim \int_{B_{\delta}^c(x)} \frac{\abs{ u(x) -u(y)} }{\abs{x-y}^{d+2s}} \, dy 
    +\lim_{\nu \to 0} \abs{ \int_{B_{\delta}(x) \setminus B_\nu(x)} \frac{u(x) - u(y) - \nabla u(x) \cdot (y-x) }{\abs{x-y}^{d+2s}}\, dy   } 
    \\ & \lesssim \norm{u}_{L^\infty(\R^d) } \int_{B_\delta^c(x)} \frac{1}{\abs{x-y}^{d+2s}} \, dy + [u]_{C^2(\Omega'')} \int_{B_{\delta}(x)} \frac{1}{\abs{x-y}^{d+2s-2}} \, dy  
    \\ & \lesssim  \norm{u}_{L^\infty(\R^d)} + [u]_{C^2(\Omega'')}. 
\end{align}
\end{proof}

\subsection{The one-dimensional optimal profile} \label{ss: optimal profile }

Throughout the whole manuscript, we denote by $w: \R \to (-1,1)$ the unique increasing solution to the fractional Allen--Cahn equation 
\begin{equation} \label{eq: fractional AC}
\begin{cases}
    (-\partial_{zz} )^s w + W'(w) = 0, 
    \\ w(0)=0, 
    \\ \lim_{z \to \pm \infty} w(z) = \pm 1. 
\end{cases}
\end{equation}
For the sake of clarity, we collect some well-known results about the optimal profile that will be needed in our analysis. 

\begin{proposition} \label{t:optimal profile}
Fix $s \in (0,1)$ and let $W$ be a double-well potential satisfying \eqref{h: zero of potential}, \eqref{h: potential is smooth},  \eqref{h: W'' positive}. There exists a unique strictly increasing solution $w: \R \to (-1,1)$ to the problem \eqref{eq: fractional AC} Moreover, $w$ is odd, $w \in  C^1(\R)$ and there exists $C(s,W)>0$ such that  
\begin{equation}
    \abs{w'(z) } \leq \frac{C}{1+ \abs{z}^{1+2s}}, \qquad \forall z \in \R. \label{eq: decay w'}
\end{equation}
\end{proposition}

We recall that $w$ is built by minimization of the $\mathcal{F}_{s,1}$ when $\Omega = \R$. More precisely, $w$ is the unique (up to translation) minimizer of $\mathcal{F}_{s,1}(\cdot, \R)$ with respect to perturbation with compact support in $\R$, see \cite{PSV13}*{Theorem 2} for instance. Since $W$ is even, it is readily checked that $w$ is odd. The regularity of $w$ is established by iterating the apriori estimates in \cite{S07}*{Proposition 2.8-2.9}. As we already mentioned, the decay estimate \eqref{eq: decay w'} is optimal, see \cite{CSbis14}*{Theorem 2.7}. Integrating \eqref{eq: decay w'}, we find that 
\begin{equation}
    \abs{w(z) - \sgn(z)} \leq \frac{C}{1+ \abs{z}^{2s}} \qquad \forall z \in \R. \label{eq: decay of w}
\end{equation}
Lastly, we recall that $w'$ solves the fractional Allen--Cahn equation
\begin{equation}
    (-\partial_{zz})^s w' + W''(w) w' =0. \label{eq: linearized fractional AC}
\end{equation}

\subsection{The fractional heat kernel} \label{ss: fractional heat kernel}

We consider the solution to the fractional heat equation
\begin{equation}
    \begin{cases}
        \partial_t P^{(s)}_d + (-\Delta)^s P^{(s)}_d = 0 & (t, x) \in (0, +\infty) \times \R^d, 
        \\ P^{(s)}_d(0) = \delta_0 & x \in \R^d, 
    \end{cases}
\end{equation}
where the initial value is taken in the sense of distributions. More precisely, $P^{(s)}_d$ is defined via the Fourier transform. 

\begin{definition} \label{d: fractional heat kernel}
Fix $s \in (0,1)$. For $(t, x) \in (0, +\infty) \times \R^d$, we define the fractional heat kernel 
$$ P^{(s)}_d(t,x) := \int_{\R^d} \exp \left( - t  \abs{2\pi \xi}^{2s} \right)e^{2 \pi i x \cdot \xi} \, d \xi = \mathscr{F}_d^{-1} \left( \exp{ \left( - t \abs{2\pi (\cdot)}^{2s}\right)}\right) (x), $$
where $\mathscr{F}_d$ is the Fourier transform. 
\end{definition}

We mention \cite{Ga19}*{Chapter 16} and the references therein for a presentation of the topic. By a change of variables, it is easy to check that
\begin{equation} \label{eq: scaling of heat kernel}
    P^{(s)}_d (t,x) = t^{-\frac{d}{2s}} P^{(s)}_d \left (1, x t^{-\frac{1}{2s}}\right).  
\end{equation}
In analogy with the classical heat kernel, $P^{(s)}_d(t,x) >0$ for any $(t,x)$ (see \cite{Ga19}*{Proposition 16.3}). Since $\exp(- t \abs{\xi}^{2s})$ is a rapidly decaying function, then $P^{(s)}_d(t,\cdot)$ is smooth. However, due to lack of differentiability of  $\exp(- t \abs{\xi}^{2s})$ at $0$ for $s \in (0,1)$, then $P^{(s)}_d$ is not a Schwarz function. In particular, $P^{(s)}_d(1, \cdot)$ enjoys a polynomial decay, whereas the classical heat kernel decays exponentially. 

\begin{proposition} [\cite{Ga19}*{Proposition 16.5}] \label{p: decay of P^s}
Let $s \in (0,1)$ and $P^{(s)}_d(1, \cdot)$ be as in \cref{d: fractional heat kernel}. For any $x \in \R^d$ it holds that 
\begin{equation}
    \frac{C_1(d,s)}{1 + \abs{x}^{d+2s}} \leq P^{(s)}_d(1, x) \leq \frac{C_2(d,s)}{1 +\abs{x}^{d+2s}}. 
\end{equation}
\end{proposition}

By the properties of the Fourier transform, for any index $j = 1, \dots, d$ we have  
\begin{equation} \label{eq:Fourier-der}
    \partial_j P^{(s)}_{d}(1,x) = 2 \pi i \mathscr{F}_d^{-1}\left(  \xi_j \exp\left(- \abs{2 \pi \xi}^{2s} \right)\right) (x). 
\end{equation}

Using the Bochner relation (see e.g. \cite{St70}*{Section 3.2}), we have 
\begin{equation}
    \mathscr{F}_d (\xi_j \exp(-\abs{\cdot}^{2s}))(x) = i x_j \mathscr{F}_{d+2} (\exp(-\abs{\cdot}^{2s})) (\tilde{x}), \label{eq: bochner's relation}
\end{equation}
where we set $\tilde{x} = (x, 0,0) \in \R^d \times \R \times \R$. Hence, combining \eqref{eq: bochner's relation} and \cref{p: decay of P^s} we obtain the following result. 

\begin{proposition} \label{p: decay of derivative of P^s} 
Let $s \in (0,1)$ and $P^{(s)}_1(1, \cdot)$ be as in \cref{d: fractional heat kernel}. For any $k \in \N$ it holds  
\begin{equation}
    \abs{\partial_x^k P^{(s)}_1(1,x)} \leq \frac{C(k,s)}{1+ \abs{x}^{k+1+2s}} \qquad \forall x \in \R. \label{eq: decay of n-th derivative of P^s}
\end{equation}
\end{proposition}

\begin{proof}
We prove by induction that 
\begin{equation}
\partial^k_x P^{(s)}_1 (1, x) = \sum_{i= \lceil \frac{k}{2} \rceil }^k c_{i,k} x^{2i-k} P^{(s)}_{1+2i} (1, \tilde{x}), \qquad \forall x \in \R \label{eq: formula n-th derivative P^s}
\end{equation}
where $c_{i,k} $ are real-valued coefficients, $\lceil \cdot \rceil$ is the upper integer part and we denote by
\begin{equation}
    P^{(s)}_{d}(1, \tilde{x}) = P^{(s)}_{d}(1, (x, 0, \cdots, 0)) \qquad \forall x \in \R.
\end{equation}
Then, \eqref{eq: decay of n-th derivative of P^s} follows immediately by \eqref{eq: formula n-th derivative P^s} and \cref{p: decay of P^s}. To check the validity of \eqref{eq: formula n-th derivative P^s}, for $k=1$ by \eqref{eq: bochner's relation} we have 
\begin{equation}
    \partial_x P^{(s)}_1(1, x) = \tilde{c}_{1,1} x P^{(s)}_3(1, \tilde{x}). 
\end{equation}
Assume that $k$ is even, being the case of $k$ odd similar. Then, deriving \eqref{eq: formula n-th derivative P^s} we find explicit constants $\tilde{c}_{i,k}$ such that  
\begin{align}
    \partial^{k+1}_x P^{(s)}_1(1, x) & = \sum_{i= \lceil \frac{k}{2}\rceil +1 }^k c_{i,k} (2i-k) x^{2i-k-1} P^{(s)}_{1+2i}(1, \tilde{x}) + \sum_{i=\lceil \frac{k}{2}\rceil }^k \tilde{c}_{i,k} c_{i,k} x^{2(i+1)-k-1} P^{(s)}_{1+2(i+1)} (1, \tilde{x}) 
    \\ & = \sum_{i= \lceil \frac{k+1}{2}\rceil }^k c^1_{i,k+1} x^{2i-(k+1)} P^{(s)}_{1+2i}(1, \tilde{x}) + \sum_{i=\lceil \frac{k}{2}\rceil +1 }^{k+1} c_{i,k+1}^2 x^{2i-(k+1)} P^{(s)}_{1+2i} (1, \tilde{x}). 
\end{align}
\end{proof}

\subsection{Distance function} \label{ss: distance function}

Throughout this note, we adopt the following notation. Given an open set $E$, we denote by $\Sigma = \partial E$ and we let $\dist_{\Sigma}$ be the signed distance from $\Sigma$, i.e. 
\begin{equation}
    \dist_{\Sigma} (x) = \begin{cases}
        \inf \{ \abs{x-y} \colon y \in \Sigma \} & \text{ if } x \in E, 
        \\ - \inf \{ \abs{x-y} \colon y \in \Sigma \} & \text{ if } x \in E^c. 
    \end{cases}
\end{equation}

It is well known that the distance function is $1$-Lipschitz continuous and $\abs{ \nabla \dist_{\Sigma}(y)} =1$ at any point of differentiability. We refer to \cite{Ambrosio} and \cite{GT01}*{Section 14.6} for an overview of the properties of the distance function.  

\begin{definition} \label{d:principal coordinates}
Let $k \geq 2$. We say that $E \subset \R^d$ is an open set of class $C^k$ ($\partial E \in C^k$ in short) if for any $x_0 \in \Sigma = \partial E$ there exist $\delta > 0$ and a map $g \colon B_{\delta}^{d-1}\to \R$ of class $C^k$ such that, up to an affine isometry of the ambient space, it holds $x_0 =0$ and 
    \begin{align} \label{eq: parameterization 1}
        & T_{x_0} \Sigma = \mathrm{Span}(e_1,\dots,e_{d-1}), \\[1ex]
        \label{eq: parameterization 2}
        & \Sigma \cap \big(B_\delta^{d-1} \times (-\delta, \delta) \big)
        =
        \big\{(y, g(y)) \colon y \in B_\delta^{d-1} \big\}, \\[1ex]
        \label{eq: parameterization 2.5}
        & E \cap \big(B_\delta^{d-1} \times (-\delta, \delta) \big) = \big\{(y, t) \colon y \in B_\delta^{d-1}, t > g(y) \big\}.
\end{align} 
Moreover, we say that the map $Y := \mathrm{Id} \times g \colon B^{d-1}_\delta \to \R^d$ is a principal parameterization of $\Sigma$ around $x_0$ if $g$ satisfies 
    \begin{equation} \label{eq: parameterization 3}
        \nabla^2 g(0) = \mathrm{diag}(k_1, \dots, k_{d-1}), 
    \end{equation}
where $k_1, \dots, k_{d-1}$ denote the principal curvatures of $\Sigma$ at $x_0$ computed with respect to the outer unit normal. 
\end{definition}

\begin{remark}
Given $E \subset \R^d$ an open set of class $C^k$, then for any $x_0 \in \Sigma$ there exists a principal parameterization around $x_0$ according to \cref{d:principal coordinates}. Indeed, if $g$ satisfies \eqref{eq: parameterization 1}, \eqref{eq: parameterization 2} and \eqref{eq: parameterization 2.5}, then we find a linear isometry of $O : \R^{d-1} \to \R^{d-1}$ such that $g \circ O$ satisfies also~\eqref{eq: parameterization 3}.
\end{remark}

\begin{remark}
In the framework of \cref{d:principal coordinates}, then $\Sigma$ is locally the graph of $C^2$ function  $g \colon B^{d-1}_{\delta} \to \R$ and we have the following well-known formulas:
    \begin{gather}
        \label{eq:formula-tangent}
        T_{(y,g(y))} \Sigma
        = 
        \mathrm{Span}\big( (e_1, \partial_1 g(y)),\dots,(e_{d-1}, \partial_{d-1} g(y)) \big), \\[1.5ex] 
        \label{eq: formula normal} 
        N(y) = \frac{(-\nabla g(y), 1)}{\sqrt{1+\abs{\nabla g(y)}^2}}, \\ 
        \label{eq:mean-curv}
        H_{\Sigma}(y,g(y))
        =
        \mathrm{div} \left( \frac{\nabla g}{\sqrt{ 1 + \abs{\nabla g}^2}} \right)
        =
        \frac{\Delta g(y)}{\sqrt{ 1 + \abs{\nabla g(y)}^2}} - \frac{\nabla^2 g(y)[\nabla g(y), \nabla g(y)]}{(1+\abs{\nabla g(y)}^2)^{3/2}},
    \end{gather}
where $N(y)$ denotes the inner unit normal to $\Sigma$ at the point $(y,g(y))$ and $H_{\Sigma}(y,g(y))$ is the scalar mean curvature of $\Sigma$ at the point $(y,g(y))$ computed with respect to $-N(y)$. In addition, if the function $g$ satisfies
    \begin{equation}
        g(0)=0, \qquad \nabla g(0) = 0, \qquad \nabla^2 g(0) = \mathrm{diag}(k_1, \dots, k_{d-1}),
    \end{equation}
then the previous formulas at $0 \in \Sigma$ simplify to
    \begin{equation}
        \label{eq:formula-tangent-s}
        T_{0} \Sigma
        =
        \R^{d-1} \times \{ 0 \}, \qquad
        N(0) = (0,1), \qquad
        H_{\Sigma}(0)
        =
        \Delta g(0)
        =
        \sum_{i=1}^{d-1} k_i.
    \end{equation} 
\end{remark}

We recall some well-known properties of the signed distance function from an hypersurface (see for instance \cite{Ambrosio}, \cite{GT01}*{Lemma 14.16}, \cite{Modica}*{Lemma 3}).

\begin{lemma} \label{l: regularity of distance function}
Let $E \subset \R^d$ be a bounded open set of class $C^k$ for some $k \geq 2$. Then, there exists $\delta>0$ depending only on $\Sigma$ with the following properties.
\begin{itemize}
    \item [(i)] For any $x \in \Sigma_{\delta}$ there exists a unique point $\pi_{\Sigma}(x) \in \Sigma$ of minimal distance between $x$ and $\Sigma$. Moreover, the map $\pi_{\Sigma} \colon \Sigma_\delta \to \Sigma$ is of class $C^{k-1}$.
    \item [(ii)] For any $x \in \Sigma_\delta$ it holds that 
    \begin{equation}
    \nabla \dist_{\Sigma}(x) = N(\pi_{\Sigma}(x))  \label{eq: formula nabla dist},
    \end{equation}
    where $N$ is the inner unit normal to $\Sigma$. In particular, $\dist_{\Sigma} \in C^{k}(\Sigma_\delta)$. 
\end{itemize}
\end{lemma}

We consider a smooth version of the signed distance, which will play a key role in our analysis. 

\begin{definition} \label{d:regular distance} 
Let $E$ be a bounded open set of class $C^2$ and let $ 0< \delta< \sfrac{1}{5}$ be such that $5\delta$ satisfies the properties of \cref{l: regularity of distance function}. We denote by $\beta_{\Sigma}: \R^d \to [-1,1]$ any function of class $C^2(\R^d)$ such that $\abs{\beta_{\Sigma} (x)} \in [4\delta, 1]$ for any $x \in \Sigma_{5\delta} \setminus \Sigma_{4\delta}$ and 
\begin{equation} \label{eq: regular distance} 
    \beta_{\Sigma}(x) = 
    \begin{cases}
        \dist_{\Sigma}(x) & x \in \Sigma_{4\delta}, 
        \\ \sgn(\dist_{\Sigma}(x)) & x \in \R^d \setminus \Sigma_{5\delta}. 
    \end{cases} 
\end{equation}

\end{definition}

\subsection{Uniform Fermi coordinates} \label{ss: uniform fermi coordinates}

We are interested in sets possessing parameterization with uniform bounds. Following \cite{GT01}*{Section~14.6}, we introduce a useful notation.

\begin{definition} \label{d: uniform principal coordinates} 
Let $E \subset \R^d$ be an open set. We say that $\Sigma = \partial E$ is of $(\delta,C,k)$-type if $E$ is of class $C^k$ and there exist positive constants $\delta$, $C$ such that for any $x_0 \in \Sigma$ there exists a principal parameterization $Y=\mathrm{Id} \times g \colon B_{\delta}^{d-1} \to \R^d$ around $x_0$ as in \cref{d:principal coordinates} such that 
    \begin{equation} \label{eq: parameterization 4}
        \norm{g}_{C^k\big(B_{\delta}^{d-1}\big)} \leq C.
    \end{equation}  
Moreover, letting $N$ be the inner unit normal defined by \eqref{eq: formula normal}, we define the Fermi coordinates $\Phi :  B_{\delta}^{d-1} \times (-\delta , \delta) \to \R^{d} $ around $x_0$ by setting
\begin{equation} \label{eq: fermi coordinates}
    \Phi(y,z) := Y(y) + z N(y).
\end{equation}
\end{definition}

\begin{remark} \label{r: bounded sets are C-k uniform}
It is clear that if $E$ is an open set of class $C^k$ and with compact boundary, then there exist positive constants $\delta_0$, $C_0$ such that $\Sigma$ is of $(\delta, C, k)$-type according to \cref{d: uniform principal coordinates} for any $\delta \le \delta_0$ and for any $C \ge C_0$. Since the proof follows by a straightforward compactness argument, we leave the details to the reader. 
\end{remark}

\begin{remark}
Let $E \subset \R^d$ be an open set of $(\delta, C, k)$-type for some $C, \delta >0$ and $k \geq 3$. With the notation of \cref{d: uniform principal coordinates}, for any $x_0 \in \Sigma$ the map $g$ can be expanded near $y=0$ as follows
\begin{align}
    \label{eq:expansion of g}
    g(y)
    & =
    \frac{1}{2} \sum_{i=1}^{d-1} k_i y_i^2 +  O(\abs{y}^3), \\[1ex]
    \label{eq:expansion of nabla g}
    \partial_i g(y) & =  k_i y_i +O(\abs{y}^2), \\[2ex]
    \label{eq:expansion of D^2 g - 1}
    \partial_{ij} g(y) & = \delta_{ij} k_i + O(\abs{y}).
\end{align}
The reminders satisfy uniform estimates with respect to $x_0 \in \Sigma$ thanks to \eqref{eq: parameterization 4}.  
\end{remark}

In view of proving \cref{t:fractional laplacian}, we need to compute integrals involving suitable powers of the Euclidean distance in small domains contained in the tubular neighbourhood of the boundary of a smooth set. Since it is natural to use Fermi coordinates as local charts of the tubular neighbourhood, it is convenient to write the Euclidean distance in terms of the Fermi coordinates. 

\begin{definition} \label{d: tangential distance}
Let $E \subset \R^d$ be an open set of $(\delta, C, k)$-type for some $C, \delta >0$ and $k \geq 3$, $x_0 \in \Sigma_\delta$, $x_0' = \pi_\Sigma(x_0)$ and $z_0 = \dist_\Sigma(x_0)$. Let $\Phi$ be Fermi coordinates around $x_0'$. For any $y \in B_\delta^{d-1}$, we define the (squared) tangential distance from $x_0$ by setting 
\begin{equation} \label{eq: tangential distance}
    \abs{(Y(y) - z_0 e_d)_\tau}^2 := \abs{Y(y) - z_0 e_d}^2 - \abs{(Y(y) - z_0 e_d)\cdot N(y)}^2. 
\end{equation}
\end{definition}

We refer to \cref{r: another tangential distance} for a geometric interpretation of the tangential distance. We recall some Taylor expansions that will be useful later on, when we integrate on domains contained in the tubular neighbourhood using Fermi coordinates. 

\begin{lemma} \label{l: taylor expansion 1} 
Let $E \subset \R^d$ be an open set such that $\Sigma = \partial E$ is of $(\delta, C, 3)$-type according to \cref{d: uniform principal coordinates}, for some $\delta, C >0$. Let $x_0 \in \Sigma_\delta$, let $x_0' = \pi_\Sigma(x_0)$ and let $z_0 = \dist_\Sigma(x_0)$. Let $\Phi$ be Fermi coordinates around $x_0'$ and $\abs{(Y(y)-z_0 e_d)_\tau}$ be the tangential distance given by \cref{d: tangential distance}. Then, the following estimates hold true: 
\begin{gather}
    \label{eq: almost orthogonality 1}
    Y(y) \cdot N(y) = - \frac{1}{2} \sum_{i=1}^{d-1} k_i y_i^2 + O(\abs{y}^3), \\[1ex]
    \label{eq: normal deviation 1}
    1 - N_d(y) = O(\abs{y}^2), \\[1.5ex] 
    \label{eq: formula nabla phi(0,z)}
    \nabla \Phi(0,z) = \mathrm{diag}(1- z k_1, \dots, 1- z k_{d-1}, 1), \\[1ex]
    \label{eq: formula determinant}
    \det( \nabla \Phi(y,z)) =  \prod_{i=1}^{d-1}(1- z k_i) + O(\abs{y} \abs{z}) + O(\abs{y}^2), \\[1ex]
    \abs{(Y(y) - z_0 e_d)_\tau}^2 = \sum_{1=1}^{d-1} y_i^2(1-k_i  z_0)^2 + O(\abs{z_0} \abs{y}^3) + O(\abs{y}^4). \label{eq: tangential distance 1}
\end{gather}
The error terms in \eqref{eq: almost orthogonality 1}, \eqref{eq: normal deviation 1}, \eqref{eq: formula determinant}, \eqref{eq: tangential distance 1} satisfy uniform bounds with respect to $x_0 \in \Sigma_\delta \cap E$. 
\end{lemma}

\begin{remark} \label{r: another tangential distance}
With the notation of \cref{d: tangential distance}, we have that $x_0 = \Phi(0, z_0)$. By completing the squares with respect to $z$, it is easy to check that 
\begin{align}
    \abs{\Phi(y,z) -\Phi(0,z_0)} & = \abs{z-z_0 + Y(y)\cdot N(y) + z_0(1-N_d(y))}^2 + \abs{(Y(y) - z_0 e_d)_{\tau}}^2. 
\end{align}
Moreover, by \cref{l: taylor expansion 1}, we infer that
\begin{equation}
    \abs{z- z_0 + Y(y)\cdot N(y) + z_0(1- N_d(y))} = \abs{z- z_0} + O(\abs{y}^2),
\end{equation} 
which is the size of the ``normal" component of $\Phi(y, z)- \Phi(0, z_0)$. By \eqref{eq: tangential distance 1}, $\abs{(Y(y) - z_0 e_d)_{\tau}}$ is proportional to the size of the ``tangential" variable stretched by the curvature of $\Sigma$ at $\pi_\Sigma (x_0)$.  
\end{remark}

\begin{proof} [Proof of \cref{l: taylor expansion 1}]
The proof is a direct computation and we include it for the reader's convenience. By \eqref{eq:expansion of g}, \eqref{eq:expansion of nabla g}, we have 
\begin{align}
    Y(y) \cdot N(y) & = \frac{g(y) - \nabla g(y) \cdot y}{\sqrt{1 + \abs{\nabla g(y)}^2}} = \left( - \frac{1}{2} \sum_{i=1}^{d-1} y_i^2 k_i  + O(\abs{y}^3) \right) \left( 1 + O(\abs{y}^2) \right), 
\end{align}
thus proving \eqref{eq: almost orthogonality 1}. Similarly, we check \eqref{eq: normal deviation 1}, that is 
\begin{equation}
    1- N_d(y) = 1- \frac{1}{\sqrt{1 + \abs{\nabla g(y)}^2 }} = O(\abs{\nabla g(y)}^2) = O(\abs{y}^2 ).  
\end{equation}
In order to check \eqref{eq: formula nabla phi(0,z)} and \eqref{eq: formula determinant}, computing explicitly $\nabla \Phi(y,z)$, we have  
\begin{equation}
\nabla \Phi(y,z)  = 
\begin{pmatrix}
    1+ z \partial_1 N_1 & z \partial_2 N_1 & \cdots & z \partial_{d-1} N_1 & N_1
    \\ z \partial_1 N_2 & 1+ z \partial_2 N_2 & \cdots & z \partial_{d-1} N_2 & N_2
    \\ \vdots & \vdots & \ddots & \vdots & \vdots 
    \\ z \partial_1 N_{d-1} & z \partial_2 N_{d-1} & \cdots & 1+ z \partial_{d-1} N_{d-1} & N_{d-1}
    \\ \partial_1 g + z \partial_1 N_d & \partial_2 g + z \partial_{2} N_{d} & \cdots & \partial_{d-1}g + z \partial_{d-1} N_d & N_d
\end{pmatrix},
\end{equation}
Then, using \eqref{eq:expansion of nabla g}, \eqref{eq:expansion of D^2 g - 1}, for $i, j = 1, \dots, d-1$, we compute 
\begin{equation}
    \partial_i N_j
    =
    \frac{\partial_{ik} g \partial_i g \partial_k g - \partial_{ij} g (1 + \abs{\nabla g}^2)}{(1+ \abs{\nabla g}^2)^{3/2}} = -\delta_{ij} k_i + O(\abs{y}), 
\end{equation}
\begin{equation}
    \partial_i g + z \partial_i N_d =  \left( 1- \frac{z}{(1+ \abs{\nabla g}^2)^{3/2}} \right) \partial_i g = O(\abs{y}).  
\end{equation}
\begin{equation}
    N_i = \frac{-\partial_i g}{\sqrt{1 + \abs{\nabla g}^2}} = O(\abs{y}), \ \ \ N_d = 1+ O(\abs{y}^2) 
\end{equation}
where $g, N$ are always computed at $y$. Hence, we obtain that 
\begin{equation}
\nabla \Phi(y,z)  = 
\begin{pmatrix}
    1 - z k_1 + O(\abs{y} \abs{z}) &  \cdots & O(\abs{y}\abs{z}) & O(\abs{y})
    \\ O(\abs{y} \abs{z}) &  \cdots & O(\abs{y} \abs{z}) & O(\abs{y})
    \\ \vdots &  \ddots & \vdots & \vdots 
    \\ O(\abs{y} \abs{z}) & \cdots & 1 -  z k_{d-1} + O(\abs{y} \abs{z}) & O(\abs{y})
    \\ O(\abs{y}) &  \cdots & O(\abs{y}) & 1+ O(\abs{y}^2)
\end{pmatrix}. 
\end{equation}
Evaluating at $(0,z)$, then \eqref{eq: formula nabla phi(0,z)} is proved. To compute the determinant, expanding with respect to the last row, it results that 
\begin{equation}
    \det(\nabla \Phi(y,z)) = \prod_{i=1}^{d-1} (1- z k_i + O(\abs{y}\abs{z})) + O(\abs{y}^2) + O(\abs{y} \abs{z}),
\end{equation}
thus proving \eqref{eq: formula determinant}. To conclude, we check the validity of \eqref{eq: tangential distance 1}. By \eqref{eq: tangential distance}, \eqref{eq:expansion of g}, \eqref{eq:expansion of nabla g} it holds  
\begin{align}
    & \abs{(Y(y) - z_0 e_d )_\tau}^2  = \sum_{i=1}^{d-1} y_i^2 + (g(y) - z_0)^2 - \frac{1}{1 + \abs{\nabla g(y)}^2} \left( - \sum_{i=1}^{d-1} y_i \partial_i g(y) + (g(y) - z_0) \right)^2
    \\ & = \abs{y}^2 + (g(y)-z_0)^2 - (1- \abs{\nabla g(y)}^2 + O(\abs{y}^4)) \left( - 2g(y) + O(\abs{y}^3) + g(y)-z_0 \right)^2
    \\ & = \abs{y}^2 + (g(y)-z_0)^2 - (g(y)+z_0 + O(\abs{y}^3))^2  + \abs{\nabla g(y)}^2 (g(y)+z_0 + O(\abs{y}^3))^2 + O(\abs{y}^4) 
    \\ & = \abs{y}^2 + (g(y)-z_0)^2 - (g(y)+z_0)^2 + \sum_{i=1}^{d-1} y_i^2 k_i ^2  (g(y)+z_0 + O(\abs{y}^3))^2  + O(\abs{z_0} \abs{y}^3)
    \\ & = \abs{y}^2 - 4g(y) z_0 + \sum_{i=1}^{d-1}y_i^2 k_i ^2 z_0^2 + O(\abs{z_0} \abs{y}^3) + O(\abs{y}^4)
    \\ & = \abs{y}^2 - 2 \sum_{i=1}^{d-1} y_i^2 k_i  z_0 + \sum_{i=1}^{d-1}y_i^2 k_i ^2 z_0^2 + O(\abs{z_0} \abs{y}^3) + O(\abs{y}^4),
\end{align}
thus proving \eqref{eq: tangential distance 1}. 
\end{proof}

The following lemma is needed to justify the change of variables in the proof of \cref{t:fractional laplacian}. 

\begin{lemma} \label{l:inner ball} 
Under the assumptions of \cref{l: taylor expansion 1}, there exists a constant $\Lambda_0\geq 1$ depending on $\Sigma$ such that for any $\Lambda \geq \Lambda_0$ for any $x_0 \in \Sigma_{\sfrac{\delta}{10\Lambda}}$ it holds that 
\begin{equation}                        
\mathscr{B}_{\sfrac{\delta}{\Lambda}}^{d-1}(z_0) := \left\{ y \in \R^{d-1} \colon \abs{(\mathrm{Id}- z_0 \nabla^2 g(0)) y} < \sfrac{\delta}{\Lambda} \right\} \subset B_{\sfrac{\delta}{2} }^{d-1}, \label{eq: outer ball 1}
\end{equation}
where as before $z_0 = \mathrm{dist}_{\Sigma}(x_0)$. In addition, for any $y \in \mathscr{B}_{\sfrac{\delta}{\Lambda} }^{d-1}(z_0)$ it holds that 
\begin{equation}
    \mathscr{I}_{\sfrac{\delta}{\Lambda} }(y,z_0) := \left\{ z \in \R \colon \abs{z-z_0 + Y(y)\cdot N(y) + z_0(1-N_d(y))} < \sfrac{\delta}{\Lambda} \right\} \subset \left(-\sfrac{\delta}{2}, \sfrac{\delta}{2} \right). \label{eq: outer ball 2}
\end{equation}
Then, denoting by 
\begin{equation}
    \mathcal{B}_{\sfrac{\delta}{\Lambda}} (z_0) := \left\{ (y,z) \in \mathscr{B}_{\sfrac{\delta}{\Lambda}}^{d-1}(z_0) \times \R \colon z \in \mathscr{I}_{\sfrac{\delta}{\Lambda}}(y, z_0) \right\},  \label{eq: domain B}
\end{equation} 
it holds that $\mathcal{B}_{\sfrac{\delta}{\Lambda}}(z_0) \subset B_{\sfrac{\delta}{2}}^{d-1}\times \left(-\sfrac{\delta}{2}, \sfrac{\delta}{2} \right)$.
Finally, we have that $B_{\sfrac{\delta}{10 \Lambda}} (x_0) \subset \Phi (\mathcal{B}_{\sfrac{\delta}{\Lambda}}(z_0))$.
\end{lemma}

\begin{proof}
Let $\Lambda \geq 1$ be a constant, which will be fixed shortly. We remark that $\abs{ k_i } \leq C $ for any $i=1, \dots, d-1$ (see~\eqref{eq: parameterization 4}). We know that $\abs{z_0}\leq \sfrac{\delta}{10 \Lambda}$, therefore for any $y \in \mathscr{B}_{\sfrac{\delta}{\Lambda}}^{d-1}(z_0)$, by the triangular inequality, it is clear that  
\begin{equation}
    \abs{y} \leq \frac{\delta}{\Lambda} \left( 1- \frac{C \delta}{10 \Lambda} \right)^{-1}. \label{eq: condition 0}
\end{equation}
Hence, \eqref{eq: outer ball 1} is satisfied provided that 
\begin{equation}
    \frac{\delta}{\Lambda} \left( 1- \frac{C \delta }{10 \Lambda} \right)^{-1} \leq \frac{\delta}{2}. \label{eq: condition 1}
\end{equation}
From~\cref{l: taylor expansion 1}, there exists a constant $\bar{C} > 0$ depending on $\Sigma$ such that for any $\Lambda \ge 1$ we have
\begin{equation}
    \abs{Y(y) \cdot N(y) + z_0( 1- N_d(y) )} \leq \bar{C} \abs{y}^2 \qquad \forall y \in \mathscr{B}^{d-1}_{\sfrac{\delta}{\Lambda}}(z_0).  \label{eq: inner variation} 
\end{equation} 
Thus, given $y \in \mathscr{B}^{d-1}_{\sfrac{\delta}{\Lambda}}(z_0)$ and $ z \in \mathscr{I}_{\sfrac{\delta}{\Lambda}}(y, z_0)$, by the triangular inequality,~\eqref{eq: condition 0} and~\eqref{eq: inner variation} it holds that 
\begin{equation}
    \abs{z} \leq \frac{\delta}{\Lambda} + \frac{\delta}{10 \Lambda} + \bar{C} \frac{\delta^2}{\Lambda^2} \left( 1- \frac{C \delta }{10 \Lambda} \right)^{-2}. \label{eq: condition 0 bis}
\end{equation}
Hence, we infer that \eqref{eq: outer ball 2} is satisfied provided that 
\begin{equation}
    \frac{\delta}{\Lambda} + \frac{\delta}{10 \Lambda} + \bar{C} \frac{\delta^2}{\Lambda^2} \left( 1- \frac{C \delta }{10 \Lambda} \right)^{-2} \leq \frac{\delta}{2}. \label{eq: condition 2}
\end{equation}
To summarize, if $\Lambda$ satisfies \eqref{eq: condition 1} and \eqref{eq: condition 2}, then $\mathcal{B}_{\sfrac{\delta}{\Lambda}}(z_0) \subset B_{\sfrac{\delta}{2}}^{d-1}\times \left(-\sfrac{\delta}{2}, \sfrac{\delta}{2}\right)$. More precisely, using again the triangular inequality,~\eqref{eq: condition 0} and~\eqref{eq: inner variation}, it can be checked in the same way that $ B^{d-1}_{\sfrac{\delta}{2\Lambda}} \times \left( z_0 - \sfrac{\delta}{2\Lambda} , z_0 + \sfrac{\delta}{2\Lambda} \right) \subset \mathcal{B}_{\sfrac{\delta}{\Lambda}}(z_0) $ provided that the following condition is satisfied: 
\begin{equation}
    \max \left\{ \frac{\delta}{2\Lambda}\left( 1+ \frac{C \delta}{10 \Lambda} \right) , \frac{\delta}{2\Lambda} + \bar{C} \frac{\delta^2}{4\Lambda^2} \right\} \leq \frac{\delta}{\Lambda}. \label{eq: condition 3}
\end{equation}
Lastly, to prove that $B_{\sfrac{\delta}{10 \Lambda}} (x_0) \subset \Phi (\mathcal{B}_{\sfrac{\delta}{\Lambda}}(z_0))$ for $\Lambda$ sufficiently large, for any $t \in [0,1]$ we define the map $\widehat{\Phi}_t \colon B^{d-1}_{\sfrac{\delta}{2\Lambda}} \times \left( z_0 - \sfrac{\delta}{2\Lambda} , z_0 + \sfrac{\delta}{2\Lambda} \right) \to \R^d$ as follows
    \begin{equation}
        \label{eq:Phi-hom}
        \widehat{\Phi}_t(y,z) := t (\Phi(y,z) - x_0) + (1-t) (y,z-z_0). 
    \end{equation}
Hence, using~\eqref{eq: formula nabla phi(0,z)} we shall write 
\begin{equation} \label{eq:deg-bndry-est}
     \widehat{\Phi}_t(y,z) 
     =
     (y, z-z_0) - t z_0 (k_1 y_1, \dots, k_{d-1} y_{d-1}, 0)
     + O(\abs{y}^2 + \abs{z-z_0}^2).  
\end{equation}
The previous formula implies that if $\Lambda$ is sufficiently large, then 
    \begin{equation} \label{eq:away-bndry}
        |\widehat{\Phi}_t(y,z)| > \sfrac{\delta}{10 \Lambda}, \qquad \forall t \in [0,1], \forall (y,z) \in 
        \partial \left( B_{\sfrac{\delta}{2 \Lambda}}^{d-1} \times (z_0 - \sfrac{\delta}{ 2 \Lambda}, z_0 + \sfrac{\delta}{2 \Lambda}) \right). 
    \end{equation}
Therefore, for every $p \in B_{\sfrac{\delta}{ 10 \Lambda}}(0)$, the standard properties of the degree imply that
    \begin{equation}
        \mathrm{deg}\Big(\widehat{\Phi}_1,p,B_{\sfrac{\delta}{2 \Lambda}}^{d-1} \times (z_0 - \sfrac{\delta}{ 2 \Lambda}, z_0 + \sfrac{\delta}{2 \Lambda})\Big)
        = \mathrm{deg}\Big(\widehat{\Phi}_0,p,B_{\sfrac{\delta}{2 \Lambda}}^{d-1} \times (z_0 - \sfrac{\delta}{ 2 \Lambda}, z_0 + \sfrac{\delta}{2 \Lambda})\Big)
        = 1,
    \end{equation}
and $B_{\sfrac{\delta}{10 \Lambda}} (x_0) \subset \Phi (\mathcal{B}_{\sfrac{\delta}{\Lambda}}(z_0))$.
In particular, \eqref{eq: condition 1}, \eqref{eq: condition 2}, \eqref{eq: condition 3} and~\eqref{eq:away-bndry} are satisfied if $\Lambda$ is large enough. 
\end{proof}

\subsection{Approximation of sets} 
The proof of Theorem~\ref{thm:main} is more direct when the set $E$ is smooth and intersects the boundary of
$\Omega$ transversely (in a measure theoretic sense). To handle the general case, we approximate $E$ with smooth bounded open sets such that both the Perimeter and the Willmore energy of the approximating sets in $\Omega$ converge to those of $E$.
When $\Omega = \R^d$ there are several ways to construct such an approximation, see e.g.~\cite{Antonini} and the references therein. In the following lemma we show that the same conclusion holds when $\Omega$ is a bounded open set of class~$C^1$. Moreover, the approximating sets that we consider intersect the boundary of $\Omega$ transversely.

\begin{lemma} \label{l:approximation}
    Let $E \subset \R^d$ be a bounded open set with $\partial E \in C^2$. For any bounded open set $\Omega \subset \R^d$ with $\partial \Omega \in C^1$, there exists a sequence $\{ E_j \}_{j \in \mathbb{N}}$ of smooth bounded open sets of $\R^d$ such that 
    \begin{enumerate} [label=($A1$),ref=$A1$]  
        \item\label{approx 1} $\displaystyle \lim_{j \to \infty} \abs{E_j \Delta E} = 0$, 
    \end{enumerate}
    \begin{enumerate} [label=($A2$),ref=$A2$] 
        \item \label{approx 2} $\displaystyle \lim_{j \to \infty} \mathrm{Per}(E_j,\Omega) = \mathrm{Per}(E,\Omega)$,
    \end{enumerate}
    \begin{enumerate} [label=($A3$),ref=$A3$] 
        \item \label{approx 3} $\displaystyle \lim_{j \to \infty} \mathcal{W}(\partial E_j, \Omega) = \mathcal{W}(\partial E, \Omega)$,
    \end{enumerate}
    \begin{enumerate} [label=($A4$),ref=$A4$] 
        \item \label{approx 4} $\displaystyle \lim_{j \to \infty} \mathcal{H}^{d-1}(\partial E_j \cap \partial \Omega) = 0$, 
    \end{enumerate}
    \begin{enumerate} [label=($A5$),ref=$A5$] 
        \item \label{approx 5} $\displaystyle \sup_{j \in \N} \norm{H_{\partial E_j}}_{C^0(\partial E_j)} \leq \norm{H_{\partial E}}_{C^0(\partial E)} +1 .$ 
    \end{enumerate} 
\end{lemma}

\begin{proof}
To begin, we show that there exists a sequence of sets of class $C^2$ with the above properties. Since $\partial \Omega$ is of class $C^1$, then $N_{\partial \Omega} \in C(\partial \Omega; \mathbb{S}^{d-1})$, where $N_{\partial \Omega}$ denotes the inner unit normal to $\Omega$. Hence, we find a smooth vector field $X \in C^{\infty}_c(\R^d; \R^d)$ such that
    \begin{equation} \label{eq:approx-normal}
        \langle X(x) , N_{\partial \Omega}(x) \rangle \le - \frac{1}{2}, \qquad \forall x \in \partial \Omega. 
    \end{equation}
    For any $t >0$, we define $f_t(x):=x+tX(x)$. In particular, since $X$ has compact support, there exists $t_0 > 0$ such that $f_t$ is a diffeomorphism of class $C^\infty$ for any $t \in (0,t_0]$. We consider the set
    \[
        E_t:=f_t(E),
    \]
    and we claim that there exists a sequence of positive real numbers $\{ t_j \}_{j \in \mathbb{N}}$ converging to zero for which the corresponding sets $E_j:=E_{t_j}$ fulfill \eqref{approx 1}, \eqref{approx 2}, \eqref{approx 3}, \eqref{approx 4} and~\eqref{approx 5}.

    Letting $g_t := f_t^{-1}$, we prove that for $t$ small enough $g_t$ pushes $\Omega$ inside, that is there exists $t_1 \in (0,t_0]$ such that $g_t(\Omega) \subset \Omega$ for any $t \in (0,t_1]$. Suppose by contradiction that there exist a sequence of positive real numbers $\{ t_n \}_{n \in \mathbb{N}}$ converging to zero and a sequence of points $\{ y_n \}_{n \in \mathbb{N}} \subset \Omega$ such that $x_n:=g_{t_n}(y_n) \in \R^{d} \setminus \Omega$. Since $\Omega$ is bounded, we can assume that $y_n$ converges to a limit point $z_{\infty} \in \overline{\Omega}$. On the other hand, $g_t$ converges uniformly to the identity map as $t \to 0$, therefore $x_n$ converges to the same limit point $z_{\infty}$. Since $x_n \in \R^{d} \setminus \Omega$ for any $n \in \mathbb{N}$, we have that $z_{\infty} \in \partial \Omega$. Using a local chart near $z_{\infty}$, we reduce to the case $z_\infty = 0$ and $\Omega = \{ x \in \R^d \colon x_d < 0 \}$. Hence, it is clear that $(y_n - x_n)/t_n$ is pointing inside $\Omega$. On the other hand, by \eqref{eq:approx-normal} and the definition of $f_t$, we have
        \[
            \lim_{n \to \infty} \Big\langle \frac{y_n-x_n}{t_n}, N_{\partial \Omega}(z_{\infty}) \Big\rangle
            =
            \lim_{n \to \infty}
            \Big\langle \frac{f_{t_n}(x_n)-x_n}{t_n}, N_{\partial \Omega}(z_{\infty}) \Big\rangle
            =
            \langle X(z_{\infty}), N_{\partial \Omega}(z_{\infty}) \rangle
            \le - \frac{1}{2},
        \]
    which is a contradiction because $N_{\partial \Omega}(z_{\infty})$ is pointing inside $\Omega$.

\textsc{Proof of \eqref{approx 1}.} For any $t \in (0,t_0]$, we have that $\partial E_t = f_t(\partial E) \in C^2$ since $f_t$ is a smooth diffeomorphism and $\partial E \in C^2$. Moreover, for any $x \in \R^d \setminus \partial E$, it holds that $\mathds{1}_{E_t}(x) \to \mathds{1}_{E}(x)$ as $t\to 0$. Indeed, for any $x \in \R^d, t \in (0,t_0]$ we have that 
    \begin{equation}
        \mathds{1}_{f_t(E)}(x) = \mathds{1}_{f_t(E)} (f_t(g_t(x)) = \mathds{1}_{E}(g_t(x)). 
    \end{equation}
    Since $g_t$ converges to the identity map uniformly as $t \to 0$ and $E$ is open, for any $x \in E$ then $g_t(x) \in E$ we have that $g_t(x) \in E$ for $t$ small enough. The same argument works for $x \in \R^d \setminus \overline{E}$, thus proving \eqref{approx 1}. 
    
    \textsc{Proof of \eqref{approx 2}.} Since the Perimeter is lower semicontinuous with respect to the $L^1$-convergence of sets (see e.g.~\cite{M12}*{Proposition~12.15}), we have
        \begin{equation}
            \label{eq:liminf-per}
            \liminf_{t \to 0} \mathrm{Per}(E_t, \Omega) \ge \mathrm{Per}(E, \Omega).
        \end{equation}
    We prove the opposite inequality. We set for convenience $\mu_{E_t}:=\mathcal{H}^{d-1} \llcorner \partial E_t$. It is known (see e.g~\cite{M12}*{(17.6), (17.29), (17.30)}) that
        \begin{equation}
            \label{eq:pushforward-1}
            (g_t)_{\#} \mu_{E_t} = p_t \mu_E, \quad \text{with} \quad \norm{p_t}_{C^0(\partial E)} = 1 + O(t).
        \end{equation} 
    Combining this fact with the property that $\Omega \subset g_t^{-1}(\Omega)$ we obtain
        \begin{equation}
            \label{eq:limsup-per-1}
            \mu_{E_t} (\Omega)
            \le
            \mu_{E_t} (g_t^{-1}(\Omega))
            =
            (g_t)_{\#}  \mu_{E_t} (\Omega)
            =
            \mu_{E} (\Omega) + O(t).
        \end{equation}
    Then, \eqref{approx 2} follows by taking the limsup as $t$ goes to zero in the previous inequality. 

    \textsc{Proof of \eqref{approx 3} and \eqref{approx 5}.} It was proved by Schätzle (see~\cite{Sch09}) that the Willmore functional is lower semicontinuous with respect to the $L^1$-convergence of $C^2$ sets, that is 
        \begin{equation}
            \label{eq:liminf-Will}
            \liminf_{t \to 0} \mathcal{W}(\partial E_t, \Omega) \ge \mathcal{W}(\partial E, \Omega).
        \end{equation}
    In our case, there are simpler and more direct ways to show the lower semicontinuity property above. For example, taking into account~\eqref{approx 2}, then~\eqref{eq:liminf-Will} follows by an application of Reshetnyak’s continuity theorem, see~\cite{AM03}*{Remark~2} and also~\cite{LM89}*{Lemma~2} for a particular case.
    At this point, we claim that
        \begin{equation} \label{eq:H-exp}
            H_{\partial E_t}(y) = H_{\partial E}(g_t(y)) + O(t), \qquad \forall y \in \partial E_t,
        \end{equation}
    where the reminder term is uniform with respect to $y \in \partial E_t$. Thus, the sets $\{E_t\}_{t \in (0, t_1] }$ satisfy \eqref{approx 5} provided that $t_1$ is small enough. By \eqref{eq:pushforward-1},~\eqref{eq:H-exp} and since $\Omega \subset g_t^{-1}(\Omega)$, we have
        \begin{equation}
            \mathcal{W}(\partial E_t, \Omega) \le \mathcal{W}(\partial E_t , \Omega) + O(t)
        \end{equation}
    and the conclusion follows taking the limsup as $t$ goes to zero in the inequality above. To check \eqref{eq:H-exp}, let $\psi \in C^2(\R^d)$ such that $E = \{ x \colon \psi(x) > 0 \}$ and $\nabla \psi (x) \neq 0$ for any $x \in \partial E$. It is well known that
        \begin{gather} 
            \label{eq:normal-psi}
            N_{\partial E} = \frac{\nabla \psi}{\abs{\nabla \psi}} \quad \text{on $\{\psi = 0\}$}, \\[0.5ex]
            \label{eq:curv}
            H_{\partial E}
            = 
            - \mathrm{div} \bigg( \frac{\nabla \psi}{\abs{\nabla \psi}} \bigg) = - \frac{\Delta \psi}{\abs{\nabla \psi}} + \frac{\nabla^2 \psi[\nabla \psi, \nabla \psi]}{\abs{\nabla \psi}^3} \quad \text{on $\{\psi = 0\}$},
        \end{gather}
    where $N_{\partial E}$ denotes the inner unit normal to $E$. In particular, the right-hand sides in the previous identities do not depend on the particular choice of $\psi$. For any $t \in (0,t_1]$, we define $\psi_t \in C^2(\R^d)$ by $\psi_t := \psi \circ g_t$. It is clear that $E_t = \{ y \colon \psi_t(y) > 0 \}$ and that $\nabla \psi_t (y) \neq 0$ for any $y \in \partial E_t$. We set $G_t(y) := \nabla g_t(y)$ and we denote by $G_t^*(y)$ its transpose matrix. A direct computation shows
        \begin{gather} \label{eq:diff-1}
            \nabla \psi_t(y) = G_t^*(y) \nabla \psi (g_t(y)), \\[0.5ex]
            \label{eq:diff-2}
            \nabla^2 \psi_t(g_t(y)) = G_t^*(y) \nabla^2 \psi (g_t(y)) G_t(y) + \langle \nabla \psi (g_t(y)), \nabla^2 g_t (y) \rangle,
        \end{gather}
    where $\nabla^2 g_t$ is the vector valued Hessian of $g_t$. Combining~\eqref{eq:curv},~\eqref{eq:diff-1},~\eqref{eq:diff-2} and the fact that $\norm{g_t - \mathrm{Id}}_{C^2(\R^d)} \to 0$ as $t \to 0$ we deduce~\eqref{eq:H-exp}.
    
    \textsc{Proof of \eqref{approx 4}.} For any $t_2 \in (0,t_1]$, we define the set
        \begin{equation}
            \label{eq:N-set}
            \mathcal{N}(t_2):=\{ t \in (0,t_2] \colon \mathcal{H}^{d-1}(\partial E_t \cap \partial \Omega) > 0 \},
        \end{equation}
    and we claim that there exists $t_2 \in (0,t_1]$ such that $\mathcal{N}(t_2)$ is at most countable. If this is the case, then we find a sequence of positive real numbers $\{ t_j \}_{j \in \mathbb{N}}$ converging to zero for which the corresponding sets satisfy a stronger property than \eqref{approx 4}, namely $\mathcal{H}^{d-1}(\partial E_{t_j} \cap \partial \Omega) = 0$. 

    We write $\partial E_t \cap \partial \Omega = A_t \cup B_t$, where
        \begin{align}
            A_t & := \{ z \in \partial E_t \cap \partial \Omega \colon T_z \partial E_t = T_z \partial \Omega \} = \{ z \in \partial E_t \cap \partial \Omega \colon \abs{\langle N_{\partial E_t}(z) , N_{\partial \Omega}(z) \rangle} = 1 \}, \\[0.5ex]
            B_t & := \{ z \in \partial E_t \cap \partial \Omega \colon T_z \partial E_t \neq T_z \partial \Omega \} =
            \{ z \in \partial E_t \cap \partial \Omega \colon \abs{\langle N_{\partial E_t}(z) , N_{\partial \Omega}(z) \rangle} < 1 \},
        \end{align}
    where $N_{\partial E_t}$ denotes the inner unit normal to $E_t$. It is well known (see e.g.~\cite{Hirsch94}*{Section~1, Theorem~3.3}) that the transverse intersection of two submanifolds of codimensions $k_1$ and $k_2$ is either empty or a submanifold of codimension $k_1+k_2$. Therefore, for any $t \in (0,t_1]$, $B_t$ is either empty or a $(d-2)$-dimensional submanifold of $\R^d$. In both cases, we have that $\mathcal{H}^{d-1}(B_t)=0$. We prove that there exists $t_2 \in (0,t_1]$ such that for any $t,s \in (0, t_2]$ with $t \neq s$ we have $A_t \cap A_s = \emptyset$. Suppose by contradiction that there exist two sequences of positive real numbers $0 < s_n < t_n$, with $t_n \to 0$ as $n \to \infty$, and two sequences of points $\{ x_n \}, \{ y_n \} \subset \partial E$ such that $f_{t_n}(x_n) = f_{s_n}(y_n) \in \partial \Omega$. Since $\partial E$ is compact, up to subsequences, we may assume that $x_n \to z_{\infty} \in \partial E$. Since $x_n = f_{t_n}(x_n) - t_n X(x_n)$, it follows that $z_n:=f_{t_n}(x_n) \to z_{\infty}$ and we infer that $z_{\infty} \in \partial \Omega$. Moreover, we have
        \begin{equation}
            \abs{x_n - y_n} = \abs{t_n X(x_n) - s_n X(y_n)} \le 2 t_n \max \{ \abs{X(x)} \colon x \in \partial E \},
        \end{equation}
    therefore $y_n \to z_\infty$. We claim that $z_{\infty} \in A_0$. Since $z_n=f_{t_n}(x_n) \in A_{t_n}$ for any $n \in \mathbb{N}$, we have
        \begin{equation} \label{eq:A0-closure}
            \abs{
            \langle N_{\partial E_t}(z_n), N_{\partial \Omega}(z_n) \rangle} = 1.
        \end{equation}
    On the other hand, from~\eqref{eq:normal-psi} and~\eqref{eq:diff-1} we have
        \begin{equation}
            N_{\partial E_t}(z)
            =
            \frac{G_t^*(z)N_{\partial E}(g_t(z))}{\abs{ G_t^*(z)N_{\partial E}(g_t(z))}}, \qquad \forall z \in \partial E_t.
        \end{equation}
    Therefore, we have $N_{\partial E_t}(z_n) \to N_{\partial E}(z_{\infty})$ as $n \to \infty$, since $z_n \to z_{\infty}$ and $g_t$ converges in $C^1$ to the identity map as $t \to 0$. Then, by \eqref{eq:A0-closure} we derive
        \begin{equation}
            \abs{
            \langle N_{\partial E}(z_{\infty}), N_{\partial \Omega}(z_{\infty}) \rangle}
            =
            \lim_{n \to \infty}
            \abs{
            \langle N_{\partial E_t}(z_n), N_{\partial \Omega}(z_n) \rangle}
            =
            1.
        \end{equation}
    This proves that $z_{\infty} \in A_0 = \{ z \in \partial E \cap \partial \Omega \colon \abs{\langle N_{\partial E}(z), N_{\partial \Omega}(z) \rangle} = 1 \}$. 

    At this point, we claim that the map $F \colon \partial E \times \R \to \R^d$ defined as $F(x,t):=f_t(x)$ is a local diffeomorphism around $(z_{\infty},0)$. If this is the case, then we find a contradiction since, for $n$ large, $F(x_n, t_n)=F(y_n,s_n)$ implies that $(x_n, t_n)=(y_n,s_n)$, but $s_n < t_n$. To prove the claim we have to check that the differential of $F$ at the point $(z_{\infty},0)$ is surjective. It is not difficult to see that the image of the differential at $(z_\infty, 0)$ is given by
        \begin{equation}
            V:=T_{z_{\infty}} \partial E \oplus \mathrm{Span} (X(z_{\infty})).
        \end{equation}
    Now, $V=\R^d$ if and only if $\langle X(z_{\infty}) , N_{\partial E}(z_{\infty})  \rangle \neq 0$. Since $z_{\infty} \in A_0$ we have $\abs{\langle X(z_{\infty}) , N_{\partial E}(z_{\infty})  \rangle} = \abs{\langle X(z_{\infty}) , N_{\partial \Omega}(z_{\infty})  \rangle}$ and the latter is different from zero because of~\eqref{eq:approx-normal}.

    \textsc{Building smooth sets.} To summarize, we have constructed a sequence $\{ E_j \}_{j \in \mathbb{N}}$ of bounded open sets of class $C^2$ satisfying \eqref{approx 1}, \eqref{approx 2}, \eqref{approx 3}, \eqref{approx 5} and the additional property $\mathcal{H}^{d-1}(\partial E_j \cap \partial \Omega) = 0$ for any $j \in \N$. To conclude, we want to pass from $C^2$ to $C^\infty$ sets. For any $j \in \mathbb{N}$, there exists a sequence of smooth bounded open sets $E_{k,j}$ converging to $E_j$ in $L^1(\R^d)$, as $k \to \infty$, and such that
        \begin{equation}        \label{eq:approx-energy}
            \lim_{k \to \infty} \mathrm{Per}(E_{k,j}, \R^d) = \mathrm{Per}(E_j, \R^d)
            \quad \text{and} \quad  
            \lim_{k \to \infty} \mathcal{W}(\partial E_{k,j}, \R^d) = \mathcal{W}(\partial E_j, \R^d).
        \end{equation}
    Moreover, the sequence $\{E_{j,k}\}_{k\in \N}$ can be chosen such that \eqref{approx 5} is satisfied, with $H_{\partial E_j}$ at the right-hand side. We refer e.g.~to~\cite{Antonini} for a rigorous proof of this fact and more general results about the approximation by smooth sets on the whole Euclidean space. By $\mathcal{H}^{d-1}(\partial E_j \cap \partial \Omega) = 0$, \eqref{eq:approx-energy} and the lower semicontinuity of the Perimeter and the Willmore functional on $\Omega$ and on $\R^d \setminus \overline{\Omega}$, we infer 
         \begin{equation}        \label{eq:approx-energy-2}
            \lim_{k \to \infty} \mathrm{Per}(E_{k,j}, \Omega) = \mathrm{Per}(E_j, \Omega)
            \quad \text{and} \quad  
            \lim_{k \to \infty} \mathcal{W}(\partial E_{k,j}, \Omega) = \mathcal{W}(\partial E_j, \Omega).
        \end{equation}
    Moreover, by the upper semicontinuity of the evaluation on closed sets with respect to the weak convergence of measures we have
        \begin{equation} \label{eq:approx-energy-3}
            \limsup_{k \to \infty} \mathcal{H}^{d-1}(\partial E_{k,j} \cap \partial \Omega) \le \mathcal{H}^{d-1}(\partial E_j \cap \partial \Omega) = 0.
        \end{equation}
    The conclusion follows from~\eqref{eq:approx-energy-2},~\eqref{eq:approx-energy-3} and a diagonal argument.
\end{proof}

\section{On the decay of optimal profile} \label{s: decay optimal profile}

In this section, we discuss the proof of \cref{thm:hi-decay}. Since $w'$ solves the fractional Allen--Cahn equation, then for any $\lambda > 0$ we have that 
\begin{equation}
    ((-\Delta)^s + \lambda) w' = (\lambda - W''(w)) w', 
\end{equation}
The proof of \cref{thm:hi-decay} relies on the decay properties of the fundamental solution $G_{s,\lambda}$ of the operator $\lambda + (-\Delta)^s$, whose symbol is $\lambda +  \abs{ 2\pi \xi}^{2s}$. Since $G_{s,\lambda}$ is formally defined by 
\begin{equation}
    G_{s,\lambda}(x) = \mathscr{F}_1^{-1} \left( \left(\lambda +  \abs{2 \pi \xi}^{2s} \right)^{-1}\right)(x) = \int_{\R} \frac{e^{2\pi ix \xi}}{ \lambda + \abs{2\pi \xi}^{2s}  } \, d\xi, 
\end{equation}
and for any $\xi \in \R$ we have
\begin{equation}
    \frac{1}{ \lambda + \abs{2\pi \xi}^{2s}  } = \int_0^{+\infty} e^{-\lambda t} \exp \left(- t \abs{2 \pi \xi}^{2s} \right) \, dt,
\end{equation}
then $G_{s,\lambda}$ formally satisfies 
\begin{equation} \label{eq: heuristic G_s,lambda}
   G_{s,\lambda}(x) =  \int_{\R} e^{2\pi i x \xi} \int_0^{+\infty} e^{-\lambda t}\exp \left(- t \abs{2\pi \xi}^{2s}\right) \, dt \, d \xi = \int_{0}^{+\infty} e^{-\lambda t} P^{(s)}_1 (t, x) \, dt. 
\end{equation} 
The above computation can be made rigorous. 

\begin{proposition} \label{p: decay of fundamental solution}
Fix $s \in (0,1)$ and $\lambda >0$. For any $x \in \R$, set 
\begin{equation} \label{eq: fundamental solution}
    G_{s,\lambda}(x) : = \int_{0}^{+\infty} e^{-\lambda t} P^{(s)}_1(t,x) \, dt. 
\end{equation} 
Then, $G_{s,\lambda} \in L^1(\R)$ and it holds that 
\begin{equation}
    \mathscr{F}_1(G_{s,\lambda})(\xi) = \frac{1}{ \lambda + \abs{2\pi \xi}^{2s} } \qquad \forall \xi \in \R. \label{eq: fourier of fundamental solution} 
\end{equation} 
Moreover, $G_{s,\lambda} \in C^{\infty}(\R \setminus \{0\})$ and for any $k \geq 0 $ it holds that 
\begin{equation} \label{eq: decay of fund sol}
    \abs{\partial_x^k G_{s,\lambda} (x)} \leq C(s,\lambda, k) \abs{x}^{-k-1-2s} \qquad \forall x\neq 0. 
\end{equation}
\end{proposition}

\begin{proof}
To begin, since $P^{(s)}_1(t, x) >0$ for any $(t,x) \in (0, +\infty) \times \R$, we notice that $G_{s,\lambda}(x)$ is always well defined with values in $[0, +\infty]$. Then, by \eqref{eq: scaling of heat kernel} and Fubini's theorem, it results that 
\begin{align}
    \int_{\R} G_{s,\lambda}(x) & = \int_0^{+\infty} \int_{\R} e^{-\lambda t} t^{-\frac{1}{2s}} P^{(s)}_1\left(1, t^{-\frac{1}{2s}}x\right)\, dx \, dt = \int_0^{+\infty} e^{-\lambda t} \int_{\R} P^{(s)}_{1} (1, y)\, dy \, dt < +\infty. 
\end{align} 
From now on, we neglect constants $C(k,s,\lambda)>0$. Moreover, by \cref{p: decay of P^s}, we have that 
\begin{equation}
    G_{s,\lambda}(x) \lesssim \int_{0}^{+\infty} e^{-\lambda t} t^{-\frac{1}{2s}} \frac{1}{1+ \abs{x}^{2s+1} t^{-\frac{2s+1}{2s}} }\, dt \lesssim \abs{x}^{-1-2s}\int_{0}^{+\infty} e^{-\lambda t} t\, dt \lesssim \abs{x}^{-1-2s}, 
\end{equation}
thus proving \eqref{eq: decay of fund sol} for $k = 0$. To compute the Fourier transform, by Fubini's theorem and the inversion formula in $L^2$, we have that  
\begin{align}
    \mathscr{F}_1( G_{s,\lambda})(\xi) & = \int_0^{+\infty} e^{-\lambda t} \mathscr{F}_1 \left(P_1^{(s)}(t, \cdot)\right)(\xi) \, dt 
    \\ & = \int_{0}^{+\infty} e^{-\lambda t} \exp\left(-   \abs{2\pi \xi}^{2s} t \right) \, dt  = \frac{1}{\lambda + \abs{2 \pi \xi}^{2s} },  
\end{align} 
thus proving \eqref{eq: fourier of fundamental solution}. To prove that $G_{s,\lambda}$ is smooth away from the origin, we check that 
\begin{equation}
    \int_0^{+\infty} e^{-\lambda t} \abs{\partial_x^k P^{(s)}_1(t,x)} \, dt \lesssim \abs{x}^{-k-1-2s} \qquad \forall k\in \N \qquad \forall x \neq 0.  \label{eq: smoothness of fund sol}
\end{equation}
Indeed, by \eqref{eq: scaling of heat kernel} and \cref{p: decay of derivative of P^s} we have  
\begin{align}
    \int_0^{+\infty} e^{-\lambda t} \abs{\partial_x^k P^{(s)}_1(t,x)} \, dt & = \int_0^{+\infty} e^{-\lambda t} t^{-\frac{k+1}{2s}} \abs{ \partial_x^k P^{(s)}_1(1, t^{-\frac{1}{2s}} x) } \, dt 
    \\ & \lesssim \int_0^{+\infty} e^{-\lambda t} t^{-\frac{k+1}{2s}} \frac{1}{1+ t^{-\frac{k+1+2s}{2s}} \abs{x}^{-k-1-2s}} \, dt 
    \\ & \lesssim \abs{x}^{-k-1-2s} \int_0^{+\infty} e^{-\lambda t} t \, dt, 
\end{align}
thus proving \eqref{eq: smoothness of fund sol}. Then, it is easy to check that $G_{s,\lambda}$ is smooth away from the origin and it holds 
\begin{equation}
    \partial^k_x G_{s,\lambda}(x) = \int_0^{+\infty} e^{-\lambda t} \partial_x^k P^{(1)}_s(t,x) \, dt. 
\end{equation}
Thus, \eqref{eq: decay of fund sol} follows by \eqref{eq: smoothness of fund sol}. 
\end{proof}

The proof of \cref{thm:hi-decay} follows from \cref{p: decay of fundamental solution}. 

\begin{proof}[Proof of \cref{thm:hi-decay}] 
We perform the proof by induction. We start with the case $k=2$. We neglect constants $C(k,s, W)>0$. Since $w'$ solves \eqref{eq: fractional Allen--Cahn bis}, iterating the estimates in \cite{S07}*{Proposition 2.8-2.9}, it is readily checked that $w \in W^{2, \infty}(\R)$ with $\norm{w}_{W^{2,\infty}(\R)} \lesssim 1$. Letting $\lambda = W''(\pm 1)$, since $w' \in L^\infty(\R)$ solves
\begin{equation}
    (\lambda + (-\Delta)^s) w' = (\lambda - W''(w)) w',  \label{eq: fractional Allen--Cahn bis} 
\end{equation}
then by \eqref{eq: fourier of fundamental solution} we have 
\begin{equation} \label{eq: integral formula for w'}
    w'(x) = \int_{\R} G_{s,\lambda}(x-z) \Psi(z) w'(z) \, dz,
\end{equation}
where $G_{s,\lambda}$ is given by \eqref{eq: fundamental solution} and we set $\Psi(z) = \lambda - W''(w(z))$. Since $G_{s,\lambda}$ is smooth away from $0$ (see \cref{p: decay of fundamental solution}), for any $x > 1$ we have that 
\begin{align}
w''(x) & = \int_{\abs{x-z} \geq \frac{x}{2}} G_{s,\lambda} (x-z) \partial_z [ \Psi(z) w'(z) ] \, dz + \int_{\abs{x-z} < \frac{x}{2}} G_{s,\lambda} (x-z) \partial_z [ \Psi(z) w'(z) ] \, dz 
\\ & = \int_{|x-z|\geq \frac{x}{2}} \partial_x G_{s,\lambda} (x-z) \Psi(z) w'(z)\,dz -\Bigg[ G_{s,\lambda}(x-z)\Psi(z) w'(z) \Bigg]_{\frac{x}{2}}^{\frac{3x}{2}}
\\ & \qquad +\int_{\frac{x}{2}}^{\frac{3x}{2}} G_{s,\lambda} (x-z) \left[ -W'''(w(z))(w'(z))^2 + (\lambda-W''(w(z)))w''(z) \right] \, dz.
\end{align}
By the decay properties of $G_{s,\lambda}, \partial_x G_{s,\lambda}$ (see \cref{p: decay of fundamental solution})
and $w \in W^{2,\infty}(\R)$, we infer that 
\begin{equation}
    \abs{w''(x)} \lesssim \abs{x}^{-2-2s} + \abs{x}^{-2s} \norm{w''}_{L^\infty([\sfrac{x}{2},\sfrac{3x}{2}])}. 
\end{equation}
Then, letting $h > 1+ \sfrac{1}{s}$ be an integer and iterating $h$ times the above estimate, for $x> 2^{h+1}$ we find that 
\begin{equation}
\abs{w''(x)} \lesssim \abs{x}^{-2-2s} + \abs{x}^{-2s h} \norm{w''}_{L^\infty([2^{-h}x,2^{h}x])} \lesssim \abs{x}^{-2-2s} \left(1 + \norm{w''}_{L^\infty([1, +\infty))}\right),  
\end{equation}
thus proving \eqref{eq: decay of w''} for $x > 2^{h+1}$. The estimate for $x \leq -2^{h+1}$ is analogous, since $w''$ is odd, and the case $x \in [-2^{h+1},2^{h+1}]$ is trivial, since $w'' \in L^\infty(\R)$. Then, \eqref{eq: decay of w''} is proved for $k=2$. 

Fix $k \geq 2$ and assume that $W \in C^{k+2}(\R)$, $w \in C^{k}(\R)$ and \eqref{eq: decay of w''} is proved for any derivative of order smaller than or equal to $k$. We prove that $w\in C^{k+1}(\R)$ and \eqref{eq: decay of w''} holds for $\partial_x^{k+1} w$. Differentiating $k$ times \eqref{eq:f-one-dim-profile}, we find that 
\begin{equation}
    (-\Delta)^s \partial_x^{k} w = -\partial_x^{k} W'(w), 
\end{equation}
where the right-hand side satisfies $\norm{\partial_x^{k} W'(w) }_{L^\infty(\R)} \lesssim 1$. Hence, iterating \cite{S07}*{Proposition 2.8-2.9}, it is readily checked that $\partial_x^k w \in C^{1}(\R)$ and it holds $\norm{\partial_x^{k+1} w}_{L^\infty(\R)} \lesssim 1$. Then, since $G_{s,\lambda}$ is smooth away from the origin (see \cref{p: decay of fundamental solution}), by \eqref{eq: integral formula for w'} and integrating by parts $k$ times, for $x>1$ we find that 
\begin{align}
    \partial_x^{k+1} w(x) & = \int_{\abs{x-z} \geq \frac{x}{2}} G_{s,\lambda}(x-z) \partial_z^k[\Psi(z) w'(z)] \, dz + \int_{\abs{x-z} < \frac{x}{2}} G_{s,\lambda}(x-z) \partial_z^k[\Psi(z) w'(z)] \, dz
    \\ & = \int_{\abs{x-z} \geq \frac{x}{2}} \partial_x^k G_s(x-z) \Psi(z) w'(z)\, dz - \sum_{i=0}^{k-1} \bigg[ \partial_x^{i} G_{s,\lambda}(x-z) \partial^{k-1-i}_z[\Psi(z) w'(z)] \bigg]_{z=\frac{x}{2}}^{z= \frac{3x}{2}} 
    \\ & \qquad + \int_{\abs{x-z} < \frac{x}{2}} G_{s,\lambda}(x-z) \partial_z^k[\Psi(z) w'(z)] \, dz = A+ \sum_{i=0}^{k-1} B_i + C. 
\end{align} 
We estimate separately each term. To begin, by \cref{p: decay of fundamental solution}, we have that 
\begin{equation}
    \abs{A} \lesssim \abs{x}^{-k-1-2s} \norm{\Psi w'}_{L^1(\R)} \lesssim \abs{x}^{-k-1-2s}. 
\end{equation}
For any $j = 0, \dots, k$, by the chain rule and since \eqref{eq: decay of w''} holds up to the order $k$, it is easy to estimate 
\begin{equation} \label{eq: decay of derivative of Psi}
    \abs{\partial_z^j \Psi(z)} \lesssim \abs{z}^{-j-2s}. 
\end{equation}
Then, fix an index $i = 0, \dots, k-1$. By Leibniz rule, we have that 
\begin{align}
    \abs{\partial^i_z [\Psi(z) w'(z)]} & \leq \sum_{j=0}^i \binom{i}{j} \abs{\partial_z^j\Psi(z)} \abs{\partial_z^{i+1-j} w(z)} \lesssim \abs{z}^{-i-1-4s}. \label{eq: decay of Psi w'}
\end{align}
Therefore, by \cref{p: decay of fundamental solution} and \eqref{eq: decay of Psi w'}, we infer that 
\begin{equation}
    \sum_{i=0}^{k-1} \abs{B_i} \lesssim \abs{x}^{-k-1-6s}. 
\end{equation}
Lastly, since $G_{s,\lambda} \in L^1(\R)$, using \eqref{eq: decay of w''} up to order $k$ and by \eqref{eq: decay of derivative of Psi}, we estimate 
\begin{align}
    \abs{C} & \lesssim \sum_{j=1}^{k-1} \norm{(\partial_z^j \Psi) ( \partial_z^{k-j+1} w ) }_{L^\infty([\sfrac{x}{2}, \sfrac{3x}{2}])} + \norm{ \Psi \partial_z^{k+1} w }_{L^\infty([\sfrac{x}{2}, \sfrac{3x}{2}])} 
    \\ & \lesssim \abs{x}^{-k-1-4s} + \abs{x}^{-2s} \norm{ \partial_z^{k+1} w }_{L^\infty([\sfrac{x}{2}, \sfrac{3x}{2}])}. 
\end{align}
To summarize, for $x>1$ we have that 
\begin{equation}
    \abs{\partial_x^{k+1} w(x)} \lesssim \abs{x}^{-k-1-2s} + \abs{x}^{-2s} \norm{ \partial_z^{k+1} w }_{L^\infty([\sfrac{x}{2}, \sfrac{3x}{2}])}. 
\end{equation} 
Then, letting $h > 1+ \sfrac{(k+1)}{2s}$ be an integer and iterating $h$ times the above estimate, for $x> 2^{h+1}$ we find that 
\begin{equation}
\abs{\partial_x^{k+1 }w(x)} \lesssim \abs{x}^{-k-1-2s} + \abs{x}^{-2s h} \norm{\partial_x^{k+1 }w}_{L^\infty([2^{-h}x,2^{h}x])} \lesssim \abs{x}^{-k-1-2s} \left(1 + \norm{\partial_x^{k+1 }w}_{L^\infty([1, +\infty))}\right),  
\end{equation}
thus proving \eqref{eq: decay of w''} for $x > 2^{h+1}$ for the derivative of order $k+1$. The estimate for $x \leq -2^{h+1}$ is analogous, since $\partial_x^{k+1} w$ is odd or even, and the case $x \in [-2^{h+1},2^{h+1}]$ is trivial, since $\partial_x^{k+1} w$ is uniformly bounded. Then, the proof is concluded. 
\end{proof}

\begin{remark} \label{r: assumption thm decay}
Using \cite{S07}*{Proposition 2.8-2.9} as in the proof of \cref{thm:hi-decay}, assuming $W \in C^{k,\alpha}_{\rm loc}$ with $\alpha + 2s >1$ would still suffice to prove that $w \in C^{k, \beta}(\R)$ for some $\beta >0$. However, the main purpose of \cref{thm:hi-decay} is to study the decay rate of the $\partial_x^k w$. Since an integration by part is needed in our argument to prove \eqref{eq: decay of w''}, we have to assume $W \in C^{k+1}$. 
\end{remark}

As a corollary, we obtain the following result. 

\begin{corollary} \label{l: L^infty bound away from the boundary}
Fix $s \in (0,1)$. Let $E$ be a bounded open set of class $C^2$ according to \cref{d:principal coordinates} and let $\delta>0$ be given by \cref{l: regularity of distance function}. Let $u_\e$ be defined by \eqref{recovery sequence}. Then, for any $\Lambda\geq 1$ it holds 
\begin{equation}
    \sup_{\e\in (0,1)} \norm{(-\Delta)^s u_\e}_{L^\infty(\R^d \setminus \Sigma_{\sfrac{\delta}{\Lambda}})} \leq C(d,s,W,\delta, \Lambda).  
\end{equation}
\end{corollary} 

\begin{proof}
We neglect multiplicative constants $C(d,s,\delta,W, \Lambda)>0$. By \cref{l:bound fractional laplacian}, we have
\begin{equation}
    \sup_{\e \in (0,1)} \norm{(-\Delta)^s u_\e}_{L^\infty (\R^d \setminus \Sigma_{\sfrac{\delta}{\Lambda}})} \lesssim \norm{u_\e}_{L^\infty(\R^d)} + \norm{u_\e}_{C^2( \R^d \setminus \Sigma_{\sfrac{\delta}{2\Lambda}} )} . 
\end{equation}
Since $\norm{u_\e}_{L^\infty(\R^d)} = 1$ for any $\e$, it remains to estimate $u_\e$ in $C^2(\R^d \setminus \Sigma_{\sfrac{\delta}{2\Lambda}} )$. Recalling that $\beta_{\Sigma} \in C^2(\R^d)$ and $\abs{\beta_\Sigma(x)} \geq \sfrac{\delta}{2\Lambda}$ for any $x \in \R^d \setminus \Sigma_{\sfrac{\delta}{2\Lambda}}$ (see \cref{d:regular distance}), we need to estimate $w_\e(t) = w\left(\sfrac{t}{\eps}\right)$ in $C^2( \{ \abs{t} > \sfrac{\delta}{2\Lambda} \})$ for $\e \in (0,1)$. Indeed, by \cref{thm:hi-decay} we have 
\begin{equation}
    \abs{w'_\e(t)} + \abs{w''_{\e}(t)} = \abs{\frac{1}{\e} w'\left( \frac{t}{\e} \right) } + \abs{\frac{1}{\e^2} w''\left( \frac{t}{\e}\right) } \lesssim \e^{2s} \qquad \forall \abs{t}\geq \frac{\delta}{2\Lambda} \qquad \forall \e \in (0,1). 
\end{equation}
\end{proof}

\section{Expansion of the fractional Laplacian around the boundary} \label{s:expansion of fractional laplacian}

As we explained in the introduction, the expansion of the fractional Laplacian in Fermi coordinates for the function defined by \eqref{recovery sequence} is crucial for proving our main result. In this section, we provide a complete proof of this expansion. We emphasize that our approach closely follows some computations from \cite{CLW17}*{Section 3}, which deal with a three-dimensional setting. However, we adapt these computations to our framework, carefully keeping track of constants and error terms.

\begin{theorem}\label{t:fractional laplacian}
Let $W$ be a double-well potential satisfying \eqref{h: zero of potential}, \eqref{h: potential is smooth}, \eqref{h: W'' positive}. Let $w: \R \to (-1,1)$ be the one-dimensional optimal profile and for any $\e>0$ let us set $w_\e(z) = w \left( \sfrac{z}{\e}\right)$. Assume that $s \in \left(\sfrac{1}{2},1\right)$. Let $E$ be a bounded open set of class $C^3$ and let $\Sigma = \partial E$. Let $\delta>0$ be such that $\beta_\Sigma$ is well defined according to \cref{d:regular distance} and let us set   
\begin{equation} \label{recovery sequence}
     u_\e(x) = w_\e (\beta_\Sigma(x)).  
\end{equation}
There exists $\Lambda_0, C \geq 1$ depending on $\Sigma$ with the following property. For any $\Lambda \geq \Lambda_0$ for any $\e \in (0,1)$ there exists $\mathcal{R}_{\e,\Lambda}: \Sigma_{\sfrac{\delta}{10 \Lambda}} \to \R$ such that for any $x_0 \in \Sigma_{\sfrac{\delta}{10 \Lambda}}$ it holds 
\begin{align}
        (-\Delta)^s u_\e(x_0) & = (-\partial_{zz})^s w_\e(z_0) + \frac{\gamma_{1,s}}{2} \frac{H_{\Sigma}(x_0')}{(2s-1)} \int_{-\sfrac{\delta}{\Lambda}}^{\sfrac{\delta}{\Lambda}} \frac{w_\e'(z_0+\bar{z})}{\abs{\bar{z}}^{2s-1}} \, d\bar{z} + \mathcal{R}_{\e, \Lambda}(x_0), \label{eq:expansion} 
\end{align} 
where we set $z_0 = \dist_{\Sigma}(x_0)$ and $x_0' = \pi_{\Sigma} (x_0)$. Moreover, it holds that 
\begin{equation} \label{eq: bound reminder}
    \norm{\mathcal{R}_{\e,\Lambda}}_{L^\infty(\Sigma_{\sfrac{\delta}{10 \Lambda}})} \leq C \Lambda^{2s}.
\end{equation}
\end{theorem}

\begin{proof} For the reader convenience, we split the proof into several steps. To begin, we summarize the strategy adopted. Let $\Lambda_0$ be the geometric constant given by \cref{l:inner ball} and fix $\Lambda \geq \Lambda_0$. Then, we take $x_0 \in \Sigma_{\sfrac{\delta}{10 \Lambda}}$ and we denote by $x_0' = \pi_{\Sigma}(x_0), z_0 = \dist_\Sigma(x_0)$. Then, with the same notation as \cref{l:inner ball} we take Fermi coordinates around $x_0'$, i.e. $\Phi: B_\delta^{d-1} \times (-\delta, \delta) \to \R^d$. We recall that $x_0 = \Phi(0, z_0)$ with respect to this coordinate system. For convenience, we set $\mathcal{T}_{\sfrac{\delta}{\Lambda}}(x_0) : = \Phi (\mathcal{B}_{\sfrac{\delta}{\Lambda}}(z_0))$, where $\mathcal{B}_{\sfrac{\delta}{\Lambda}}(z_0) \subset B_\delta^{d-1} \times (-\delta, \delta)$ is defined by~\eqref{eq: domain B}. Since the fractional Laplacian is a singular integral (see \eqref{eq: fractional laplacian}) and $\mathcal{T}_{\sfrac{\delta}{\Lambda}}(x_0)$ is an open set around $x_0$ (see \cref{l:inner ball}), we analyse separately the contribution in $\T_{\sfrac{\delta}{\Lambda}}(x_0)$ and in $\R^d \setminus \T_{\sfrac{\delta}{\Lambda}}(x_0)$. More precisely, we have  
\begin{align}
    (-\Delta)^s u_\e(x_0) = \gamma_{d,s} \int_{\R^d \setminus \T_{\sfrac{\delta}{\Lambda}}(x_0)} \frac{u_\e(x_0) - u_\e(x)}{\abs{x-x_0}^{d+2s}} \, dx + \lim_{\nu \to 0} \gamma_{d,s} \int_{\T_{\sfrac{\delta}{\Lambda}}(x_0) \setminus B_\nu(x_0)} \frac{u_\e(x_0) - u_\e(x)}{\abs{x-x_0}^{d+2s}} \, dx. \label{eq: splitting of fractional laplacian}
\end{align}
Unless otherwise specified, the reminders involved in the following computations satisfy bounds depending on $\Sigma$. In particular, they are independent on the choice of the point $x_0$. This fact follows essentially by \cref{r: bounded sets are C-k uniform}, \cref{l: taylor expansion 1} and \cref{l:inner ball}. 
  
\textsc{\underline{Step 1}: Estimating the outer contribution.} Up to an implicit constant depending only on $d, \delta, \Sigma$, by \cref{l:inner ball} we estimate the integral in $\R^d \setminus \T_{\sfrac{\delta}{\Lambda}}(x_0)$ as follows
\begin{align} 
    \abs{ \int_{\R^d \setminus \T_{\sfrac{\delta}{\Lambda}}(x_0)}  \frac{u_\e(x_0)-u_\e(x)}{\abs{x-x_0}^{d+2s}} \, dx } &  \lesssim \int_{\R^d \setminus B_{\sfrac{\delta}{10 \Lambda}}(x_0)  } \frac{1}{\abs{x-x_0}^{d+2s}} \, dx \lesssim \Lambda^{2s}.  \label{eq: the outer contribution}
\end{align}

\textsc{\underline{Step 2}: Rewriting the inner contribution.} The estimate of the singular term is much more delicate and it requires a careful analysis. To ease the notation, given $\nu>0$, we set 
\begin{equation} \label{eq: Delta inner 1}
    \Delta^s_{\nu} u_\e(x_0) := \int_{\T_{\sfrac{\delta}{\Lambda}} (x_0) \setminus B_\nu(x_0)} \frac{u_\e(x) - u_\e(x_0)}{\abs{x-x_0}^{d+2s}} \, dx. 
\end{equation}
We changed sign to avoid many negative terms in the computations. We aim to write the integral in \eqref{eq: Delta inner 1} as an integral with respect to the variables $z,y$. Furthermore, writing the Euclidean distance as in \cref{r: another tangential distance}, by definition of $\T_{\sfrac{\delta}{\Lambda} } (x_0)$ and recalling that $\Phi$ is a diffeomorphism by \cref{l:inner ball}, we have 
$$\T_{\sfrac{\delta}{\Lambda}}(x_0) \setminus B_{\nu}(x_0) = \Phi(\mathcal{U}^1_{\nu}), $$
where we set 
\begin{equation} \label{eq: U^1}
    \mathcal{U}^1_{\nu} := \left\{ (y,z) \in \mathcal{B}_{\sfrac{\delta}{\Lambda}} (z_0) \colon    
    \begin{aligned}
        \abs{z-z_0 + Y(y)\cdot N(y) + z_0(1-N_d(y))}^2 + \abs{(Y(y) - z_0 e_d)_\tau}^2 \geq \nu^2
\end{aligned} \right\} . 
\end{equation}
Here $\mathcal{B}_{\sfrac{\delta}{\Lambda}}(z_0)$ is defined by \eqref{eq: domain B}. Since $\beta_\Sigma$ is the proper distance from the boundary (see \cref{d:regular distance} and \cref{l:inner ball}) in $\Phi(\mathcal{U}^1_{\nu})$, changing variables in \eqref{eq: Delta inner 1} and recalling that $u_\e$ is given by \eqref{recovery sequence}, we have that  
\begin{align}
    \Delta^s_{\nu} u_\e(x_0) = \int_{\mathcal{U}^1_{\nu}} \frac{(w_\e(z)-w_\e(z_0)) \abs{\det(\nabla \Phi(y,z))} }{\left(\abs{z-z_0 + Y(y)\cdot N(y) + z_0(1-N_d(y))}^2 + \abs{(Y(y) - z_0 e_d)_\tau}^2\right)^{\frac{d+2s}{2}}}  \, dz \, dy. 
\end{align}
Next, with the notation of \cref{l:inner ball}, for any $y\in \mathscr{B}_{\sfrac{\delta}{\Lambda}}^{d-1}(z_0)$, we set 
\begin{equation}
    \bar{z}(y,z) = z-z_0 + Y(y)\cdot N(y) + z_0 (1- N_d(y)). \label{eq: overline z}
\end{equation}
Therefore, changing again variables, we have that 
\begin{equation}
    \Delta^s_{\nu} u_\e (x_0) = \int_{\mathcal{U}^2_{\nu}} \frac{(w_\e(z(y, \bar{z}))-w_\e(z_0)) \abs{\det(\nabla \Phi(y,z(y, \bar{z})))} }{\left(\abs{\bar{z}}^2 + \abs{(Y(y) - z_0 e_d)_\tau}^2\right)^{\frac{d+2s}{2}}}  \, d\bar{z} \, dy,  
\end{equation}
where we set 
\begin{equation} \label{eq: U^2}
    \mathcal{U}^2_{\nu} := \left\{ (y,\bar{z}) \in \mathscr{B}_{\sfrac{\delta}{\Lambda}}^{d-1}(z_0) \times \left(-\sfrac{\delta}{\Lambda}, \sfrac{\delta}{\Lambda}\right) \colon\abs{\bar{z}}^2 + \abs{(Y(y) - z_0 e_d)_\tau}^2 \geq \nu^2 \right\} . 
\end{equation}
Hence, setting 
\begin{equation}
    \bar{y}(y) = (\mathrm{Id} - \nabla^2 g(0) z_0) y,  \label{eq: overline y}
\end{equation}
and changing again variables, we have that 
\begin{align}
    \Delta^s_{\nu} u_\e(x_0) = \int_{\mathcal{U}^3_{\nu}} \frac{(w_\e(z(y(\bar{y}), \bar{z}))-w_\e(z_0))  }{\left(\abs{\bar{z}}^2 + \abs{(Y(y(\bar{y})) - z_0 e_d)_\tau}^2\right)^{\frac{d+2s}{2}}} \frac{\abs{\det(\nabla \Phi(y(\bar{y}),z(y(\bar{y}), \bar{z})))}}{\abs{\det (\mathrm{Id} - z_0 \nabla^2 g (0))}}  \, d\bar{z} \, d\bar{y}, \label{eq: integral in U^3}
\end{align}
where we set 
\begin{equation} \label{eq: U^3}
    \mathcal{U}^3_{\nu} := \left\{ (\bar{y},\bar{z}) \in B_{\sfrac{\delta}{\Lambda}}^{d-1} \times \left(-\sfrac{\delta}{\Lambda}, \sfrac{\delta}{\Lambda} \right) \colon     \abs{\bar{z}}^2 + \abs{(Y(y(\bar{y})) - z_0 e_d)_\tau}^2 \geq \nu^2 \right\}. 
\end{equation}

\textsc{\underline{Step 3}: Computing the leading order terms.} In order to estimate the integral in \eqref{eq: integral in U^3}, we compute the leading orders of the terms involved. By \cref{d: tangential distance}, \eqref{eq: overline y} and \eqref{eq: tangential distance 1}, we have 
\begin{equation}
    \abs{(Y(y(\bar{y})) - z_0 e_d)_\tau}^2 = \abs{\bar{y}}^2 + O(\abs{z_0} \abs{\bar{y}}^3) + O(\abs{\bar{y}}^4). \label{eq: tangential distance 2}
\end{equation}
By \cref{l: taylor expansion 1} and \eqref{eq: overline y}, it is clear that 
\begin{equation}
    Y(y(\bar{y}))\cdot N(y(\bar{y})) = -\frac{1}{2} \sum_{i=1}^{d-1} \bar{y}_i^2 k_i  +  O( \abs{z_0} \abs{\bar{y}}^2 ),  \label{eq: almost orthogonality 2}
\end{equation}
\begin{equation}
    1- N_d(y(\bar{y})) = O(\abs{\bar{y}}^2), \label{eq: normal deviation 2}
\end{equation}
Thus, by \eqref{eq: overline z}, \eqref{eq: overline y}, \eqref{eq: almost orthogonality 2} and \eqref{eq: normal deviation 2}, we infer that
\begin{equation}
    z(y, \bar{z}) = z_0 + \bar{z} + \frac{1}{2} \sum_{i=1}^{d-1} y_i^2 k_i  + O(\abs{y}^3) +  O( \abs{z_0} \abs{y}^2).  \label{eq: shifted 1}
\end{equation} 
\begin{equation}
    z(y(\bar{y}), \bar{z}) = z_0 + \bar{z} + f(z_0, \bar{y}), \label{eq: shifted 2}
\end{equation}
where we set 
\begin{equation}
    f(z_0, \bar{y}) = \sum_{i=1}^{d-1} \frac{1}{2} k_i  \bar{y}_i^2 +  O( \abs{z_0} \abs{\bar{y}}^2) + O( \abs{ \bar{y}}^3). \label{eq: formula f}  
\end{equation}
By \cref{l: taylor expansion 1}, \eqref{eq: shifted 1} and \eqref{eq: overline z}, we have that 
\begin{align}
    \det \nabla \Phi(y, z(y, \bar{z})) & = \prod_{i=1}^{d-1} (1- (\bar{z}+z_0 + O(\abs{y}^2)) k_i  ) + O((\abs{\bar{z}} + \abs{z_0} + \abs{y}^2 ) \abs{y} ) + O(\abs{y}^2) 
    \\ & = \prod_{i=1}^{d-1} (1- k_i (z_0 + \bar{z})) + O(\abs{z_0} \abs{y}) + O(\abs{\bar{z}} \abs{y})  + O (\abs{y}^2) \label{eq: jacobian 2}.
\end{align}
By \eqref{eq: overline y} and \eqref{eq: jacobian 2}, we have that 
\begin{align}
    \frac{\det \nabla \Phi(y(\bar{y}), z(y(\bar{y}) , \bar{z}))}{\det(\mathrm{Id} - z_0 \nabla^2 g (0))} & = \left( \prod_{i=1}^{d-1} (1- z_0 k_i  - \bar{z} k_i ) + O((\abs{z_0} + \abs{\bar{z}})\abs{\bar{y}} + \abs{\bar{y}}^2) \right) \left( \prod_{i=1}^{d-1} (1- z_0 k_i )\right)^{-1} 
    \\ & = \prod_{i=1}^{d-1} \left( 1- \bar{z} \frac{k_i }{ 1- z_0 k_i } \right) + O((\abs{z_0} + \abs{\bar{z}})\abs{\bar{y}} + \abs{\bar{y}}^2) 
    \\ & = 1- \bar{z} H_{\Sigma}(x_0') + O(\abs{z_0} \abs{\bar{y}}) + O(\abs{\bar{y}}^2) + O ( \abs{\bar{z}}^2). \label{eq: jacobian 3} 
\end{align}
From now on, we denote by 
$$\rho^2 = \abs{\bar{z}}^2 + \abs{\bar{y}}^2. $$
Hence,  by \eqref{eq: tangential distance 2}, \eqref{eq: shifted 2}, \eqref{eq: formula f}, \eqref{eq: jacobian 3}, 
we write \eqref{eq: integral in U^3} as follows 
\begin{align}
    \Delta^s_{\nu} u_\e(x_0)  = \int_{\mathcal{U}^3_{\nu}} \frac{\left[w_\e( z_0 + \bar{z} + f(z_0, \bar{y}))-w_\e(z_0)\right]  \left[ 1- \bar{z} H_{\Sigma}(x_0') + O(\abs{z_0} \abs{\bar{y}}) + O( \rho^2)  \right] }{\left(\rho^2 + O(\abs{z_0} \abs{\bar{y}}^3) + O (\abs{\bar{y}}^4) ) \right)^{\frac{d+2s}{2}}} \, d\bar{z} \, d\bar{y}. \label{eq: integral U^3_nu bis}
\end{align}
By standard manipulations, we write 
\begin{align}
    \frac{1- \bar{z} H_\Sigma(x_0') + O(\abs{z_0} \abs{\bar{y}}) + O(\rho^2) }{\left(\rho^2 + O(\abs{z_0} \abs{\bar{y}}^3) + O (\abs{\bar{y}}^4) ) \right)^{\frac{d+2s}{2}}} & = \frac{1- \bar{z} H_\Sigma(x_0') + O(\abs{z_0} \abs{\bar{y} }) + O(\rho^2) }{ \rho^{d+2s} } \left( 1 + O(\abs{z_0} \abs{\bar{y}}) + O(\abs{\bar{y}}^2) \right)
    \\ & = \frac{1- \bar{z} H_\Sigma(x_0') + O(\abs{z_0} \abs{\bar{y} }) + O(\rho^2) }{ \rho^{d+2s} }. 
\end{align}
Hence, by \eqref{eq: integral U^3_nu bis} we write 
\begin{equation}
    \Delta^s_{\nu} u_\e(x_0)  = \int_{\mathcal{U}^3_{\nu}} \frac{ g_\e(\bar{y}, \bar{z}, z_0)  }{\rho^{d+2s}} \, d\bar{z} \, d\bar{y}, \label{eq: integral U^3_nu tris}
\end{equation}
where we denote by 
\begin{equation} \label{eq: formula g_e}
    g_\e(\bar{y}, \bar{z}, z_0) := \left[w_\e( z_0 + \bar{z} + f(z_0, \bar{y}))-w_\e(z_0)\right]  \left[ 1- \bar{z} H_{\Sigma}(x_0') + O(\abs{z_0} \abs{\bar{y}}) + O(\rho^2) \right] 
\end{equation}
We claim that 
\begin{equation}
    \lim_{\nu \to 0} \int_{\mathcal{U}^3_{\nu}} \frac{g_\e(\bar{y}, \bar{z}, z_0)}{\rho^{d+2s}} \, d\bar{z} \, d\bar{y} = \lim_{\nu \to 0}  \int_{\mathcal{C}_{\nu}} \frac{g_\e(\bar{y}, \bar{z}, z_0)}{\rho^{d+2s}} \, d\bar{z} \, d\bar{y}, \label{eq: integral in C_nu}
\end{equation}
where $\mathcal{C}_\nu$ is the complement of a ball in a cylinder
\begin{equation} \label{eq: C_nu}
    \mathcal{C}_{\nu} := \left\{ (\bar{y},\bar{z}) \in B_{\sfrac{\delta}{\Lambda}}^{d-1} \times \left(-\sfrac{\delta}{\Lambda}, \sfrac{\delta}{\Lambda}\right) \colon \rho \geq \nu \right\}.  
\end{equation}
To begin, we estimate the symmetric difference between $\mathcal{U}_\nu^3$ and $\mathcal{C}_\nu$. By \eqref{eq: U^3}, \eqref{eq: tangential distance 2} and \eqref{eq: C_nu}, if $(\bar{y}, \bar{z}) \in \mathcal{U}_\nu^3 \setminus \mathcal{C}_\nu$ we have $ \nu^2 - O(\abs{\bar{y}}^3) \leq  \rho^2 \leq \nu^2$. Since $\abs{\bar{y}} \leq \rho \leq \nu$, we find a purely geometric constant $\bar{c}>0$ such that $\rho^2 \in ( \nu^2- \bar{c} \nu^{3}, \nu^2 )$. Similarly, if $(\bar{y}, \bar{z}) \in \mathcal{C}_\nu \setminus \mathcal{U}_\nu^3 $, we have that $\rho^2 \in ( \nu^2, \nu^2 + \bar{c} \nu^{3})$. To summarize, we have that 
$$ \mathcal{U}^3_{\nu} \Delta \mathcal{C}_{\nu} \subset \{ \nu - \bar{c} \nu^2 \leq \rho \leq \nu + \bar{c} \nu^2 \}. $$ 
Since $w_\e$ is Lipschitz, $f$ satisfies \eqref{eq: formula f} and $s \in (0,1)$, we have 
\begin{align}
    \lim_{\nu \to 0} \abs{ \int_{\mathcal{U}^3_{\nu}} \frac{g_\e(\bar{y}, \bar{z}, z_0) }{\rho^{d+2s}} \, d \bar{y}\, d\bar{z} - \int_{\mathcal{C}_\nu} \frac{g_\e(\bar{y}, \bar{z}, z_0) }{\rho^{d+2s}} \, d \bar{y}\, d\bar{z}  } & \lesssim \lim_{\nu \to 0} \int_{\nu - \bar{c} \nu^2}^{\nu + \bar{c} \nu^2} \frac{\rho + \rho^2 }{\rho^{d+2s}} \rho^{d-1} \, d \rho = 0. \label{eq: U_3 minus C_nu 1}
\end{align}
Here the implicit constant depends on $\e$, but it is independent of $\nu$. Hence, \eqref{eq: integral in C_nu} is proved. 

\textsc{\underline{Step 4}: Collecting the estimates.} To summarize, by \eqref{eq: integral in C_nu} and letting $f,g_\e$ be as in \eqref{eq: formula f}, \eqref{eq: formula g_e}, we have that 
\begin{align}
    \lim_{\nu \to 0} \Delta_\nu^s u_\e(x_0) = & \lim_{\nu \to 0} \int_{\mathcal{C}_\nu} \frac{w_\e(z_0 + \bar{z}) - w_\e(z_0) }{\rho^{d+2s}} (1- \bar{z} H_{\Sigma}(x_0')) \, d \bar{z} \, d \bar{y}
    \\ & + \int_{\mathcal{C}_\nu} \frac{w_\e'(z_0 + \bar{z}) f(z_0, \bar{y})}{\rho^{d+2s}} (1- \bar{z} H_{\Sigma}(x_0')) \, d\bar{z}\, d \bar{y}
    \\ & + \int_{\mathcal{C}_\nu} \frac{w_\e(z_0 + \bar{z} + f(z_0, \bar{y}) ) - w_\e(z_0+ \bar{z}) - w_\e'(z_0+\bar{z}) f(z_0, \bar{y}) }{\rho^{d+2s}} (1- \bar{z} H_{\Sigma}(x_0'))\, d\bar{z}\, d \bar{y} 
    \\ & + O \left( \int_{\mathcal{C}_\nu} \frac{\abs{w_\e(z_0+ \bar{z} + f(z_0, \bar{y}) )- w_\e(z_0) }}{\rho^{d+2s}} ( \abs{z_0} \abs{\bar{y}}  + \rho^2 ) \, d\bar{z} \, d \bar{y} \right). 
    \\ & = \lim_{\nu \to 0} I_1^\nu + I_2^\nu + I_3^\nu + O(I_4^\nu). \label{eq: expansion of inner contribution}
\end{align}
To conclude the proof, we estimate separately the four terms. For the reader's convenience, we postpone these computations to \cref{ss: estimates of the four integrals}. We summarize the contribution of each term: 
\begin{itemize}
    \item [(i)] $I_1^\nu$ gives the fractional Laplacian of $w$ and a contribution of the principal curvatures;
    \item [(ii)] $I_2^\nu$ gives the remaining part of the term involving the principal curvatures;
    \item [(iii)] $I_3^\nu$ is the nonlinear error;  
    \item [(iv)] $I_4^\nu$ is the error when expanding the ambient distance and the Jacobian. 
\end{itemize}
To conclude, by \eqref{eq: the outer contribution}, \eqref{eq: expansion of inner contribution}, \cref{l: expansion for I_1}, \cref{l: expansion of I_2}, \cref{l: expansion of I_3}, \cref{l: expansion of I_4}, we infer that
\begin{equation}
    \lim_{\nu \to 0} \gamma_{d,s} \Delta^s_\nu u_\e(x_0) = - (-\Delta)^s w_\e (z_0) - \frac{\gamma_{1,s}}{2} \frac{H_{\Sigma} (x_0')}{(2s-1)} \int_{-\sfrac{\delta}{\Lambda}}^{\sfrac{\delta}{\Lambda}} \frac{w_\e'(z_0+\bar{z})}{\abs{\bar{z}}^{2s-1}} \, d\bar{z} + \mathcal{R}_{\Lambda, \e}(x_0), 
\end{equation}
where $\mathcal{R}_{\Lambda,\e}: \Sigma_{\sfrac{\delta}{10 \Lambda}} \to \R$ is a bounded function satisfying \eqref{eq: bound reminder}. Then, the proof is concluded.
\end{proof}

\subsection{Estimates of the four integrals} \label{ss: estimates of the four integrals}

We estimate the terms $I_1^\nu, I_2^\nu, I_3^\nu, I_4^\nu$ in \eqref{eq: expansion of inner contribution}. Throughout this section, we implicitly assume $\e \in (0,1)$ and $\Lambda \geq \Lambda_0$, where $\Lambda_0$ is given by \cref{l:inner ball}. We neglect constants $C(d,s,W,\delta, \Sigma)>0$, whereas it is crucial to keep the dependence of $\e, \Lambda$ explicit We recall that $w$ is the optimal profile and $w_\e(z) = w \left( \sfrac{z}{\e}\right) $. We need a preliminary lemma. 

\begin{lemma} \label{l:formula eta}
Fix $s \in (\sfrac{1}{2},1)$. For any $z_0 \in \R, \ell, \e>0$ it holds
\begin{equation} \label{eq:formula eta 2}
\int_{-\ell}^\ell \frac{w_\e(z_0+z)-w_\e(z_0)}{\abs{z}^{2s+1}} z \, dz = \frac{\ell^{1-2s}}{1-2s} (w_\e(\ell+z_0) - w_\e(z_0-\ell)) + \frac{1}{2s-1} \int_{-\ell}^\ell \frac{w_\e'(z_0+z)}{\abs{z}^{2s-1}} \, dz. 
\end{equation}

\end{lemma}

\begin{proof}
We have that 
$$ \int_{-\ell}^\ell \frac{w_\e(z_0+z) - w_\e(z_0)}{\abs{z}^{2s+1}} z \, dz = \left( \int_0^\ell + \int_{-\ell}^0 \right) \frac{w_\e(z_0+z) - w_\e(z_0)}{\abs{z}^{2s+1}} z \, dz = I+II. $$
Then, carefully integrating by parts $I$, we get that 
\begin{align}
    I & = \lim_{\nu \to 0} \int_{\nu}^\ell \frac{w_\e(z_0+z)-w_\e(z_0)}{z^{2s}} \, dz 
    \\ & = \lim_{\nu \to 0} \left[ \frac{z^{1-2s}}{1-2s} (w_\e (z_0+z)-w_\e(z_0))  \right]_{z=\nu}^{z=\ell} + \frac{1}{2s-1} \int_{\nu}^\ell \frac{w_\e'(z_0+z)}{z^{2s-1}} \, dz
    \\ & = \lim_{\nu \to 0} \frac{\ell^{1-2s}}{1-2s} (w_\e(\ell+z_0)-w_\e(z_0)) - \frac{\nu^{1-2s}}{1-2s} (w_\e(\nu +z_0) - w_\e(z_0)) + \frac{1}{2s-1} \int_{\nu}^\ell \frac{w_\e'(z_0+z)}{z^{2s-1}} \, dz.
\end{align}
Since $w_\e$ is Lipschitz continuous and $s \in (0,1)$, we have that 
$$\lim_{\nu\to 0} \abs{\frac{\nu^{1-2s }}{1-2s} (w_\e(\nu +z_0) - w_\e(z_0))} \lesssim \lim_{\nu \to 0} \nu^{2-2s} = 0.  $$
Recalling that $w_\e$ is increasing, by monotone convergence, we have that 
$$ I  -\frac{\ell^{1-2s}}{1-2s} (w_\e(\ell+z_0)-w_\e(z_0))  = \lim_{\nu \to 0} \frac{1}{2s-1} \int_{\nu}^\ell \frac{w_\e'(z_0+z)}{z^{2s-1}}\, dz = \frac{1}{2s-1} \int_0^{\ell} \frac{w_\e'(z_0+z)}{z^{2s-1}} \, dz. $$
Similarly, it is easy to check that 
$$ II =  \frac{\ell^{1-2s}}{1-2s} (w_\e(z_0) - w_\e(z_0-\ell)) + \frac{1}{2s-1} \int_{-\ell}^0 \frac{w_\e'(z_0+z)}{\abs{z}^{2s-1}} \, dz. $$
\end{proof}

\begin{lemma} \label{l: expansion for I_1}
Let $I_1^\nu$ be given by \eqref{eq: expansion of inner contribution}. It holds 
\begin{align}
    \sup_{x_0 \in \Sigma_{\sfrac{\delta}{10 \Lambda}}} \abs{ \lim_{\nu \to 0} I_1^\nu + \frac{1}{\gamma_{d,s}}(-\Delta)^s w_\e(z_0) + \frac{\gamma_{1,s}}{\gamma_{d,s}} \cdot \frac{H_{\Sigma}(x_0') }{(2s-1)} \int_{-\sfrac{\delta}{\Lambda}}^{\sfrac{\delta}{\Lambda}} \frac{w_\e'(z_0+\bar{z})}{\abs{\bar{z}}^{2s-1}} \, d\bar{z} } \lesssim \Lambda^{2s}.  
\end{align}
\end{lemma}

\begin{proof} 
For simplicity, we denote by $\ell = \sfrac{\delta}{\Lambda}$. Using the second order difference to get rid of the principal value, we split 
\begin{align}
    \lim_{\nu \to 0} I_1^\nu & = \int_{B_\ell^{d-1}} \int_{-\ell}^\ell \frac{w_\e(z_0 + \bar{z}) + w_\e(z_0-\bar{z}) -2 w_\e(z_0) }{2 \rho^{d+2s}} \, d \bar{z} \, d \bar{y}  
    \\ & \qquad - H_{\Sigma}(x_0') \int_{B_\ell^{d-1}} \int_{-\ell}^\ell \frac{w_\e(z_0 + \bar{z}) - w_\e(z_0)}{\rho^{d+2s}} \bar{z} \, d \bar{z} \, d \bar{y} 
    \\ & = I_{1,1} + I_{1,2}.
\end{align}
\textsc{Estimate of $I_{1,1}$}. By \cref{l: reduction of the kernel} and recalling that $\abs{w_\e} \leq 1$, we have  
\begin{align}
    I_{1,1} & = \frac{\gamma_{1,s}}{\gamma_{d,s}} \int_{-\ell}^{\ell}  \frac{w_\e(z_0 +\bar{z}) + w_\e(z_0 - \bar{z}) - 2 w_\e(z_0)}{2 \abs{\bar{z}}^{1+2s}} \, d\bar{z} + O\left(\ell^{-2s}\right) = - \frac{1}{\gamma_{d,s}} (-\Delta)^s w_\e(z_0) + O\left(\ell^{-2s}\right).  
\end{align}

\textsc{Estimate of $I_{1,2}$}. Similarly, we have  
\begin{align}
    I_{1,2} & = -H_{\Sigma}(x_0') \int_{-\ell}^{\ell} \left( \frac{\gamma_{1,s}}{\gamma_{d,s}} \frac{w_\e(z_0 + \bar{z})- w_\e(z_0)}{\abs{\bar{z}}^{1+2s}} + O\left( \ell^{-1-2s} \right) \right) \bar{z} \, d \bar{z}  
    \\ & = - H_{\Sigma} (x_0') \frac{\gamma_{1,s}}{\gamma_{d,s}} \int_{-\ell}^{\ell}  \frac{w_\e(z_0 + \bar{z})- w_\e(z_0)}{\abs{\bar{z}}^{1+2s}} \bar{z} \, d \bar{z} + O\left( \ell^{1-2s} \right). 
\end{align}
Then, using \eqref{eq:formula eta 2} and recalling that $\abs{w_\e} \leq 1$, we conclude  
$$\abs{ I_{1,2} +  H_{\Sigma}(x_0') \frac{\gamma_{1,s}}{(2s-1) \gamma_{d,s}} \int_{-\ell}^\ell \frac{w_\e'(z_0+\bar{z})}{\abs{\bar{z}}^{2s-1}} \, d\bar{z}} \lesssim \ell^{1-2s}. $$
\end{proof}

\begin{lemma} \label{l: expansion of I_2}
Let $I_2^\nu$ be defined by \eqref{eq: expansion of inner contribution}. It holds  
\begin{equation}
    \sup_{x_0 \in \Sigma_{\sfrac{\delta}{10 \Lambda}}} \abs{ \lim_{\nu \to 0} I_2^\nu - \frac{H_{\Sigma}(x_0') }{2 (2s-1) } \cdot \frac{\gamma_{1,s}}{\gamma_{d,s}} \int_{-\sfrac{\delta}{\Lambda}}^{\sfrac{\delta}{\Lambda}} \frac{w'_\e(z_0+\bar{z})}{\abs{\bar{z}}^{2s-1}} \, d \bar{z} } \lesssim \Lambda^{2s-1}.   
\end{equation}
\end{lemma}

\begin{proof}
For simplicity, we denote by $\ell = \sfrac{\delta}{\Lambda}$. Then, we split 
\begin{align}
    \lim_{\nu \to 0} I_2^\nu & = \int_{B_\ell^{d-1}} \int_{-\ell}^\ell \frac{w_\e'(z_0+ \bar{z}) \sum_{i=1}^{d-1} k_i  \bar{y}_i^2 }{2 \rho^{d+2s}} (1- \bar{z} H_{\Sigma}(x_0')) \, d \bar{z}\, d \bar{y} 
    \\ & \quad + O \left( \abs{z_0} \int_{B_\ell^{d-1}} \int_{-\ell}^\ell \frac{w'_\e(z_0 + \bar{z})}{\rho^{d+2s}} \abs{\bar{y}}^2 \, d \bar{z} \, d \bar{y} \right) + O \left( \int_{B_\ell^{d-1}} \int_{-\ell}^\ell \frac{w'_\e(z_0 + \bar{z})}{\rho^{d+2s}} \abs{\bar{y}}^3 \, d \bar{z} \, d \bar{y} \right)
    \\ & = I_{2,1} + O\left(I_{2,2}\right) + O\left(I_{2,3}\right) 
\end{align}
and we estimate separately each term. 

\textsc{Estimate of $I_{2,1}$}. By \cref{l: reduction of the kernel} and the fundamental theorem of calculus, we have that  
\begin{align}
    I_{2,1} & = \frac{H_{\Sigma}(x_0')}{2(2s-1)} \int_{-\ell}^\ell w_\e'(z_0 + \bar{z})(1- \bar{z} H_{\Sigma}(x_0')) \left( \frac{\gamma_{1,s}}{\gamma_{d,s}} \cdot \frac{1}{\abs{\bar{z}}^{2s-1}} + O\left(\ell^{1-2s}\right) \right) \, d\bar{z}  
    \\ & = \frac{H_{\Sigma}(x_0') }{2 (2s-1) }\cdot \frac{\gamma_{1,s}}{\gamma_{d,s}} \int_{-\ell}^\ell \frac{w'_\e(z_0+\bar{z})}{\abs{\bar{z}}^{2s-1}} \, d \bar{z} + O\left(\ell^{1-2s}\right).
\end{align}

\textsc{Estimate of $I_{2,2}$}. By \cref{l: reduction of the kernel} we compute 
\begin{align}
    I_{2,2} & = \abs{z_0} \int_{-\ell}^{\ell} w_\e'(z_0 + \bar{z}) \int_{B_\ell^{d-1}} \frac{\abs{\bar{y}}^2}{ (\abs{\bar{z}}^2 + \abs{\bar{y}}^2)^{\frac{d+2s}{2}}}\, d \bar{y} \, d \bar{z} 
    \\ & \lesssim \abs{z_0} \int_{-\ell}^{\ell} \frac{ w_\e'(z_0 + \bar{z}) }{\abs{\bar{z}}^{2s-1}} \, d\bar{z} + \abs{z_0} \ell^{1-2s} \int_{-\ell}^{\ell} w_\e'(z_0 + \bar{z})\, d \bar{z} = A + B.
\end{align}
By the fundamental theorem of calculus, we compute 
$$B \lesssim \abs{z_0} \ell^{1-2s} \int_{\R} w_\e'(z_0 + \bar{z})\, d \bar{z} \lesssim \ell^{2-2s}. $$
To estimate $A$, we split 
\begin{equation}
    A = \abs{z_0} \left( \int_{\abs{\bar{z}} \leq \sfrac{\abs{z_0}}{2}} + \int_{\sfrac{\abs{z_0}}{2} \leq \abs{\bar{z}} \leq \ell } \right) \frac{w_\e'(z_0+ \bar{z})}{\abs{\bar{z}}^{2s-1}}\, d\bar{z} = A_1 + A_2. 
\end{equation}
Then, by \eqref{eq: decay w'} we have 
\begin{align}
    A_1 & \lesssim \abs{z_0} \sup_{\abs{\zeta} \geq \sfrac{\abs{z_0}}{2} } w_\e'(\zeta) \int_{\abs{\bar{z}} \leq \sfrac{\abs{z_0}}{2}} \frac{1}{\abs{\bar{z}}^{2s-1}} \, d \bar{z}  \lesssim \frac{\abs{z_0}}{\e} \frac{1}{1+ \abs{z_0}^{1+2s} \e^{-1-2s}} \abs{z_0}^{2-2s} \lesssim \ell^{2-2s},
\end{align}
since the function $t \mapsto \frac{t}{1+ \abs{t}^{1+2s}}$ is bounded in $\R$. To estimate $A_2$, we have 
\begin{equation}
    A_2 \lesssim \abs{z_0}^{2-2s} \int_{-\ell}^{\ell} w_\e'(z_0 + \bar{z})\, d \bar{z} \lesssim \ell^{2-2s}. 
\end{equation}

\textsc{Estimate of $I_{2,3}$}. Since $d+2s-3 < d-1$, we estimate 
\begin{align}
    \abs{I_{2,3}} & \lesssim \int_{-\ell}^\ell w_\e'(z_0 + \bar{z}) \int_{B_\ell^{d-1}} \frac{1}{\abs{\bar{y}}^{d+2s-3}} \, d \bar{y} \lesssim \ell^{2-2s} \int_{-\ell}^{\ell} w_\e'(z_0 + \bar{z}) \, d \bar{z} \lesssim \ell^{2-2s}. 
\end{align}

\end{proof}

\begin{lemma} \label{l: expansion of I_3}
Let $I_3^\nu$ be defined by \eqref{eq: expansion of inner contribution}. It holds that 
\begin{equation}
    \sup_{x_0 \in \Sigma_{\sfrac{\delta}{10 \Lambda}} } \limsup_{\nu \to 0} \abs{I_3^\nu} \lesssim \e^{2s} \Lambda^{4s-2} + \Lambda^{2s-2}. 
\end{equation}
\end{lemma}

\begin{proof}
For simplicity, we denote by $\ell = \sfrac{\delta}{\Lambda}$. By Taylor's expansion, we have 
\begin{align}
    \lim_{\nu \to 0} I_3^\nu  = \int_0^1 (1-t) \int_{B_{\ell}^{d-1}} \int_{-\ell}^{\ell} \frac{w_\e''(z_0 + \bar{z} + t f(z_0, \bar{y}) ) f(z_0, \bar{y})^2 }{\rho^{d+2s}} (1- \bar{z} H_{\Sigma}(x_0'))\, d\bar{z}\, d \bar{y} 
\end{align}
Hence, integrating by parts twice, we get that 
\begin{align}
    \lim_{\nu \to 0} I_3^\nu & = \int_0^1 \int_{B_{\ell}^{d-1}} \bigg[  \left[ w'_\e(z_0 + \bar{z} + t f(\bar{y}, z_0)) \frac{1- \bar{z} H_{\Sigma}(x_0')}{\rho^{d+2s}} \right]_{\bar{z} = -\ell}^{\bar{z}= \ell} 
    \\ & \qquad - \left[ w_{\e}(z_0+ \bar{z} + t f(\bar{y}, z_0) ) \frac{d}{d \bar{z}} \left( \frac{1- \bar{z} H_{\Sigma}(x_0')}{\rho^{d+2s}} \right) \right]_{\bar{z}=-\ell}^{\bar{z} = \ell}
    \\ & \qquad + \int_{-\ell}^{\ell} w_\e(z_0 + \bar{z} + t f(\bar{y}, z_0)) \frac{d^2}{d \bar{z}^2} \left( \frac{1- \bar{z} H_{\Sigma}(x_0')}{\rho^{d+2s}} \right)  \, d \bar{z} \bigg] f(\bar{y}, z_0)^2 \, d \bar{y} \, dt
    \\ & = I_{3,1} + I_{3,2} + I_{3,3}. 
\end{align}

\textsc{Estimate of $I_{3,1}$}. To estimate $I_{3,1}$, by \eqref{eq: decay w'}, we have 
\begin{align}
    w'_\e\left(z_0 + \ell + t f(\bar{y}, z_0)\right) & \lesssim  \frac{ \e^{-1}}{1+ \e^{-1-2s} \abs{z_0 + \ell + t f(\bar{y}, z_0)}^{1+2s} } \lesssim \frac{ \e^{2s} }{\abs{z_0 + \ell + t f(\bar{y}, z_0)}^{1+2s}}. 
\end{align}
Moreover, by \eqref{eq: formula f}, it results that $\abs{f(\bar{y}, z_0)} \leq C \ell^2$ for any $\bar{y} \in B_{\ell}^{d-1}$. Here $C>0$ is a purely geometric constant. Recall that $\abs{z_0} \leq \sfrac{\ell}{10}$. Hence, we have that 
\begin{equation}
    \abs{z_0 + \ell + t f(\bar{y}, z_0)} \geq \ell - \frac{\ell}{10} - C \ell^2  \geq \frac{\ell}{2},
\end{equation}
provided that we take $\Lambda \geq \Lambda_0$, where $\Lambda_0$ in \cref{l:inner ball} is chosen large enough, depending on $\Sigma$. Thus, since $f(\bar{y}, z_0)^2 = O (\abs{\bar{y}}^4)$ (see \eqref{eq: formula f}), we estimate 
\begin{align}
    \int_0^1 \int_{B_{\ell}^{d-1}} & \left[ w'_\e(z_0 + \ell + t f(\bar{y}, z_0)) \frac{1- \ell H_{\Sigma}(x_0')}{( \ell^2 + \abs{\bar{y}}^2 )^{\frac{d+2s}{2}}} \right] f(\bar{y}, z_0)^2 \, d \bar{y} \, dt  \lesssim  \e^{2s} \ell^{2-4s}. 
\end{align}
Similarly, we estimate the evaluation in $-\ell$, thus proving that $\abs{I_{3,1}} \lesssim \e^{2s} \ell^{2-4s}$.  

\textsc{Estimate of $I_{3,2}$}. By direct computation, for any $(\bar{y}, \bar{z}) \in B_{\ell}^{d-1} \times \left(-\ell,\ell\right)$ we have 
\begin{equation}
    \frac{d}{d \bar{z}} \left( \frac{1- \bar{z} H_{\Sigma}(x_0')}{\rho^{d+2s}} \right) = -\frac{H_{\Sigma} (x_0') }{\rho^{d+2s}} - (d+2s) \frac{\bar{z} - \bar{z}^2 H_{\Sigma}(x_0')}{\rho^{d+2s+2}}, \label{eq:est first derivative}
\end{equation}
\begin{equation} 
    \abs{\frac{d^2}{d \bar{z}^2} \left( \frac{1- \bar{z} H_{\Sigma}(x_0')}{\rho^{d+2s}} \right) } \lesssim \frac{1}{\rho^{d+2s+2}}.  \label{eq:est second derivative} 
\end{equation}
Since $ \abs{ w_\e} \leq 1$, by \eqref{eq:est first derivative} and $f(\bar{y}, z_0)^2 = O (\abs{\bar{y}}^4)$ (see \eqref{eq: formula f}), we estimate 
\begin{align}
    \abs{I_{3,2}} \lesssim \int_{B_\ell^{d-1}} \frac{\abs{\bar{y}}^4}{ (\ell^2 + \abs{\bar{y}}^2)^{\frac{d+2s}{2}} } \, d \bar{y} \lesssim \ell^{3-2s}
\end{align}

\textsc{Estimate of $I_{3,3}$}. Using \eqref{eq:est second derivative}, \cref{l: reduction of the kernel} and recalling that $f(\bar{y}, z_0)^2 = O (\abs{\bar{y}}^4)$ (see \eqref{eq: formula f}), $s \in \left(\sfrac{1}{2},1\right)$, it is readily checked that 
\begin{align}
    \abs{I_{3,3}} & \lesssim  \int_{B_{\ell}^{d-1}} \int_{-\ell}^{\ell}\frac{\abs{\bar{y}}^4}{\rho^{d+2s+2}} \, d\bar{z} \, d\bar{y} \lesssim \int_{-\ell}^{\ell} \left( \abs{\bar{z}}^{1-2s} + O\left(\ell^{1-2s} \right) \right) \, d\bar{z} \lesssim \ell^{2-2s}.
\end{align}
\end{proof}

\begin{lemma} \label{l: expansion of I_4}
Let $I_4^\nu$ be defined by \eqref{eq: expansion of inner contribution}. It holds that 
\begin{equation}
    \sup_{x_0 \in \Sigma_{\sfrac{\delta}{10 \Lambda}}} \limsup_{\nu \to 0} \abs{I_4^\nu} \lesssim \Lambda^{2s-2}. 
\end{equation}
\end{lemma}

\begin{proof}
For simplicity, we denote by $\ell = \sfrac{\delta}{\Lambda}$. Then, we split 
\begin{align}
    I_4 & \leq \abs{z_0}  \int_{B_\ell^{d-1}} \int_{-\ell}^\ell  \frac{\abs{w_\e(z_0 + \bar{z} + f(\bar{y}, z_0)) - w_\e(z_0+ \bar{z}) }}{\rho^{d+2s-1}}  \, d \bar{z} \, d \bar{y} 
    \\ & \qquad + \abs{z_0}  \int_{B_\ell^{d-1}} \int_{-\ell}^\ell  \frac{\abs{w_\e(z_0 + \bar{z}) - w_\e(z_0) }}{\rho^{d+2s-1}}  \, d \bar{z} \, d \bar{y} 
    \\ &  \qquad + \int_{B_\ell^{d-1}} \int_{-\ell}^\ell  \frac{\abs{w_\e(z_0 + \bar{z} + f(\bar{y}, z_0)) - w_\e(z_0) }}{\rho^{d+2s-2}}  \, d \bar{z} \, d \bar{y}
    \\ & = I_{4,1}+ I_{4,2}+ I_{4,3}. 
\end{align}

\textsc{Estimate of $I_{4,1}$}. Since $f(\bar{y}, z_0) = O(\abs{\bar{y}}^2)$ (see \eqref{eq: formula f}) and $\abs{z_0} \leq \sfrac{\ell}{10}$, by the fundamental theorem of calculus and integrating by parts, we have that 
\begin{align}
    & I_{4,1} \leq \ell \int_0^1 \int_{B_\ell^{d-1}} \int_{-\ell}^\ell \frac{w'_\e(z_0 + \bar{z} + t f(\bar{y}, z_0))}{\rho^{d+2s-3}} \, d \bar{z} \, d \bar{y} \, dt  
    \\ & = \ell \int_0^1 \int_{B_\ell^{d-1}} \left( \left[ \frac{w_\e(z_0+ \bar{z} + t f(\bar{y}, z_0))}{\rho^{d+2s-3}} \right]_{\bar{z}= -\ell}^{\bar{z}= \ell} + (d+2s-3) \int_{-\ell}^\ell \frac{w_\e(z_0+ \bar{z} + t f(\bar{y}, z_0)) \bar{z}}{\rho^{d+2s-1}} \, d \bar{z} \right) \, d \bar{y} \, dt
    \\ & \lesssim  \ell^{3-2s} + \int_{B_\ell^{d-1}} \int_{-\ell}^\ell \frac{1}{\rho^{d+2s-2}} \, d \bar{z} \, d \bar{y} \lesssim \ell^{2-2s}. 
\end{align}

\textsc{Estimate of $I_{4,2}$}. The estimate of $I_{4,2}$ is similar to that of $I_{2,2}$ in \cref{l: expansion of I_2}. By the fundamental theorem of calculus and recalling that $d+2s-1 > d-1$, we compute 
\begin{align}
    \abs{I_{4,2}} & \leq \abs{z_0} \int_0^1 \int_{-\ell}^\ell \abs{\bar{z}} w_\e'(z_0 + t \bar{z}) \int_{B_\ell^{d-1}} \frac{1}{(\abs{\bar{z}}^2 +\abs{\bar{y}}^2)^{\frac{d+2s-1}{2}}} \, d \bar{y} \, d \bar{z} \, dt 
    \\ & = \abs{z_0} \int_0^1 \int_{-\ell}^\ell \abs{\bar{z}}^{1-2s} w_\e'(z_0 + t \bar{z}) \int_{B_{\sfrac{\ell}{\abs{\bar{z}}}}^{d-1}} \frac{1}{(1 +\abs{\bar{y}}^2)^{\frac{d+2s-1}{2}}} \, d \bar{y} \, d \bar{z} \, dt 
    \\ & \leq \abs{z_0} \int_0^1 \frac{1}{t^{2-2s}} \int_{-\ell t}^{\ell t} \frac{w_\e'(\bar{z} + z_0) }{\abs{\bar{z}}^{2s-1}}\, d \bar{z} \, dt. 
\end{align}
Since $2-2s <1 $, we have 
\begin{equation}
    \abs{I_{4,2}} \lesssim \abs{z_0} \int_{-\ell }^{\ell } \frac{w_\e'(\bar{z} + z_0) }{\abs{\bar{z}}^{2s-1}}\, d \bar{z} \lesssim \ell^{2-2s} 
\end{equation}
as for the term $A$ in the estimate of $I_{2,2}$ in \cref{l: expansion of I_2}.

\textsc{Estimate of $I_{4,3}$ }.  Since $\abs{w_\e} \leq 1$ and $d+2s-2 < d$, we have that 
\begin{align}
    I_{4,3} & \lesssim \int_{B_\ell^{d-1}} \int_{-\ell}^\ell \frac{1}{\rho^{d+2s-2}} \, d \bar{z} \, d \bar{y}  \lesssim \ell^{2-2s}.
\end{align}
\end{proof}

\section{On the finiteness of a constant} \label{s: finiteness of constants}

In this section, we show that the constant $\kappa_\star$ in \eqref{eq:main} is finite. Motivated by \cref{t:fractional laplacian}, we introduce the following notation.  

\begin{definition} \label{d: eta_e,kappa}
Fix $s\in (\sfrac{1}{2},1)$ and let $w$ be the optimal profile of \cref{t:optimal profile}. For $\ell, \e>0$, set $w_\e(z) = w \left( \sfrac{z}{\e}\right)$ and define
\begin{equation} \label{eq: eta_e,kappa}
    \eta_{\e,\ell}(z_0) = \int_{-\ell}^\ell \frac{w_\e'(z_0+z)}{\abs{z}^{2s-1}} \, dz. 
\end{equation} 
\end{definition}

\begin{remark} \label{l:prop of eta} 
Let $\eta_{\e,\ell}$ be given by \eqref{eq: eta_e,kappa}. Since $w$ is strictly increasing and $w'$ is even and globally integrable, then $\eta_{\e, \ell}$ is nonnegative, even and globally bounded. Moreover, by the change of variable $z= \sfrac{\bar{z}}{\e}$ it holds 
\begin{equation}
    \eta_{\e,\ell}(z_0) = \e^{1-2s} \eta_{1,\sfrac{\ell}{\e}}\left( \sfrac{z_0 }{\e} \right). \label{eq: eta scaling}
\end{equation}
\end{remark}

Finally, we show the following result. 

\begin{proposition} \label{l: eta in L^2}
Fix $\ell, \ell'>0$. Let $\eta_{\e,\ell}$ be defined by \eqref{eq: eta_e,kappa}. For any $s\in (\sfrac{3}{4},1)$, it holds 
\begin{equation}
    \lim_{\e\to 0^+} \int_{-\sfrac{\ell'}{\e}}^{\sfrac{\ell'}{\e}} \eta_{1, \sfrac{\ell}{\e}}(z)^2 \, dz = : \mu_w \in (0,\infty). \label{eq: limit constant 1}
\end{equation} 
The constant $\mu_w$ depends only on $w$ and it is independent of $\ell, \ell'$. For $s = \sfrac{3}{4}$, it holds 
\begin{equation}
    \lim_{\e\to 0^+} \frac{1}{\abs{\log(\e)}} \int_{-\sfrac{\ell'}{\e}}^{\sfrac{\ell'}{\e}} \eta_{1, \sfrac{\ell}{\e}}(z)^2 \, dz = 8. \label{eq: limit constant 2}
\end{equation} 
\end{proposition}

\begin{proof}
Fix $\ell, \ell' >0$. Given $\e>0$ we estimate the decay of the inner integral uniformly with respect to $\e$. From now on, unless otherwise specified, we neglect constants depending on $s,W, \ell, \ell'$. By the decay of $w'$ (see \cref{t:optimal profile}) we estimate
\begin{equation}
    \sup_{\e >0} \sup_{\abs{z} \leq 1} \int_{-\sfrac{\ell}{\e}}^{\sfrac{\ell}{\e}} \frac{w'(z+t)}{\abs{t}^{2s-1}} \, dt \lesssim \int_{\abs{t} \leq 2} \frac{1}{\abs{t}^{2s-1}} \, dt + \sup_{\abs{z} \leq 1} \int_{ \abs{t}\geq 2 } \frac{1}{1 + \abs{t+z}^{1+2s} } \, dt < +\infty. \label{eq: new proof 1}
\end{equation}
To estimate the inner integral for $\abs{z} \geq 1$, we fix a constant $\nu >1$ and we define 
\begin{equation}
I_{z, \nu} := \int_{\abs{t} \leq \sfrac{\abs{z}}{\nu}}  \frac{w'(t+z)}{\abs{t}^{2s-1}}\, dt, \quad II_{z, \nu} =: \int_{\sfrac{\abs{z}}{\nu} \leq \abs{t} \leq \nu \abs{z}} \frac{w'(t+z)}{\abs{t}^{2s-1}} \, dt, \quad III_{z,\nu} := \int_{ \abs{t}\geq \nu \abs{z} } \frac{w'(t+z)}{\abs{t}^{2s-1}} \, dt. 
\end{equation}
For the second term we have
\begin{equation}
    \nu^{1-2s} \abs{z}^{1-2s} \int_{\sfrac{\abs{z}}{\nu} \leq \abs{t} \leq \nu \abs{z}} w'(t+z) \, dt \leq II_{z,\nu} \leq \nu^{2s-1} \abs{z}^{1-2s} \int_{\sfrac{\abs{z}}{\nu} \leq \abs{t} \leq \nu \abs{z}} w'(t+z) \, dt.  
\end{equation}
Moreover, if $ z\geq 1 $ by direct computation and \eqref{eq: decay of w}, we estimate 
\begin{align}
    \int_{\sfrac{\abs{z}}{\nu} \leq \abs{t} \leq \nu \abs{z}} w'(t+z) \, dt & = w(z(1+\nu)) -  w\left( z \left( 1+ \frac{1}{\nu}\right)  \right) + w\left( z \left( 1- \frac{1}{\nu}\right)  \right) - w(z(1-\nu)) 
    \\ & = 2 + O(\abs{z}^{-2s}),    
\end{align}
where the reminder depends also on $\nu$. The same estimate can be proved for $z<-1$. Hence, for $\abs{z} \geq 1$ we have that 
\begin{equation}
    2 \nu^{1-2s} \abs{z}^{1-2s} + O(\abs{z}^{1-4s}) \leq II_{z,\nu} \leq 2 \nu^{2s-1} \abs{z}^{1-2s} + O(\abs{z}^{1-4s}). \label{eq: new proof 10}
\end{equation}
By \cref{t:optimal profile}, for $\abs{z} \geq 1$ we have the following estimate for the first term
\begin{equation}
    I_{z,\nu} \lesssim \abs{z}^{-1-2s} \left( 1- \frac{1}{\nu}\right)^{-1-2s} \int_{\abs{t} \leq \sfrac{\abs{z}}{\nu}} \frac{1}{\abs{t}^{2s-1}}\, dt \leq O(\abs{z}^{1-4s}), \label{eq: new proof 11}
\end{equation}
where the reminder depends on $\nu$. Similarly, for the third term, for $\abs{z} \ge 1$ we have
\begin{equation}
    III_{z, \nu} \leq \nu^{1-2s} \abs{z}^{1-2s} \int_{\abs{t} \geq \nu \abs{z} } w'(t+z) \, dt \leq \nu^{1-2s} \abs{z}^{1-2s} \int_{\abs{t'} \geq \abs{z} (\nu-1)} w'(t')\, dt' \leq O(\abs{z}^{1-4s}), \label{eq: new proof 12}
\end{equation}
where the reminder depends on $\nu$. Then, if $\abs{z} \in \left(1, \sfrac{\ell}{\e \nu}\right)$, we have that
\begin{equation}
    II_{z, \nu} \leq \int_{-\sfrac{\ell}{\e}}^{\sfrac{\ell}{\e}} \frac{w'(t+z)}{\abs{t}^{2s-1}} \, dt \leq I_{z,\nu} + II_{z,\nu} +III_{z,\nu}
\end{equation}
and by \eqref{eq: new proof 10}, \eqref{eq: new proof 11} and~\eqref{eq: new proof 12} we estimate
\begin{equation} \label{eq: new proof 2}
    2\nu^{1-2s} \abs{z}^{1-2s} + O(\abs{z}^{1-4s}) \leq \int_{-\sfrac{\ell}{\e}}^{\sfrac{\ell}{\e}} \frac{w'(z+t)}{\abs{t}^{2s-1}} \, dt \leq 2 \nu^{2s-1} \abs{z}^{1-2s} + O(\abs{z}^{1-4s}). 
\end{equation}
Here the reminders depend on $\nu$, but they are independent of $\e$. With the same technique, for $\abs{z} \in \left( \sfrac{\ell}{\e \nu} , \sfrac{\ell \nu}{\e} \right)$ we have 
\begin{equation}
\int_{-\sfrac{\ell}{\e}}^{\sfrac{\ell}{\e}} \frac{w'(z+t)}{\abs{t}^{2s-1}} \, dt \leq I_{z,\nu} + II_{z,\nu} \leq O(\abs{z}^{1-2s}).  \label{eq: new proof 3}
\end{equation} 
Similarly, for $\abs{z} \geq \sfrac{\ell \nu}{\e }$ we have 
\begin{equation}
\int_{-\sfrac{\ell}{\e}}^{\sfrac{\ell}{\e}} \frac{w'(z+t)}{\abs{t}^{2s-1}} \, dt \leq I_{z,\nu} \leq O(\abs{z}^{1-4s}). \label{eq: new proof 4}
\end{equation} 
Here the reminders depend on $\nu$ but not on $\e$. In conclusion, if $s \in \left( \sfrac{3}{4}, 1\right)$ it is trivial to see that 
\begin{equation}
    \mu_w = \int_{\R} \left( \int_{\R} \frac{w'(t+z)}{\abs{t}^{2s-1}} \, dt \right)^2 \, d z \in (0,+\infty)
\end{equation}
and it is independent of $\ell, \ell'$. To show that $\mu_w$ is finite, using \eqref{eq: new proof 1}, \eqref{eq: new proof 2}, \eqref{eq: new proof 3} and~\eqref{eq: new proof 4} with $\nu =2$, we have
\begin{align}
    \sup_{\e > 0} \int_{-\sfrac{\ell'}{\e}}^{\sfrac{\ell'}{\e}} \left( \int_{-\sfrac{\ell}{\e}}^{\sfrac{\ell}{\e} } \frac{w'(t+z)}{\abs{t}^{2s-1}} \, dt \right)^2 \, dz & \lesssim 1+ \int_{ \abs{z} >1 } \abs{z}^{2-4s}\, dz  < +\infty. 
\end{align}
To conclude, we consider the case $s = \sfrac{3}{4}$. Then, using \eqref{eq: new proof 1}, \eqref{eq: new proof 2}, \eqref{eq: new proof 3} and~\eqref{eq: new proof 4} with $\nu >1$ (to be chosen later), we have 
\begin{align}
    & \limsup_{\e \to 0^+} \frac{1}{\abs{\log(\e)}}\int_{-\sfrac{\ell'}{\e}}^{\sfrac{\ell'}{\e}}  \left( \int_{-\sfrac{\ell}{\e}}^{\sfrac{\ell}{\e} } \frac{w'(t+z)}{\abs{t}^{\frac{1}{2}}} \, dt \right)^2 \, dz \leq \limsup_{\e \to 0^+} \frac{1}{\abs{\log(\e)}}\int_{\R}  \left( \int_{-\sfrac{\ell}{\e}}^{\sfrac{\ell}{\e} } \frac{w'(t+z)}{\abs{t}^{\frac{1}{2}}} \, dt \right)^2 \, dz \
    \\ \hspace{0.5 cm} & \leq \limsup_{\e \to 0^+} \frac{1}{\abs{\log(\e)}} \left( \int_{\abs{z}\leq 1} + \int_{1 \leq \abs{z} \leq \sfrac{\ell}{\e \nu}} + \int_{\sfrac{\ell}{\e \nu} \leq \abs{z} \leq \sfrac{\ell \nu}{\e}}  + \int_{ \abs{z} \geq \sfrac{\ell \nu}{\e}  } \right)  \left( \int_{-\sfrac{\ell}{\e}}^{\sfrac{\ell}{\e} } \frac{w'(t+z)}{\abs{t}^{\frac{1}{2}}} \, dt \right)^2 \, dz 
    \\ & \leq \limsup_{\e \to 0^+} \frac{1}{\abs{\log(\e)}} \bigg( O(1) + \int_{1 \leq \abs{z} \leq \sfrac{\ell}{\e \nu} } \left( \frac{4\nu}{\abs{z}} + O(\abs{z}^{-4}) \right) \, dz + \int_{ \sfrac{\ell}{\e \nu} \leq \abs{z} \leq \sfrac{\ell \nu}{ \e}  } O(\abs{z}^{-1}) \, dz 
    \\ & \hspace{1 cm} + \int_{ \abs{z} \geq \sfrac{\ell \nu}{\e}  } O(\abs{z}^{-4}) \, dz \bigg) = 8 \nu. \label{eq: new proof 5}
\end{align}
Using the lower bound in \eqref{eq: new proof 2} and setting $\bar{\ell} := \min\{\ell, \ell'\}$, we estimate 
\begin{align}
    \liminf_{\e \to 0^+} \frac{1}{\abs{\log(\e)}} & \int_{-\sfrac{\ell'}{\e}}^{\sfrac{\ell'}{\e}}  \left( \int_{-\sfrac{\ell}{\e}}^{\sfrac{\ell}{\e} } \frac{w'(t+z)}{\abs{t}^{\frac{1}{2}}} \, dt \right)^2 \, dz \geq \liminf_{\e \to 0^+} \frac{1}{\abs{\log(\e)}}\int_{ 1 \leq \abs{z} \leq \sfrac{\bar{\ell}}{\e \nu} }  \left( \int_{-\sfrac{\ell}{\e}}^{\sfrac{\ell}{\e} } \frac{w'(t+z)}{\abs{t}^{\frac{1}{2}}} \, dt \right)^2 \, dz 
    \\ & \geq \liminf_{\e \to 0^+} \frac{1}{\abs{\log(\e)}} \int_{1 \leq \abs{z} \leq \sfrac{\bar{\ell}}{\e \nu} } \left( \frac{4 \nu^{-1}}{\abs{z}} + O(\abs{z}^{-4}) \right) \, dz  \geq 8 \nu^{-1}. \label{eq: new proof 6}
\end{align}
Combining \eqref{eq: new proof 5} and \eqref{eq: new proof 6} and letting $\nu \to 1^+$, we prove \eqref{eq: limit constant 2}.
\end{proof}

\section{Proof of the main result} \label{s:proof of main theorem}

This section is entirely devoted to the proof of \cref{thm:main}. Our analysis is based on the following heuristic argument which allows to guess the right scaling of the energy for different values of the parameter $s \in \left(\sfrac{1}{2},1\right)$. Given a smooth set $E$, let $u_\e(x) = w_\e(\beta_\Sigma(x))$, where $w_\e$ is the scaled one-dimensional optimal profile and $\beta_\Sigma$ is the modified signed distance function (see \cref{d:regular distance}). We split the energy in the contribution around the interface and the contribution far from the interface. By \cref{t:fractional laplacian}, we exploit the cancellations in the energy around the interface due to the structure of the recovery sequence and we use the decay property of the optimal profile to estimate separately the two terms in the integral away from the interface. More precisely, it is not difficult to see that 
\begin{align}
    \int_{\Omega} \left( \e^{2s-1} (-\Delta)^s u_\e + \frac{W'(u_\e)}{\e} \right)^2 \, dx \simeq &
    \, \e^{4s-2} \mathcal{W}(\Sigma, \Omega) \int_{-\delta}^\delta \left( \int_{-\delta}^\delta \frac{w_\e'(z+t)}{\abs{t}^{2s-1}} \, dt \right)^2 \, dz 
    \\
    & + \e^{4s-2} \left( \norm{R_\e}_{L^\infty(\Sigma_\delta)}^2 + \norm{(-\Delta)^s u_\e}^2_{L^\infty(\Omega\setminus\Sigma_\delta)}\right) \\
    & + \e^{-2} \norm{W'(u_\e)}_{L^\infty(\Omega \setminus \Sigma_\delta)}^2 
    \\
    \simeq & \, \e \mathcal{W}(\Sigma, \Omega) \int_{-\sfrac{\delta}{\e}}^{\sfrac{\delta}{\e}} \left( \int_{-\sfrac{\delta}{\e}}^{\sfrac{\delta}{\e}} \frac{w'(z+t)}{\abs{t}^{2s-1}} \, dt \right)^2 \, dz  +  O(\e^{4s-2}). 
\end{align} 
Moreover, we have
    \begin{equation}
    \int_{-\sfrac{\delta}{\e}}^{\sfrac{\delta}{\e}} \left( \int_{-\sfrac{\delta}{\e}}^{\sfrac{\delta}{\e}} \frac{w'(z+t)}{\abs{t}^{2s-1}} \, dt \right)^2 \, dz \simeq 
    \begin{cases} 
    1 &\text{ if } s \in \left( \sfrac{3}{4},1 \right), \\[0.5ex]
    \abs{\log(\e)} & \text{ if } s = \sfrac{3}{4}, \\[0.5ex]
    \e^{4s-3} & \text{ if } s \in (\sfrac{1}{2},\sfrac{3}{4}),
    \end{cases}
\end{equation}
where we proved the first two estimates in~\cref{l: eta in L^2}, while the last one can be obtained with an simple modification of the argument. Therefore, we conclude that
\begin{equation}
    \int_{\Omega} \left( \e^{2s-1} (-\Delta)^s u_\e + \frac{W'(u_\e)}{\e} \right)^2 \, dx \simeq 
    \begin{cases}
    \e \mathcal{W}(\Sigma, \Omega) + O(\e^{4s-2}) & \text{ if } s \in \left( \sfrac{3}{4}, 1 \right), 
    \\[1ex]
    \e \abs{\log \e} \mathcal{W}(\Sigma, \Omega) + O(\e^{4s-2}) & \text{ if } s = \sfrac{3}{4},
    \\[1ex]
    \e^{4s-2} \mathcal{W}(\Sigma, \Omega) + O(\e^{4s-2}) & \text{ if } s \in \left( \sfrac{1}{2}, \sfrac{3}{4} \right). 
    \end{cases}
\end{equation}
This argument shows different behaviours according to the value of the parameter $s$. Indeed, if $s > \sfrac{3}{4}$ then $\e \mathcal{W}(\Sigma, \Omega)$ is the leading order term and, after dividing the energy by $\e$, we see the Willmore functional in the limit. The situation is similar for $s = \sfrac{3}{4}$, provided that we divide the energy by $\e \abs{\log(\e)}$. Roughly speaking, if $s \in \left[ \sfrac{3}{4}, 1\right)$, it turns out that the energy around the interface is much larger than the energy away from the interface and, after introducing an appropriate scaling, the limit has a purely local behaviour. In the proof of \cref{thm:main}, we make this argument rigorous. This heuristic breaks down when $s \in \left( \sfrac{1}{2}, \sfrac{3}{4}\right)$. Our analysis suggests that all terms have the same order of magnitude and, after dividing the energy by $\e^{4s-2}$, a local/nonlocal behaviour might persist in the limit. To conclude, it seems that a finer analysis of both the error term in the expansion of the fractional Laplacian around the interface and the energy away from the interface would be needed to understand the limiting behaviour of our energy in the case $s \in \left( \sfrac{1}{2}, \sfrac{3}{4}\right)$. From now on, we focus on the case $s \in \left[ \sfrac{3}{4},1\right)$.  

\begin{proof} [Proof of \cref{thm:main}]
We divide the proof in some steps. We recall that $c_\star$ is the constant in \eqref{eq:SV} and we define
\begin{equation} \label{eq: the exact constant}
        \kappa_\star
        := \begin{cases}
        \displaystyle \frac{\gamma_{1,s}^2}{4(2s-1)^2} \int_{\R} \left( \int_{\R} \frac{w'(t+z)}{\abs{t}^{2s-1}} \, dt \right)^2 \, dz & s \in \left(\sfrac{3}{4}, 1\right), 
        \\[0.5ex]
        8 \gamma_{1, \sfrac{3}{4}}^2 & s = \sfrac{3}{4}.
    \end{cases}
\end{equation} 

\textsc{\underline{Step 1}}. Let $s \in \left( \sfrac{3}{4}, 1 \right)$ and let $E \subset \R^d$ be a bounded open set with $\partial E \in C^3$. Define 
\begin{equation}
    u_\e(x) = w_\e(\beta_\Sigma(x)), \label{eq:formula recovery sequence}
\end{equation}
where $w_\e(z) = w \left(\sfrac{z}{\e} \right)$, $w$ is the one-dimensional optimal profile and $\beta_\Sigma$ is the modified distance function according to \cref{d:regular distance}. We claim that for any bounded open set $\Omega \subset \R^d$ we have 
\begin{equation}
    \limsup_{\e \to 0^+} \left( \mathcal{F}_{s, \e} + \mathcal{G}_{s,\e} \right) (u_\e, \Omega) \leq c_\star \mathcal{H}^{d-1}(\partial E \cap \overline{\Omega}) + \kappa_{\star} \int_{\partial E \cap \overline{\Omega}} H_{\partial E}(y)^2 \, d \mathcal{H}^{d-1}(y). \label{eq: limsup of recovery sequence}
\end{equation}
The same conclusions holds for $s = \sfrac{3}{4}$ adding an extra $\abs{\log(\e)}^{-1}$ in front of $\mathcal{G}_{s,\e}$. 

For $s \in \left[ \sfrac{3}{4}, 1\right)$, by a simple modification of the proof of \cite{SV12}*{Proposition 4.7} due to the fact that $\beta_\Sigma$ is the proper distance function in a tubular neighbourhood of $\Sigma = \partial E$, it is readily checked that 
\begin{equation} \label{eq: limsup SV}
    \limsup_{\e \to 0^+} \mathcal{F}_{\e,s}(u_\e, \Omega) \leq c_{\star} \mathcal{H}^{d-1}( \Sigma \cap \overline{\Omega}).
\end{equation} 
We point out that no regularity on $\partial \Omega$ is needed. Therefore, if $s \in \left( \sfrac{3}{4}, 1 \right)$ it remains to check that 
\begin{equation} \label{eq: limsup of G_s,e}
    \limsup_{\e \to 0^+} \mathcal{G}_{s,\e}(u_\e, \Omega) \leq \kappa_\star \int_{\Sigma \cap \overline{\Omega}} H_{\Sigma}(y)^2 \, d \mathcal{H}^{d-1}(y).
\end{equation}
If $s = \sfrac{3}{4}$, \eqref{eq: limsup of G_s,e} holds adding an extra $\abs{\log(\e)}^{-1}$ in front of $\mathcal{G}_{s,\e}$. 

\textsc{The case $s \in \left( \sfrac{3}{4}, 1 \right)$}. Given $\delta>0$ as in \cref{d:regular distance} and $\Lambda_0 \geq 1 $ as in \cref{t:fractional laplacian}, for $\Lambda \geq \Lambda_0$ we split
\begin{align}
    \mathcal{G}_{s,\e}(u_\e, \Omega) & =  \left( \frac{1}{\e} \int_{\Omega \cap \Sigma_{\sfrac{\delta}{10 \Lambda}} } + \frac{1}{\e} \int_{\Omega \setminus \Sigma_{\sfrac{\delta}{10 \Lambda}}} \right) \left( \e^{2s-1} (-\Delta)^s u_\e(x) + \frac{W'(u_\e(x))}{\e}  \right)^2 \, dx = I_{\e,\Lambda} + II_{\e,\Lambda}. 
\end{align}
We claim that 
\begin{equation}
\lim_{\e \to 0^+} II_{\e,\Lambda}  = 0 \qquad \forall \Lambda \geq \Lambda_0. \label{claim 2}    
\end{equation}
To deal with the energy away from the interface, we neglect constants independent of $\e$. Thus, for $\Lambda \geq \Lambda_0$, we have  
\begin{align}
    II_{\e,\Lambda} & \lesssim \e^{4s-3} \int_{\Omega \setminus \Sigma_{\sfrac{\delta}{10 \Lambda}}} \left( (-\Delta)^s u_\e (x)  \right)^2 \, dx + \e^{-3} \int_{\Omega \setminus \Sigma_{\sfrac{\delta}{10 \Lambda}}} \left(  W'(u_\e(x)) \right)^2 \, dx =  II_{1, \e,\Lambda} + II_{2, \e, \Lambda}. 
\end{align} 
By \cref{l: L^infty bound away from the boundary}, we conclude that $II_{1, \e, \Lambda} \to 0$ as $\e \to 0$ for $s>\sfrac{3}{4}$. Regarding $I_{2, \e, \Lambda}$, since $W'$ is Lipschitz and odd, using \eqref{eq: decay of w}, the monotonicity of $w$ and the fact that $\abs{\beta_{\Sigma}(x)} > \sfrac{\delta}{10 \Lambda}$ in the domain of integration (see \cref{d:regular distance}), we have that 
\begin{align}
   II_{2, \e,\Lambda} & = \e^{-3} \int_{ E \setminus \Sigma_{\sfrac{\delta}{10 \Lambda}}} (W'(u_\e(x)) - W'(1))^2\, dx + \e^{-3} \int_{(\Omega \setminus E) \setminus \Sigma_{\sfrac{\delta}{10 \Lambda}} } (W'(u_\e(x)) - W'(-1) )^2 \, dx   
   \\ & \lesssim \e^{-3} \int_{E \setminus \Sigma_{\sfrac{\delta}{10 \Lambda}}} \abs{ u_\e(x) - 1}^2\, dx + \e^{-3} \int_{(\Omega \setminus E) \setminus \Sigma_{\sfrac{\delta}{10 \Lambda}} } \abs{u_\e(x) + 1 }^2 \, dx 
   \\ & \lesssim \e^{-3} \int_{E \setminus \Sigma_{\sfrac{\delta}{10 \Lambda}} } \abs{ 1- w\left( \frac{\delta}{10 \Lambda \e}\right) }^2\, dx + \e^{-3} \int_{(\Omega \setminus E) \setminus \Sigma_{\sfrac{\delta}{10 \Lambda}} } \abs{ 1+ w\left(-\frac{\delta}{10 \Lambda \e}\right) }^2 \, dx \lesssim \e^{4s-3}. 
\end{align}
Hence, we have that $II_{2,\e,\Lambda} \to 0$ as $\e \to 0^+$, since $s >\sfrac{3}{4}$. 

Letting be $\kappa_\star$ as in~\eqref{eq: the exact constant}, it remains to check that 
\begin{equation}
\inf_{\Lambda \geq \Lambda_0} \limsup_{\e \to 0^+} I_{\e,\Lambda} \leq \kappa_\star \int_{\Sigma \cap \overline{\Omega}} H_{\Sigma}(x) ^2 \, d \mathcal{H}^{d-1}(x). \label{claim 1}
\end{equation}
Then, \eqref{eq: limsup of G_s,e} follows by \eqref{claim 2} and \eqref{claim 1}. The proof of \eqref{claim 1} is based on the expansion of the fractional Laplacian \eqref{eq:expansion}. Fix $\Lambda \geq \Lambda_0$. Since $\beta_\Sigma$ is the proper signed distance in $\Sigma_{\sfrac{\delta}{10 \Lambda}}$ (see \cref{d:regular distance}), we have  
\begin{equation}
    I_{\e,\Lambda} = \frac{1}{\e} \int_{\Sigma_{\sfrac{\delta}{10 \Lambda}} \cap \Omega} \left( \frac{W'(u_\e(x))}{\e}  + \e^{2s-1} \left( (-\partial_{zz})^s w_\e(z) + \frac{\gamma_{1,s}}{4s-2} H_{\Sigma}(x') \eta_{\e,\sfrac{\delta}{10 \Lambda}} (z) + \mathcal{R}_{\e, \Lambda}(x) \right)   \right)^2  \, dx, \label{eq:limsup_1}
\end{equation}
where we set $x'= \pi_{\Sigma}(x)$, $z = \dist_{\Sigma}(x)$, $\eta_{\e, \sfrac{\delta}{10 \Lambda}}$ is given by \eqref{eq: eta_e,kappa} and $\mathcal{R}_{\e, \Lambda}$ is uniformly bounded with respect to $\e$ in $L^\infty(\Sigma_{\sfrac{\delta}{10 \Lambda}})$ (see \cref{t:fractional laplacian}). By the scaling properties of the fractional Laplacian and the fact that $w$ solves \eqref{eq: fractional AC}, it results that
$$(-\partial_{zz})^s w_\e (z) = \e^{-2s} (-\partial_{zz})^s w\left( \sfrac{z}{\e}\right) = -\e^{-2s} W'\left(w\left( \sfrac{z}{\e}\right) \right) = - \e^{-2s} W'(u_\e(x)). $$
Therefore, \eqref{eq:limsup_1} reads as 
\begin{equation}
    I_{\e, \Lambda} = \e^{4s-3} \int_{\Sigma_{\sfrac{\delta}{10 \Lambda}} \cap \Omega}  \left( \frac{\gamma_{1,s}}{4s-2} H_{\Sigma}(x') \eta_{\e,\sfrac{\delta}{10 \Lambda}}(z) + \mathcal{R}_{\e, \Lambda}(x)  \right)^2  \, dx. \label{eq:limsup_2}
\end{equation}
Furthermore, using \eqref{eq: eta scaling} and changing variables, we have that 
\begin{align}
    \e^{4s-3} \int_{-\sfrac{\delta}{10 \Lambda}}^{\sfrac{\delta}{10 \Lambda}} \eta_{\e,\sfrac{\delta}{10 \Lambda}} (z)^2 \, dz & = \e^{-1} \int_{-\sfrac{\delta}{10 \Lambda}}^{\sfrac{\delta}{10 \Lambda}} \eta_{1,\sfrac{\delta}{10 \e \Lambda}}\left(\frac{z}{\e}\right)^2 \, dz = \int_{-\sfrac{\delta}{10 \e \Lambda}}^{\sfrac{\delta}{10 \e \Lambda}} \eta_{1,\sfrac{\delta}{10 \e \Lambda}}(z)^2 \, dz.
\end{align} 
Thus, using \cref{l: eta in L^2} we have that 
\begin{equation}
    \lim_{\e \to 0^+} \e^{4s-3} \int_{-\sfrac{\delta}{10 \Lambda}}^{\sfrac{\delta}{10 \Lambda}} \eta_{\e, \sfrac{\delta}{10 \Lambda}}(z)^2 \, dz = \int_{\R} \left( \int_{\R} \frac{w'(t+z)}{\abs{t}^{2s-1}} \, dt \right)^2 \, dz = \mu_w. \label{eq: limit computation constant}
\end{equation}
Thus, since $s >\sfrac{3}{4}$ and $\mathcal{R}_{\e, \Lambda}$ is uniformly bounded with respect to $\e$ in $L^{\infty}(\Sigma_{\sfrac{\delta}{10 \Lambda}})$, we have that 
\begin{equation}
    \limsup_{\e \to 0^+} I_{\e,\Lambda} =
    \frac{\gamma_{1,s}^2}{4(2s-1)^2}
    \limsup_{\e \to 0^+} \e^{4s-3} 
    \int_{\Omega \cap \Sigma_{\sfrac{\delta}{10 \Lambda}}} \abs{H_{\Sigma}(x') \eta_{\e,\sfrac{\delta}{10 \Lambda}}(z)}^2 \, dx. \label{eq: limsup_3} 
\end{equation}
We remark that 
$$ \Sigma_{\sfrac{\delta}{10 \Lambda}} \cap \Omega \subset \left\{ x'+z N_{\Sigma}(x') \colon x' \in \pi_{\Sigma}\left( \Sigma_{\sfrac{\delta}{10 \Lambda}} \cap \Omega \right), \ z \in \left(-\sfrac{\delta}{10 \Lambda}, \sfrac{\delta}{10 \Lambda} \right) \right\}. $$
Thus, by \eqref{eq: limit computation constant}, \eqref{eq: limsup_3} and the Coarea formula, we write 
\begin{align}
    \limsup_{\e \to 0^+} I_{\e, \Lambda} & = \limsup_{\e \to 0^+} \e^{4s-3} \int_{-\sfrac{\delta}{10 \Lambda}}^{\sfrac{\delta}{10 \Lambda}} \abs{\eta_{\e,\sfrac{\delta}{10 \Lambda}} (z)}^2 \left( \int_{ \pi_{\Sigma} (\Sigma_{\sfrac{\delta}{10 \Lambda}} \cap \Omega) } \abs{H_{\Sigma}(x') }^2 \, d \mathcal{H}^{d-1}(x') \right) \, dz 
    \\ & = \mu_w \int_{ \pi_{\Sigma} (\Sigma_{\sfrac{\delta}{10 \Lambda}} \cap \Omega ) } \abs{H_{\Sigma}(x') }^2 \, d \mathcal{H}^{d-1}(x'), \label{eq: limsup 100}
\end{align}
Hence, \eqref{claim 1} follows by the computation above and the fact $\bigcap_{\Lambda\geq \Lambda_0} \pi_{\Sigma} (\Sigma_{\sfrac{\delta}{10 \Lambda}} \cap \Omega ) \subset \Sigma \cap \bar{\Omega}$.

\textsc{The case $s = \sfrac{3}{4}$}. Here, we have $4s-3 = 0$, and hence $\e^{4s-3} =1 $, for any $\e > 0$. However, there is an extra factor $\abs{\log(\e)}^{-1}$ in the definition of $\mathcal{G}_{s, \e}$ and \cref{l: eta in L^2} is modified accordingly. Therefore, the proof is similar to the previous case and we leave the straightforward modifications to the reader.

\textsc{\underline{Step 2}}. Let $\Omega, E \subset \R^d$ be bounded open sets with $\partial \Omega \in C^1$ and $\partial E \in C^2$. By \cref{l:approximation}, we find a sequence of smooth sets $\{E_j\}_{j \in \N}$ satisfying \eqref{approx 1}, \eqref{approx 2}, \eqref{approx 3}, \eqref{approx 4}, \eqref{approx 5}. By the argument shown in the previous step, for any $j \in \N$ we find a sequence $u_\e^j \to \chi_{E_j}$ in $L^1_{\rm loc}(\R^d)$ such that \eqref{eq: limsup of recovery sequence} holds for $E_j$. By \eqref{approx 2} and \eqref{approx 4}, we have that $\mathcal{H}^{d-1}(\partial E_j \cap \overline{\Omega}) \to \textrm{Per}(E, \Omega)$. Then, by \eqref{approx 3}, \eqref{approx 4}, \eqref{approx 5}, it is immediate to check that 
\begin{equation}
    \lim_{j \to \infty} \int_{\partial E_j \cap \overline{\Omega}} H_{\partial E_j}(y)^2 \, d \mathcal{H}^{d-1}(y) = \mathcal{W}(\Sigma, \Omega).  \label{eq: USC willmore}
\end{equation}
Then, by a diagonal argument we build a sequence $u_\e \to \chi_E$ in $L^1_{\rm loc}(\R^d)$ such that \eqref{eq: limsup of recovery sequence} holds. 
\end{proof}

\begin{remark}
We point out that by a simple modification of the proof of \cref{thm:main} it is possible to check that
\begin{equation}
    \lim_{\e \to 0^+} \left( \mathcal{F}_{s, \e} + \mathcal{G}_{s,\e} \right) (u_\e, \Omega) = c_\star \mathcal{H}^{d-1}(\partial E \cap \Omega) + \kappa_{\star} \int_{\partial E \cap \Omega} H_{\partial E}(y)^2 \, d \mathcal{H}^{d-1}(y), 
\end{equation}
where $s \in \left(\sfrac{3}{4},1 \right)$, $E$ is a bounded open set of class $C^3$, $u_\e$ is defined by \eqref{eq:formula recovery sequence} and $\Omega$ is any bounded open set whose boundary intersects transversally $\partial E$, that is $\mathcal{H}^{d-1}(\partial E \cap \partial \Omega) = 0$. Here, no regularity on $\partial \Omega$ is needed. The same conclusions holds for $s = \sfrac{3}{4}$ adding an extra $\abs{\log(\e)}^{-1}$ in front of $\mathcal{G}_{s,\e}$. We leave the details to the interested reader. 
\end{remark}

\begin{remark} \label{r: limsup inequality with closure}
Once \cref{thm:main} is established, we can prove \eqref{eq: limsup of recovery sequence} for bounded open sets $E$ of class~$C^2$. Indeed, given any bounded open set $\Omega$ (without further assumptions), we find a sequence of smooth bounded open sets $\Omega_j$ such that $\bigcap_j \Omega_j = \Omega$. Then, using \cref{thm:main} for $E$ in $\Omega_j$ and recalling that 
\begin{equation}
    \lim_{j \to \infty} c_\star \mathrm{Per}(E, \Omega_j) + \kappa_\star \mathcal{W}(E, \Omega_j) = c_\star \mathcal{H}^{d-1}(\partial E\cap \overline{\Omega}) + \kappa_\star \int_{\partial E \cap \overline{\Omega}} H_{\partial E}(y)^2 \, d \mathcal{H}^{d-1}(y), 
\end{equation} 
we build the sequence $\{u_\e\}$ by a standard diagonal argument. 
\end{remark}

\section{Appendix: The exact constants} \label{s:appendix}

We collect some useful results on the exact value of some constants involved in our computations. We recall that the Euler Gamma function and Beta function are defined respectively by 
\begin{equation}
    \Gamma(x) = \int_0^{+\infty} t^{x-1} e^{-t} \, dt, \qquad \mathrm{B}(x,y) = \int_0^1(1-t)^{x-1} t^{y-1} \, dt \qquad x,y >0.
\end{equation}
We recall that 
\begin{equation}
    \Gamma(x+1) = x \Gamma(x) ,  \qquad     \Gamma(\sfrac{1}{2}) = \sqrt{\pi}, \qquad \mathcal{H}^{d-1} (\mathbb{S}^{d-1}) = \frac{2 \pi^{\frac{d}{2}}}{\Gamma\left(\frac{d}{2}\right)}, \qquad \mathrm{B}(x,y) = \frac{\Gamma(x) \Gamma(y)}{\Gamma(x+y)}.  \label{eq: properties of Gamma 1}
\end{equation}

\begin{lemma}
\label{lem:const-1} 
Given $a >-1, b > a+1$, it holds
\begin{equation}
    2\int_0^\infty \frac{ r^a }{ (r^2+1)^{\frac{b}{2}} }\,dr = \frac{ \Gamma\left(\frac{a+1}{2}\right) \Gamma\left(\frac{b-(a+1)}{2}\right) }{ \Gamma\left(\frac{b}{2}\right) }. \label{eq: integral computation constant 1}
\end{equation}
\end{lemma}

\begin{proof}
The computation is straightforward. After the change of variable $t= (r^2+1)^{-1}$, by standard manipulations, we find that 
\begin{align}
    2 \int_0^\infty \frac{ r^a }{ (r^2+1)^{\frac{b}{2}} } \, dr & = \int_0^1 t^{\frac{b}{2}} \left( \frac{1-t}{t} \right)^{\frac{a-1}{2}} \frac{dt}{t^2} = \mathrm{B}\left( \frac{b-a-1}{2}, \frac{a+1}{2} \right). 
\end{align}
Hence, \eqref{eq: integral computation constant 1} follows by the above computation and \eqref{eq: properties of Gamma 1}. 
\end{proof}

\begin{corollary} \label{l: reduction of the kernel} 
Fix $\delta>0$ and $\alpha, \beta \geq 0$ such that $2s+\beta+1 >\alpha$. There exists an explicit constant $M(d,s,\alpha, \beta) >0$ such that for any $z \in (-\delta, \delta)$ it holds  
\begin{equation} \label{eq: reduction 4}
    \int_{B_\delta^{d-1}} \frac{\abs{y}^\alpha }{(\abs{z}^2 + \abs{y}^2)^{\frac{d+2s + \beta}{2}}} \, d y = \mathcal{M}(d,s,\alpha,\beta) \abs{z}^{\alpha-1-\beta-2s} + O(\delta^{\alpha -1 -\beta-2s }). 
\end{equation}
More precisely, it holds  
\begin{equation} \label{eq: reduction 1}
    \int_{B_\delta^{d-1}} \frac{1}{(\abs{z}^2 + \abs{y}^2)^{\frac{d+2s}{2}}} \, d y = \frac{\gamma_{1,s}}{\gamma_{d,s}} \abs{z}^{-1-2s} + O(\delta^{-1-2s}), 
\end{equation} 
\begin{equation} \label{eq: reduction 2}
    \int_{B_\delta^{d-1}} \frac{\abs{y}^2}{(\abs{z}^2 + \abs{y}^2)^{\frac{d+2s}{2}}} \, d y = \frac{(d-1) \gamma_{1,s}}{(2s-1) \gamma_{d,s} } \abs{z}^{1-2s} + O(\delta^{1-2s}), 
\end{equation} 
\begin{equation} \label{eq: reduction 3}
    \int_{B_\delta^{d-1}} \frac{y_i^2}{(\abs{z}^2 + \abs{y}^2)^{\frac{d+2s}{2}}} \, d y = \frac{1}{2s-1} \frac{ \gamma_{1,s}}{ \gamma_{d,s} } \abs{z}^{1-2s} + O(\delta^{1-2s}) \qquad i = 1, \dots, d-1.  
\end{equation} 
The reminders are uniform with respect to $z \in (-\delta, \delta)$.  
\end{corollary} 

\begin{proof}
We set  
\begin{equation}
    \mathcal{M}(d,s,\alpha, \beta) := \int_{\R^{d-1}} \frac{\abs{y}^{\alpha}}{(1+\abs{y}^2)^{\frac{d+2s+\beta}{2}}} \, dy
\end{equation}
and we remark that $M(d,s,\alpha, \beta)$ is finite since $d+2s+\beta - \alpha > d-1$. Then, to check \eqref{eq: reduction 4}, we change variables $y= \abs{z} \tilde{y}$ and we compute 
\begin{align} 
\int_{B_\delta^{d-1}} \frac{\abs{y}^{\alpha}}{(\abs{z}^2 + \abs{y}^2)^{\frac{d+2s+\beta}{2}}} \, d y  & = \abs{z}^{\alpha-1-\beta-2s} \int_{B^{d-1}_{\sfrac{\delta}{\abs{z}} }} \frac{\abs{\tilde{y}}^{\alpha}}{ (1+ \abs{\tilde{y}}^2) ^{\frac{d+2s +\beta}{2}}} \, d \tilde{y}
\\ & = \abs{z}^{\alpha-1-\beta-2s} \left( \mathcal{M}(d,s,\alpha,\beta) - \int_{\abs{\tilde{y}} \geq \frac{\delta}{\abs{z}} } \frac{\abs{\tilde{y}}^{\alpha}}{ (1+ \abs{\tilde{y}}^2) ^{\frac{d+2s+\beta}{2}}} \, d \tilde{y} \right)
\\ & = \abs{z}^{\alpha-1-\beta-2s} \left( \mathcal{M}(d,s,\alpha,\beta) - O \left( \left(\frac{\delta}{\abs{z}} \right)^{\alpha-1-\beta-2s} \right) \right ).  
\end{align}
To prove \eqref{eq: reduction 1} and \eqref{eq: reduction 2}, we compute the explicit value of $\mathcal{M}(d,s,\alpha, \beta)$. When $(\alpha, \beta) = (0,0)$, by \cref{lem:const-1}, \eqref{eq: properties of Gamma 1} and \eqref{eq: constant fractional laplacian} we have 
\begin{align}
    \mathcal{M}(d,s,0,0) & = \int_{\R^{d-1}} \frac{1}{(1+ \abs{y}^2)^{\frac{d+2s}{2}}} \, dy = \mathcal{H}^{d-2}(\mathbb{S}^{d-2}) \int_0^{+\infty} \frac{r^{d-2}}{(1+r^2)^{\frac{d+2s}{2}}} \, dr = \pi^{\frac{d-1}{2}} \frac{\Gamma(\frac{2s+1}{2})}{\Gamma(\frac{d+2s}{2})} = \frac{\gamma_{1,s}}{\gamma_{d,s}}, 
\end{align}
Similarly, it is immediate to check that 
\begin{equation}
    \mathcal{M}(d,s, 2,0) = \frac{d-1}{2s-1} \cdot \frac{\gamma_{1,s}}{\gamma_{d,s}}. 
\end{equation}
Lastly, we observe that \eqref{eq: reduction 3} follows by \eqref{eq: reduction 2}, since by symmetry we have that 
\begin{equation}
    \int_{B_\delta^{d-1}} \frac{y_i^2}{(\abs{z}^2 + \abs{y}^2)^{\frac{d+2s}{2}}} \, d y = \frac{1}{d-1} \int_{B_\delta^{d-1}} \frac{\abs{y}^2}{(\abs{z}^2 + \abs{y}^2)^{\frac{d+2s}{2}}} \, d y \qquad i =1, \dots, d-1. 
\end{equation}
\end{proof}

\textbf{Acknowledgements}. H.~C. has received funding from the Swiss National Science Foundation under Grant PZ00P2\_202012/1. M.~F. is a member of the \selectlanguage{italian}{``Gruppo Nazionale per l'Analisi Matematica, la Probabilità e le loro Applicazioni'' (GNAMPA)} of the \selectlanguage{italian}{``Istituto Nazionale di Alta Matematica''} (INdAM). M.~F. is a member of the PRIN Project 2022AKNSE4 {\em Variational and Analytical aspects of Geometric PDEs}. M.~I. acknowledges the support of the SNF grant FLUTURA: Fluids, Turbulence, Advection No. 212573. 

\selectlanguage{english}
\begin{center}
    \textbf{Statements and Declarations}
\end{center}

\noindent \textbf{Conflict of interest.} The authors have no relevant financial or non-financial interests to disclose.

\smallskip

\noindent \textbf{Data availability.} Data sharing is not applicable, since no data were used for this research.

\bibliographystyle{plain} 
\bibliography{biblio}

\end{document}